\RequirePackage{fix-cm}
\documentclass[smallextended,referee,envcountsect]{svjour3}       
\smartqed  

\usepackage{graphicx}
\usepackage{amsmath, ntheorem,dsfont}

\usepackage{pifont}

\usepackage{amssymb}
\usepackage{stackrel}
\usepackage{marvosym, enumerate}
\usepackage{amstext, amscd, latexsym}
\usepackage{amsfonts, ragged2e}
\usepackage{mathrsfs, pgfplots, enumerate, setspace}
\usepackage{subfigure}
\usepackage{cases}

\smartqed
\usepackage{cite}
\usepackage{amstext, graphicx, amscd, latexsym, amssymb}
\usepackage{amsfonts, ragged2e}
\usepackage{pgfplots, setspace}
\usepackage{geometry}
\usepackage{latexsym,bm}
\usepackage{float}
\usepackage{epstopdf,caption}
 
\usepackage{amsfonts,amsmath,amsbsy,amssymb,amscd,mathrsfs}
\usepackage{graphicx}
\usepackage{epstopdf}
\usepackage{algorithmic} 
\usepackage{graphicx,epstopdf}

 \newcommand{\nm}[1]{\left\lVert {#1} \right\rVert}
 
 \newcommand{\dual}[1]{\left\langle {#1} \right\rangle}

\usepackage[figuresright]{rotating}
\usepackage[sort&compress,numbers]{natbib}
\usepackage[ruled,linesnumbered]{algorithm2e} 
\usepackage[colorlinks,linkcolor=blue,anchorcolor=blue,citecolor=blue,urlcolor=blue]{hyperref}

\newtheorem{assumption}{Assumption}[section]

\journalname{}

\begin{document}

\title{Barzilai-Borwein Descent Methods for Multiobjective Optimization Problems with Variable Trade-off Metrics}


\author{Jian Chen \and Liping Tang \and  Xinmin Yang  }

\institute{J. Chen \at National  Center  for  Applied  Mathematics in Chongqing, Chongqing Normal University, Chongqing 401331, China, and School of Mathematical Sciences, University of Electronic Science and Technology of China, Chengdu, Sichuan 611731, China\\
                    \href{mailto:chenjian_math@163.com}{chenjian\_math@163.com}\\
                   L.P. Tang \at National Center for Applied Mathematics in Chongqing, and School of Mathematical Sciences,  Chongqing Normal University, Chongqing 401331, China\\
                   \href{mailto:tanglipings@163.com}{tanglipings@163.com}\\
        \Letter X.M. Yang \at National Center for Applied Mathematics in Chongqing, and School of Mathematical Sciences,  Chongqing Normal University, Chongqing 401331, China\\
        \href{mailto:xmyang@cqnu.edu.cn}{xmyang@cqnu.edu.cn} \\}

\date{Received: date / Accepted: date}

\maketitle

\begin{abstract}
The imbalances and conditioning of the objective functions influence the performance of first-order methods for multiobjective optimization problems (MOPs). The latter is related to the metric selected in the direction-finding subproblems. Unlike single-objective optimization problems, capturing the curvature of all objective functions with a single Hessian matrix is impossible. On the other hand, second-order methods for MOPs use different metrics for objectives in direction-finding subproblems, leading to a high per-iteration cost. To balance per-iteration cost and better curvature exploration, we propose a Barzilai-Borwein descent method with variable metrics (BBDMO\_VM). In the direction-finding subproblems, we employ a variable metric to explore the curvature of all objectives. Subsequently, Barzilai-Borwein's method relative to the variable metric is applied to tune objectives, which mitigates the effect of imbalances. We investigate the convergence behaviour of the BBDMO\_VM, confirming fast linear convergence for well-conditioned problems relative to the variable metric. In particular, we establish linear convergence for problems that involve some linear objectives. These convergence results emphasize the importance of metric selection, motivating us to approximate the trade-off of Hessian matrices to better capture the geometry of the problem. Comparative numerical results confirm the efficiency of the proposed method, even when applied to large-scale and ill-conditioned problems.

\end{abstract}

\section{Introduction}
An unconstrained multiobjective optimization problem can be stated as follows:
\begin{align*}
	\min\limits_{x\in\mathbb{R}^{n}} F(x), \tag{MOP}\label{MOP}
\end{align*}
where $F:\mathbb{R}^{n}\rightarrow\mathbb{R}^{m}$ is a continuously differentiable function. 
In multiobjective optimization, the primary goal is to simultaneously optimize multiple objective functions. In general, finding a single solution that optimizes all objectives is infeasible. Therefore, optimality is defined by {\it Pareto optimality} or {\it efficiency}. A solution is considered Pareto optimal or efficient if no objective can be improved without sacrificing the others. As society and the economy advance, the applications of this type of problem have expanded into various domains, including engineering \cite{MA2004}, economics \cite{FW2014}, management science \cite{E1984}, and machine learning \cite{SK2018}, among others.
\par Solution strategies play a pivotal role in the realm of applications involving multiobjective optimization problems (MOPs). Over the past two decades, multiobjective gradient descent methods have garnered increasing attention within the multiobjective optimization community. These methods generate descent directions by solving subproblems, eliminating the necessity for predefined parameters. Subsequently, line search techniques are employed along the descent direction to ensure sufficient improvement for all objectives. Attouch et al. \cite{AGG2015} highlighted an appealing characteristic of this method in fields such as game theory, economics, social science, and management:{\it it improves each of the objective functions}. As far as we know, the study of multiobjective gradient descent methods can be traced back to the pioneering works by Mukai \cite{M1980}. and Fliege and Svaiter \cite{FS2000}. The later clarified that the multiobjective steepest descent direction reduces to the steepest descent direction when dealing with a single objective. This observation inspired researchers to extend ordinary numerical algorithms for solving MOPs (see, e.g., \cite{AP2021,BI2005,CL2016,FD2009,FV2016,GI2004,LP2018,MP2019,P2014,QG2011} and references therein). 
\subsection{First-order methods}
\par Fliege and Svaiter \cite{FS2000} introduced the steepest descent method for MOPs (SDMO). The steepest descent direction is the optimal solution of the following subproblem:
\begin{align*}
	&\min\limits_{d\in\mathbb{R}^{n}} \max\limits_{i=1,2,...,m}\ \dual{\nabla F_{i}(x),d}+\frac{1}{2}\|d\|^{2}.
\end{align*}
This subproblem can be reformulated as a quadratic problem and efficiently solved through its dual \cite{SK2018}. Subsequently, Gra$\rm\tilde{n}$a Drummond and Iusem extended this method to constrained MOPs, proposing the projected gradient method for MOPs. For multiobjective composite optimization problems, Tanabe et al. \cite{TFY2019} extended the proximal gradient method to MOPs. Like most first-order methods for single-objective optimization problems (SOPs), these MOP counterparts enjoy cheap per-step computation cost but suffer slow convergence, especially for ill-conditioned problems. In response to this challenge, some classic methods were extended to MOPs, including Barzilai-Borwein's method \cite{MP2016}, nonlinear conjugate gradient method \cite{LP2018}, and Nesterov's accelerated method \cite{TFY2022,SP2022,SP2023}. In addition to issues stemming from ill-conditioning, another inherent challenge arises from imbalances among objective functions. Chen et al. \cite{CTY2023} highlighted that even when all objective functions are not ill-conditioned, imbalances among them can lead to slow convergence of first-order methods for MOPs. To address this issue, Chen et al. \cite{CTY2023} applied Barzilai-Borwein's method to alleviate the impact of imbalances. They demonstrated that the Barzilai-Borwein proximal gradient method \cite{CTY2023b} converges at a rate of $\sqrt{1-\min\limits_{i=1,2,...,m}\left\{\frac{\mu_{i}}{L_{i}}\right\}}$, where $\mu_{i}$ and $L_{i}$ are the constants of strong convexity and smoothness of $f_{i}$, respectively. It is worth noting that the performance of this type of method also depends on the conditioning of problems.  
\subsection{Second-order methods}
 Fliege et al. \cite{FD2009} proposed Newton's method for MOPs (NMO). The Newton direction is the optimal solution of 
\begin{align*}
	&\min\limits_{d\in\mathbb{R}^{n}} \max\limits_{i=1,2,...,m}\ \dual{\nabla F_{i}(x),d}+\frac{1}{2}\dual{d,\nabla^{2}F_{i}(x)d}.
\end{align*}
It has been proven that NMO possesses desirable properties \cite{FD2009}, including local superlinear and quadratic convergence under standard assumptions. Furthermore, quasi-Newton methods have garnered considerable attention \cite{QG2011,P2014,LM2023,PS2022,PS2023} and demonstrate local superlinear convergence. While these methods for MOPs are superior in capturing the local geometry of objective functions and offering rapid convergence, the per-step cost is computationally expensive. In contrast to their Single-Objective Problem (SOP) counterparts, the high per-step computation cost arises not only from the computation of Hessian matrices and their inverses but also from the costly subproblems.\footnote{The subproblems of second-order methods for MOPs can only be solved by reformulating into quadratic constrained problems, see \cite{CLY2023}, or the prime minimax problems. Solving these problems is much more time-consuming than quadratic dual problems of first-order methods for MOPs.}.
\par In summary, the slow convergence observed in first-order methods for MOPs can be primarily attributed to the ill-conditioning and imbalances among objective functions. Meanwhile, second-order methods for MOPs face the dual challenges of increased computation costs, particularly concerning (inverse) Hessian matrices and expensive subproblems. This leads us to a natural and compelling question: Can we devise an algorithm that maintains an affordable per-step computation cost, achieves rapid convergence, and is not sensitive to conditioning?
\par Ansary and Panda \cite{AP2015} utilized a single quasi-Newton approximation to approximate all Hessian matrices. Subsequently, this idea was adopted by Chen et al. \cite{CLY2023} and Lapucci and Mansueto \cite{LM2023}. While this approximation better captures the problem's geometry and allows for efficient subproblem solving, Chen et al. \cite{CLY2023} identified a limitation: the monotone line search cannot accept a unit step size, thereby hindering superlinear convergence. As a result, a significant question remains: How to overcome this limitation and accelerate this method?
\par This paper elucidates that imbalances among objective functions lead to a small stepsize in the variable metric method for MOPs (VMMO), which decelerates the convergence. To address this issue, Barzilai-Borwein's rule relative to the variable metric is applied to tune the gradients in the direction-finding subproblem. The main contributions of this paper can be summarized in the following points:
\par $\bullet$ We introduce the Barzilai-Borwein descent method for MOPs with variable metrics (BBDMO\_VM). This approach utilizes a variable metric for all objectives, ensuring a simplified subproblem and improved curvature exploration. We apply Barzilai-Borwein's rule relative to the metric to fine-tune the gradients in the direction-finding subproblem, effectively mitigating the effects of imbalances among objective functions.
\par $\bullet$ We investigate the convergence properties of BBDMO\_VM.  Notably, every accumulation point of the sequence generated by BBDMO\_VM is a Pareto critical point. Under the assumptions of relative smoothness and strong convexity, we establish the fast linear convergence of BBDMO\_VM. Furthermore, we demonstrate its linear convergence for problems involving some linear objectives. The fast linear convergence confirms BBDMO\_VM's potential to alleviate imbalances among objective functions.
\par $\bullet$ We stress the vital importance of metric selection in BBDMO\_VM. To capture the problem's geometry more effectively, we employ a matrix to approximate the trade-off Hessian for the multiobjective Newton-type method. From a computational standpoint, we determine the trade-off parameters based on previous information. We have also developed a variant incorporating a trade-off with quasi-Newton approximation to reduce computational costs.
\par The paper is organized as follows. In section \ref{sec2}, we introduce necessary notations and definitions that will be used later. Section \ref{sec3} revisits several multiobjective gradient descent methods. Section \ref{sec4}, we propose the BBDMO\_VM and present some preliminary lemmas. The convergence analysis of BBDMO\_VM is detailed in section \ref{sec5}. In section \ref{sec6}, we investigate the metric selection of BBDMO\_VM. The numerical results are presented in section \ref{sec7}, demonstrating the efficiency BBDMO\_VM. Finally, we draw some conclusions at the end of the paper.

\section{Preliminaries}\label{sec2}
Throughout this paper, the $n$-dimensional Euclidean space $\mathbb{R}^{n}$ is equipped with the inner product $\langle\cdot,\cdot\rangle$ and the induced norm $\|\cdot\|$. Denote $\mathbb{S}^{n}_{++}(\mathbb{S}^{n}_{+})$ the set of symmetric (semi-)positive definite matrices in $\mathbb{R}^{n\times n}$.  We denote by $JF(x)\in\mathbb{R}^{m\times n}$ the Jacobian matrix of $F$ at $x$, by $\nabla F_{i}(x)\in\mathbb{R}^{n}$ the gradient of $F_{i}$ at $x$ and by $\nabla^{2}F_{i}(x)\in\mathbb{R}^{n\times n}$ the Hessian matrix of $F_{i}$ at $x$. For a positive definite matrix $H$, the notation $\|x\|_{H}=\sqrt{\langle x,Hx \rangle}$ is used to represent the norm induced by $H$ on vector $x$. For simplicity, we denote $[m]:=\{1,2,...,m\}$, and $$\Delta_{m}:=\left\{\lambda:\sum\limits_{i\in[m]}\lambda_{i}=1,\lambda_{i}\geq0,\ i\in[m]\right\}$$ the $m$-dimensional unit simplex. To prevent any ambiguity, we establish the order $\preceq(\prec)$ in $\mathbb{R}^{m}$ as follows: $$u\preceq(\prec)v~\Leftrightarrow~v-u\in\mathbb{R}^{m}_{+}(\mathbb{R}^{m}_{++}),$$
and in $\mathbb{S}^{n}$ as:
$$U\preceq(\prec)V~\Leftrightarrow~V-U\in\mathbb{S}^{n}_{+}(\mathbb{S}^{n}_{++}).$$
\par In the following, we introduce the concepts of optimality for (\ref{MOP}) in the Pareto sense. 
\vspace{2mm}
\begin{definition}\label{def1}
	A vector $x^{\ast}\in\mathbb{R}^{n}$ is called Pareto solution to (\ref{MOP}), if there exists no $x\in\mathbb{R}^{n}$ such that $F(x)\preceq F(x^{\ast})$ and $F(x)\neq F(x^{\ast})$.
\end{definition}
\vspace{2mm}
\begin{definition}\label{def2}
	A vector $x^{\ast}\in\mathbb{R}^{n}$ is called weakly Pareto solution to (\ref{MOP}), if there exists no $x\in\mathbb{R}^{n}$ such that $F(x)\prec F(x^{\ast})$.
\end{definition}
\vspace{2mm}
\begin{definition}\label{def3}
	A vector $x^{\ast}\in\mathbb{R}^{n}$ is called  Pareto critical point of (\ref{MOP}), if
	$$\mathrm{range}(JF(x^{*}))\cap-\mathbb{R}_{++}^{m}=\emptyset,$$
	where $\mathrm{range}(JF(x^{*}))$ denotes the range of linear mapping given by the matrix $JF(x^{*})$.

\end{definition}

\par From Definitions \ref{def1} and \ref{def2}, it is evident that Pareto solutions are always weakly Pareto solutions. The following lemma shows the relationships among the three concepts of Pareto optimality.
\vspace{2mm}
\begin{lemma}[Theorem 3.1 of \cite{FD2009}] The following statements hold.
	\begin{itemize}
		\item[$\mathrm{(i)}$]  If $x\in\mathbb{R}^{n}$ is a weakly Pareto solution to (\ref{MOP}), then $x$ is Pareto critical point.
		\item[$\mathrm{(ii)}$] Let every component $F_{i}$ of $F$ be convex. If $x\in\mathbb{R}^{n}$ is a Pareto critical point of (\ref{MOP}), then $x$ is weakly Pareto solution.
		\item[$\mathrm{(iii)}$] Let every component $F_{i}$ of $F$ be strictly convex. If $x\in\mathbb{R}^{n}$ is a Pareto critical point of (\ref{MOP}), then $x$ is Pareto solution.
	\end{itemize}
\end{lemma}
\vspace{2mm}
\begin{definition}
	A differentiable function $h:\mathbb{R}^{n}\rightarrow\mathbb{R}$ is $L$-smooth if $$\nm{\nabla h(y)-\nabla h(x)}\leq\nm{y-x}$$ holds for
	all $x,y\in\mathbb{R}^{n}$. And $h$ is $\mu$-strongly convex if $$\dual{\nabla h(y)-\nabla h(x),y-x}\geq\mu\nm{y-x}^{2}$$ holds for all $x,y\in\mathbb{R}^{n}$.
\end{definition}
\vspace{2mm}
\par $L$-smoothness of $h$ implies the following quadratic upper bound:
$$h(y)\leq h(x) + \dual{\nabla h(x),y-x}+\frac{L}{2}\|y-x\|^{2},~\forall x,y\in\mathbb{R}^{n}.$$ 
On the other hand, $\mu$-strong convexity yields the quadratic lower bound:
$$h(y)\geq h(x) + \dual{\nabla h(x),y-x}+\frac{\mu}{2}\|y-x\|^{2},~\forall x,y\in\mathbb{R}^{n}.$$
When the Euclidean distance is replaced by $\|\cdot\|_{B}$, where $B$ is a positive definite matrix, then $h$ is $L$-smooth and $\mu$-strongly convex relative to $\|\cdot\|_{B}$.
\section{Gradient descent methods for MOPs}\label{sec3}
In this section, we revisit some gradient descent methods for MOPs.
\subsection{Steepest descent method}
\par For $x\in\mathbb{R}^{n}$, the steepest descent direction \cite{FS2000} is defined as the optimal solution of the following subproblem:
\begin{align}\label{eq3.1}
	\min\limits_{d\in\mathbb{R}^{n}} \max\limits_{i\in[m]}\ \langle \nabla F_{i}(x),d\rangle+\frac{1}{2}\|d\|^{2}.
\end{align}
Since $d\mapsto\langle \nabla F_{i}(x),d\rangle+\frac{1}{2}\|d\|^{2}$ is strongly convex for $i\in[m]$, then (\ref{eq3.1}) has a unique minimizer. We denote by $d_{SD}(x)$ and $\theta_{SD}(x)$ the optimal solution and optimal value of (\ref{eq3.1}), respectively. Hence,
\begin{equation}\label{E2.1}
	\theta_{SD}(x) = \min\limits_{d\in \mathbb{R}^{n}}\max\limits_{i\in[m]}\langle \nabla F_{i}(x),d\rangle+\frac{1}{2}\|d\|^{2},
\end{equation}
and
\begin{equation}\label{E2.2}
	d_{SD}(x) = \mathop{\arg\min}\limits_{d\in \mathbb{R}^{n}}\max\limits_{i\in[m]}\langle \nabla F_{i}(x),d\rangle+\frac{1}{2}\|d\|^{2}.
\end{equation}
Indeed, problem (\ref{eq3.1}) can be equivalently rewritten as the following smooth quadratic problem:
\begin{align*}\tag{QP}\label{QP}
	&\min\limits_{(t,d)\in\mathbb{R}\times\mathbb{R}^{n}}\ t+\frac{1}{2}\|d\|^{2},\\
	&\ \ \ \ \ \ \mathrm{ s.t.} \ \ \ \ \langle \nabla F_{i}(x),d\rangle \leq t,\ i\in[m].
\end{align*}
Notice that (\ref{QP}) is a convex problem with linear constraints, thus strong duality holds. The Lagrangian of (\ref{QP}) is
$$L((t,d),\lambda)=t+\frac{1}{2}\|d\|^{2} +\sum\limits_{i\in[m]}\lambda_{i}(\langle \nabla F_{i}(x),d\rangle-t).$$
By Karush-Kuhn-Tucker (KKT) conditions, we have
\begin{equation}\label{E3.3}
	\sum\limits_{i\in[m]}\lambda_{i}=1,
\end{equation}

\begin{equation}\label{E3.4}
	d + \sum\limits_{i\in[m]}\lambda_{i}\nabla F_{i}(x)=0,
\end{equation}

\begin{equation}\label{E3.5}
	\langle \nabla F_{i}(x),d\rangle \leq t,\ i\in[m],
\end{equation}

\begin{equation}\label{E3.6}
	\lambda_{i}\geq0,\ i\in[m],
\end{equation}

\begin{equation}\label{E3.7}
	\lambda_{i}(\langle \nabla F_{i}(x),d\rangle-t)=0,\ i\in[m].
\end{equation}
From (\ref{E3.4}), we obtain
\begin{equation}\label{E3.8}
	d_{SD}(x) = -\sum\limits_{i\in[m]}\lambda_{i}^{SD}(x)\nabla F_{i}(x),
\end{equation}
where $\lambda^{SD}(x)\in\Delta_{m}$ is the solution to the dual problem:
\begin{align*}\tag{DP}\label{DP}
	-\min\limits_{\lambda}&\frac{1}{2} \left\|\sum\limits_{i\in[m]}\lambda_{i}\nabla F_{i}(x)\right\|^{2}\\
	\mathrm{ s.t.}& \ \lambda\in\Delta_{m}.	
\end{align*}
Recall that strong duality holds, we obtain
\begin{equation}\label{E3.9}
	\theta(x)=-\frac{1}{2} \left\|\sum\limits_{i\in[m]}\lambda^{SD}_{i}(x)\nabla F_{i}(x)\right\|^{2}=-\frac{1}{2}\|d_{SD}(x)\|^{2}.
\end{equation}
From (\ref{E3.5}), we have
\begin{equation}\label{E3.10}
	\langle \nabla F_{i}(x),d_{SD}(x)\rangle \leq t(x)=\theta_{SD}(x)-\frac{1}{2}\|d_{SD}(x)\|^{2}=-\|d_{SD}(x)\|^{2},\ i\in[m].
\end{equation}
Denote $$\mathcal{A}_{SD}(x):=\{i:\lambda^{SD}_{i}(x)>0,~i\in[m]\}$$ the set of active constraints at $x$, then (\ref{E3.7}) leads to
\begin{equation}\label{e3.10}
	\langle \nabla F_{i}(x),d_{SD}(x)\rangle = t_{SD}(x)=-\|d_{SD}(x)\|^{2},~i\in\mathcal{A}_{SD}(x).
\end{equation}

\begin{remark}
	As described in \cite{SK2018,CTY2023}, it is recommended to solve (\ref{DP}) rather than (\ref{QP}) to obtain the steepest descent direction for the following reasons: Firstly, the dimension of dual variables, which is equal to the number of objectives, is often much less than the dimension of decision variables. Secondly, (\ref{DP}) is a quadratic programming problem with a unit simplex constraint, and it can be efficiently solved by the Frank-Wolfe/conditional gradient method (the linear subproblem has a closed-form solution since the vertices of unit simplex constraint are known).
\end{remark}

\subsection{Newton-type methods}
Similar to its counterparts for SOPs, SDMO is sensitive to problem's conditioning. In response to this challenge, Fliege et al. \cite{FD2009} proposed Newton's method for MOPs. Newton's direction is the optimal solution to the following subproblem:
\begin{align}\label{nt}
	\min\limits_{d\in\mathbb{R}^{n}} \max\limits_{i\in[m]}\ \langle \nabla F_{i}(x),d\rangle+\frac{1}{2}\|d\|_{\nabla^{2}F_{i}(x)}^{2}.
\end{align}
The dual problem can be expressed as
\begin{align*}
	-\min\limits_{\lambda}&~\frac{1}{2} \nm{\sum\limits_{i\in[m]}\lambda_{i}\nabla F_{i}(x)}^{2}_{\left[\sum\limits_{i\in[m]}\lambda_{i}\nabla^{2} F_{i}(x)\right]^{-1}}\\
	\mathrm{ s.t.} &\ \lambda\in\Delta_{m}.	
\end{align*}
Denote $\lambda^{NT}(x)\in\Delta_{m}$ the optimal solution of the dual problem, then $$d_{NT}(x)=-\left[\sum\limits_{i\in[m]}\lambda_{i}^{NT}\nabla^{2} F_{i}(x)\right]^{-1}\left(\sum\limits_{i\in[m]}\lambda_{i}^{NT}(x)\nabla F_{i}(x)\right),$$ 
and
\begin{equation}\label{NTe}
	\dual{\nabla F_{i}(x),d_{NT}(x)}=\theta_{NT}(x)-\frac{1}{2}\nm{d_{NT}(x)}^{2}_{\nabla^{2}F_{i}(x)},~\forall i\in\mathcal{A}_{NT}(x).	
\end{equation}
Since Hessian matrices are not readily available, Qu et al. \cite{QG2011} and Povalej \cite{P2014} adopted BFGS formulation to approximate the Hessian matrices, namely, replacing $\nabla^{2}F_{i}(x)$ by $B_{i}(x)$ in (\ref{nt}) for $i\in[m]$. While Newton-type methods offer attractive convergence properties like locally superlinear convergence \cite{FD2009, P2014}, the high per-step computational cost counteracts the efficiency of outer iterations, resulting in suboptimal performance from a computational perspective.

\subsection{Variable metric method}
In order to balance per-step cost and better curvature exploration, Ansary and Panda \cite{AP2015} utilized a single positive matrix to approximate all the Hessian matrices. They developed the following variable metric descent direction, which represents the optimal solution of the following subproblem:
\begin{align}
	\min\limits_{d\in\mathbb{R}^{n}} \max\limits_{i\in[m]}\ \langle \nabla F_{i}(x),d\rangle+\frac{1}{2}\|d\|_{B(x)}^{2},
\end{align}
where $B(x)$ is a positive definite matrix. The subproblem can be efficiently solved via its dual:
\begin{align*}
	-\min\limits_{\lambda}&~\frac{1}{2} \nm{\sum\limits_{i\in[m]}\lambda_{i}\nabla F_{i}(x)}^{2}_{B(x)^{-1}}\\
	\mathrm{ s.t.} &\ \lambda\in\Delta_{m}.	
\end{align*}
Denote $\lambda^{VM}(x)\in\Delta_{m}$ an optimal solution of the dual problem, then $$d_{VM}(x)=-B(x)^{-1}\left(\sum\limits_{i\in[m]}\lambda_{i}^{VM}(x)\nabla F_{i}(x)\right),$$ 
and
\begin{equation}\label{VME}
	\dual{\nabla F_{i}(x),d_{VM}(x)}=-\nm{d_{VM}(x)}^{2}_{B(x)},~\forall i\in\mathcal{A}_{VM}(x).	
\end{equation}

\par For each iteration $k$, once the unique descent direction $d^{k}\neq0$ is obtained, the classical Armijo technique is employed for line search. 

\begin{algorithm}  
	\caption{\ttfamily Armijo\_line\_search}\label{alg1} 
	\LinesNumbered 
	\KwData{$x^{k}\in\mathbb{R}^{n},d^{k}\in\mathbb{R}^{n},JF(x^{k})\in\mathbb{R}^{m\times n},\sigma,\gamma\in(0,1), t_{k}=1$}
	\While{$F(x^{k}+t_{k} d^{k})- F(x^{k}) \not\preceq t_{k}\sigma  JF(x^{k})d^{k}$}{Update $t_{k}:= \gamma t_{k}$ }
	\Return{$t_{k}$}
\end{algorithm}

\par Next, we give the lower and upper bounds of stepsize along with $d^{k}_{VM}$.
\begin{proposition}\label{p1}
		Assume that $F_{i}$ is $L_{i}^{k}$-smooth and $\mu_{i}^{k}$-strongly convex relative to $\|\cdot\|_{B_{k}}$ ,\ $i\in[m]$. Then the stepsize along with $d_{VM}^{k}$ satisfies $\min\left\{1,\frac{2\gamma(1-\sigma)}{L^{k}_{\max}}\right\}\leq t_{k}\leq \min\left\{\frac{2(1-\sigma)}
		{\mu^{k}_{\max}},1\right\}$, where $L^{k}_{\max}:=\max\{L_{i}^{k}:i\in[m]\}$, $\mu^{k}_{\max}:=\max\{\mu^{k}_{i}:i\in\mathcal{A}_{VM}(x^{k})\}$.
\end{proposition}
\begin{proof}
	We use line search condition in Algorithm \ref{alg1} to get
	\begin{equation}\label{e4.7}
		F_{i}(x^{k}+t_{k}d^{k}_{VM})-F_{i}(x^{k}) \leq \sigma t_{k}\langle\nabla F_{i}(x^{k}),d^{k}_{VM}\rangle,\ \forall i\in[m].
	\end{equation}
	From the relative $\mu_{i}^{k}$-strong convexity of $F_{i}$, we have
	\begin{equation}\label{e4.8}
		F_{i}(x^{k}+t_{k}d_{VM}^{k})-F_{i}(x^{k})\geq t_{k}\langle\nabla F_{i}(x^{k}),d_{VM}^{k}\rangle+\frac{\mu^{k}_{i}}{2}||t_{k}d_{VM}^{k}
		||_{B_{k}}^{2},\ \forall i\in[m].
	\end{equation}
	It follows by (\ref{e4.7}) and (\ref{e4.8}) that
	$$\frac{\mu^{k}_{i}}{2}||t_{k}d^{k}_{VM}
	||_{B_{k}}^{2}\leq(\sigma-1)t_{k}\langle\nabla F_{i}(x^{k}),d_{VM}^{k}\rangle,\ \forall i\in[m].$$
	Hence,
	$$\max\limits_{i\in\mathcal{A}_{VM}(x^{k})}\frac{\mu^{k}_{i}}{2}||t_{k}d^{k}_{VM}
	||_{B_{k}}^{2}\leq(1-\sigma)t_{k}||d^{k}_{VM}
	||_{B_{k}}^{2},$$
	due to the definition of $\mathcal{A}_{VM}(x^{k})$.
	Then we conclude that
	$$t_{k}\leq\frac{2(1-\sigma)}{\max\limits_{i\in\mathcal{A}_{VM}(x^{k})}\mu^{k}_{i}}=
	\frac{2(1-\sigma)}{\mu^{k}_{\max}}.$$
	Notice that $t_{k}\leq1$, thus,
	$$t_{k}\leq \min\left\{\frac{2(1-\sigma)}
	{\mu^{k}_{\max}},1\right\}.$$
Next, we analyze the lower bound. Assume that $t_{k}<1$, then backtracking is conducted, so that
	\begin{equation}\label{E4.4}
		F_{i}\left(x^{k}+\frac{t_{k}}{\gamma}d^{k}_{VM}\right)-F_{i}(x^{k})>\sigma\frac{t_{k}}{\gamma}\langle\nabla F_{i}(x^{k}),d^{k}_{VM}\rangle
	\end{equation}
	for some $i\in[m]$. From the relative $L_{i}^{k}$-smoothness of $F_{i}$, we have
	\begin{equation}\label{E4.5}
		F_{i}\left(x^{k}+\frac{t_{k}}{\gamma}d^{k}_{VM}\right)-F_{i}(x^{k})\leq\frac{t_{k}}{\gamma}\langle\nabla F_{i}(x^{k}),d^{k}_{VM}\rangle + \frac{L_{i}^{k}}{2}\left\|\frac{t_{k}}{\gamma}d^{k}_{VM}\right\|_{B_{k}}^{2}
	\end{equation}
	for all $i\in[m]$. It follows by (\ref{E4.4})-(\ref{E4.5}) and $\langle\nabla F_{i}(x^{k}),d^{k}_{VM}\rangle\leq-\|d^{k}_{VM}\|_{B_{k}}^{2}$ that
	\begin{equation*}\label{E4.6}
		t_{k}\geq\frac{2\gamma(1-\sigma)}{L^{k}_{i}}
	\end{equation*}
	for some $i\in[m]$. It follows that
	\begin{equation*}\label{E4.8}
		t_{k}\geq\frac{2\gamma(1-\sigma)}{L^{k}_{\max}}.
	\end{equation*}
	Note that the above analysis is under the assumption $t_{k}<1$, then
	\begin{equation*}\label{E4.9}
		t_{k}\geq\min\left\{1,\frac{2\gamma(1-\sigma)}{L^{k}_{\max}}\right\}.
	\end{equation*}
	This completes the proof.
\end{proof}
\begin{remark}
	The stepsize along with $d^{k}_{VM}$ can be relatively small when  $\mu^{k}_{\max}$  has a significant value, even if $F_{i}$ is not ill-conditioned relative to $\|\cdot\|_{B_{k}}$ (a relatively small value of $\frac{L^{k}_{i}}{\mu^{k}_{i}}$). This small stepsize hampers the local superlinear convergence of VMMO and leads to inferior performance.
\end{remark}

\begin{remark}
	In \cite{CLY2023}, an aggregated line search was employed to achieve larger stepsizes, resulting in local superlinear convergence for VMMO. However, it is essential to note that the aggregated line search cannot guarantee that all objective functions decrease in each iteration, and the global convergence of VMMO with the aggregated line search approach remains unestablished.
\end{remark}
\subsection{Barzilai-Borwein descent method}
As described in \cite{CTY2023}, imbalances among objective functions lead to small stepsize in SDMO, which decelerates the convergence. This is primarily due to equation (\ref{E3.10}), where the steepest descent direction results in a similar decrease in objective values for different objectives between two consecutive iterations. Observe the equation (\ref{VME}); imbalances among objective functions also decelerate the convergence of VMMO. It is worth noting that Newton's direction achieves distinctive inner products for different objectives, as shown in (\ref{NTe}). This explains why Newton-type methods accommodate larger stepsize and have the potential to alleviate imbalances among objective functions. 
\par To achieve distinctive inner products between descent direction and gradients for a first-order method, Chen et al. \cite{CTY2023} devised the Barzilai-Borwein descent direction, which is the optimal solution of the following subproblem:
\begin{equation}\label{dbb}
	\min\limits_{d\in\mathbb{R}^{n}}\max\limits_{i\in[m]}\left\{\frac{\langle\nabla F_{i}(x^{k}),d\rangle}{\alpha_{i}(x^{k})}+\frac{1}{2}\nm{d}^{2}\right\},
\end{equation}
where $\alpha(x^{k})\in\mathbb{R}^{m}_{++}$ is given by Barzilai-Borwein method:
\begin{equation}\label{bbalpha_k}
	\alpha_{i}(x^{k})=\left\{
	\begin{aligned}
		&\max\left\{\alpha_{\min},\min\left\{\frac{\langle s^{k-1},y^{k-1}_{i}\rangle}{\nm{s^{k-1}}^{2}}, \alpha_{\max}\right\}\right\}, & \langle s^{k-1},y^{k-1}_{i}\rangle&>0, \\
		&\max\left\{\alpha_{\min},\min\left\{\frac{\nm{y^{k-1}_{i}}}{\nm{s^{k-1}}}, \alpha_{\max}\right\}\right\}, & \langle s^{k-1},y^{k-1}_{i}\rangle&<0, \\
		& \alpha_{\min}, &  \langle s^{k-1},y^{k-1}_{i}\rangle&=0,
	\end{aligned}
	\right.
\end{equation}
for all $i\in[m]$, where $\alpha_{\max}$ is a sufficient large positive constant and $\alpha_{\min}$ is a sufficient small positive constant, $s^{k-1}=x^{k}-x^{k-1},\ y^{k-1}_{i}=\nabla F_{i}({x^{k}})-\nabla F_{i}(x^{k-1}),\ i\in[m].$ In this case, the dual problem can be written as:
\begin{align*}
	-\min\limits_{\lambda}&~\frac{1}{2} \nm{\sum\limits_{i\in[m]}\frac{\lambda_{i}\nabla F_{i}(x^{k})}{\alpha_{i}(x^{k})}}^{2}\\
	\mathrm{ s.t.} &\ \lambda\in\Delta_{m}.
\end{align*}
Denote $\lambda^{BB}(x^{k})$ an optimal solution of the dual problem. Similarly, we have $$d_{BB}^{k}=-\sum\limits_{i\in[m]}\frac{\lambda_{i}^{BB}(x^{k})\nabla F_{i}(x^{k})}{\alpha_{i}(x^{k})},$$
and
\begin{equation}\label{bbEdiff}
	\langle\nabla F_{i}(x^{k}),d_{BB}^{k}\rangle=-\alpha_{i}(x^{k})\|d^{k}_{BB}\|^{2},~\forall i \in\mathcal{A}_{BB}(x^{k}).
\end{equation}
It is evident that $\langle\nabla F_{i}(x^{k}),d_{BB}^{k}\rangle\neq\langle\nabla F_{j}(x^{k}),d_{BB}^{k}\rangle,~\forall i,j\in\mathcal{A}_{BB}(x^{k})$ due to the objective-based $\alpha_{i}(x^{k})$. We establish the following bounds for the stepsize along the Barzilai-Borwein descent direction.
\begin{lemma}[Proposition 2 of \cite{CTY2023}]
	Assume that $F_{i}$ is $L_{i}$-smooth and $\mu_{i}$-strongly convex for $i\in[m]$, and let $\sigma\leq\frac{1}{2}$ in line search. Then the stepsize along with $d^{k}_{BB}$ satisfies $\min\{1,\bar{t}_{\min}\}\leq t_{k}\leq1$, where $\bar{t}_{\min}:=\min\left\{\frac{2\gamma(1-\sigma)\mu_{i}}{L_{i}}:i\in[m]\right\}$.
\end{lemma}
\begin{remark}
The Barzilai-Borwein descent method for MOPs (BBDMO) can attain relatively large stepsizes as long as all objective functions are not ill-conditioned. Recently, Chen et al. \cite{CTY2023b} demonstrated that BBDMO can mitigate interference and imbalances among objectives, resulting in improved convergence rates compared to SDMO. However, it is essential to note that BBDMO remains sensitive to conditioning, as observed from a theoretical perspective \cite{CTY2023b}.
\end{remark}
\section{BBDMO\_VM:~Barzilai-Borwein descent method for MOPs with variable metrics}\label{sec4}
This section attempts to develop a method that enjoys cheap per-step cost and is not sensitive to imbalances and conditioning. Before presenting the method, let us summarize the characteristics of the methods discussed in the previous section.

\begin{table}[h]
	\centering
	\begin{tabular}{|c|c|c|c|}
		\hline
		method &
		\begin{tabular}[c]{@{}c@{}}\textbf{cheap} per-step \\ cost\end{tabular} &
		\begin{tabular}[c]{@{}c@{}}\textbf{not} sensitive \\ to imbalances\end{tabular} &
		\begin{tabular}[c]{@{}c@{}}\textbf{not} sensitive \\ to conditioning\end{tabular} \\ \hline
		SDMO  & \ding{51} & \ding{55} & \ding{55} \\ \hline
		NMO   & \ding{55} & \ding{51} & \ding{51} \\ \hline
		VMMO  & \ding{51} & \ding{55} & \ding{51} \\ \hline
		BBDMO & \ding{51} & \ding{51} & \ding{55} \\ \hline
	\end{tabular}
\caption{The characteristics of SDMO, NMO, VMMO, and BBDMO.}
\end{table}
Naturally, we aim to leverage the strengths of both VMMO and BBDMO in the development of the Barzilai-Borwein descent direction with variable metrics:

\begin{equation}\label{d}
	d^{k}:=\mathop{\arg\min}\limits_{d\in\mathbb{R}^{n}}\max\limits_{i\in[m]}\left\{\frac{\langle\nabla F_{i}(x^{k}),d\rangle}{\alpha^{k}_{i}}+\frac{1}{2}\nm{d}^{2}_{B_{k}}\right\},
\end{equation}
where $\alpha^{k}\succ 0$ mitigates the imbalances among objective functions, and $B_{k}\succ0$ is applied to better capture the local geometry of the problem. Notably, the subproblem (\ref{d}) can also be efficiently solved via its dual:
\begin{align*}
	-\min\limits_{\lambda}&~\frac{1}{2} \nm{\sum\limits_{i\in[m]}\frac{\lambda_{i}\nabla F_{i}(x^{k})}{\alpha_{i}^{k}}}^{2}_{B_{k}^{-1}}\\
	\mathrm{ s.t.} &\ \lambda\in\Delta_{m},
\end{align*}
provided that $B_{k}^{-1}$ can be available with cheap computation. Denote $\lambda^{k}$ an optimal solution of the dual problem. It is evident that $$d^{k}=-B_{k}^{-1}\left(\sum\limits_{i\in[m]}\frac{\lambda_{i}^{k}\nabla F_{i}(x^{k})}{\alpha_{i}^{k}}\right),$$
and
\begin{equation}\label{Ediff}
	\langle\nabla F_{i}(x^{k}),d^{k}\rangle=-\alpha^{k}_{i}\|d^{k}\|^{2}_{B_{k}},~\forall\lambda^{k}_{i}>0.
\end{equation}
\begin{remark}
	Given that $\alpha^{k}_{i}$ is objective-based, equation (\ref{Ediff}) implies that different objective functions achieve distinct descent along the variable metric Barzilai-Borwein descent direction.
\end{remark}
\par In contrast to BBDMO, here we approximate the secant equation relative to $\|\cdot\|_{B_{k}}$ and then set
$\alpha^{k}\in\mathbb{R}^{m}_{++}$ as follows:
\begin{equation}\label{alpha_k}
	\alpha^{k}_{i}=\left\{
	\begin{aligned}
		&\max\left\{\alpha_{\min},\min\left\{\frac{\langle s^{k-1},y^{k-1}_{i}\rangle}{\nm{s^{k-1}}^{2}_{B_{k}}}, \alpha_{\max}\right\}\right\}, & \langle s^{k-1},y^{k-1}_{i}\rangle&>0, \\
		&\max\left\{\alpha_{\min},\min\left\{\frac{\nm{y^{k-1}_{i}}}{\nm{B_{k}s^{k-1}}}, \alpha_{\max}\right\}\right\}, & \langle s^{k-1},y^{k-1}_{i}\rangle&<0, \\
		& \alpha_{\min}, &  \langle s^{k-1},y^{k-1}_{i}\rangle&=0.
	\end{aligned}
	\right.
\end{equation}

Next, we will present several properties of $d^{k}$. 
\begin{lemma}\label{lem1}
	Assume that $aI\preceq B_{k}\preceq bI$. Let $d^{k}$ be defined as (\ref{d}) and $\alpha^{k}_{i}$ be set as (\ref{alpha_k}). Then, the following statements hold.
	\begin{itemize}
		\item[$\mathrm{(i)}$] the following assertions are equivalent:
		\subitem$\mathrm{(a)}$ The point $x^{k}$ is non-critical;
		\subitem$\mathrm{(b)}$ $d^{k}\neq0$;
		\subitem$\mathrm{(c)}$ $d^{k}$ is a descent direction.
		\item[$\mathrm{(ii)}$] if there exists a convergent subsequence $x^{k}\stackrel{\mathcal{K}}{\longrightarrow} x^{*}$ such that $d^{k}\stackrel{\mathcal{K}}{\longrightarrow}0$, then $x^{*}$ is Pareto critical.
	\end{itemize}
\end{lemma}
\begin{proof}
	Since $\alpha_{\min}\leq\alpha_{i}^{k}\leq\alpha_{\max}$ and $aI\preceq B_{k}\preceq bI$, assertion (i)  can be obtained by using the same arguments as in the proof of \cite[Lemma 3.2]{P2014}. Nexct, we prove assertion (ii). We use the definition of $d^{k}$ and the fact that $\alpha_{i}^{k}\leq\alpha_{\max}$ to get 
	\begin{equation}\label{ew13}
		\begin{aligned}
			\frac{1}{2}\|d^{k}\|_{B_{k}}^{2}
			&=\max\limits_{d\in\mathbb{R}^{n}}\min\limits_{i\in[m]}\left\{\frac{
				\left\langle\nabla F_{i}(x^{k}),-d\right\rangle}{\alpha_{i}^{k}}-\frac{1}{2}\|d\|^{2}_{B_{k}}\right\}\\
			&\geq\max\limits_{d\in\mathbb{R}^{n}}\min\limits_{i\in[m]}\left\{\frac{
				\left\langle\nabla F_{i}(x^{k}),-d\right\rangle }{\alpha_{\max}}-\frac{b}{2}\|d\|^{2}\right\}\\
			&=\frac{1}{b(\alpha_{\max})^{2}}\max\limits_{b\in\mathbb{R}^{n}}\min\limits_{i\in[m]}\left\{\left\langle\nabla F_{i}(x^{k}),-b\alpha_{\max}d\right\rangle-\frac{(b\alpha_{\max})^{2}}{2}\|d\|^{2}\right\}\\
			&=\frac{1}{b(\alpha_{\max})^{2}}\max\limits_{b\in\mathbb{R}^{n}}\min\limits_{i\in[m]}\left\{\left\langle\nabla F_{i}(x^{k}),-d\right\rangle -\frac{1}{2}\|d\|^{2}\right\}\\
			&=\frac{1}{2b(\alpha_{\max})^{2}}\|d_{SD}^{k}\|^{2}.
		\end{aligned}
	\end{equation}
This, together with the fact $d^{k}\stackrel{\mathcal{K}}{\longrightarrow}0$, implies $d^{k}_{SD}\stackrel{\mathcal{K}}{\longrightarrow}0$. Moreover, from the continuity of $d_{SD}$ (\cite[Lemma 3]{FS2000}) and the fact that $x^{k}\stackrel{\mathcal{K}}{\longrightarrow} x^{*}$, we can deduce that $d_{SD}(x^{*})=0$. The desired result follows.
\end{proof}
\begin{remark}
	The continuity of $d_{SD}$ plays a crucial role in proving the global convergence of SDMO. For BBDMO\_VM, Lemma \ref{lem1}(ii) can replace the corresponding condition. 
\end{remark}

\par We also give the lower and upper bounds of stepsize along with $d^{k}$.
\begin{proposition}
	Assume that $F_{i}$ is $L_{i}^{k}$-smooth and $\mu_{i}^{k}$-strongly convex relative to $\|\cdot\|_{B_{k}}$ ,\ $i\in[m]$, and let $\sigma\leq\frac{1}{2}$ in line search. Then the stepsize along with $d^{k}$ satisfies  $\min\{1,t_{\min}\}\leq t_{k}\leq1$, where $t_{\min}:=\min\left\{\frac{2\gamma(1-\sigma)\mu^{k}_{i}}{L^{k}_{i}}:i\in[m]\right\}$.
\end{proposition}
\begin{proof}
	From the relative $L^{k}_{i}$-smoothness and $\mu_{i}^{k}$-strong convexity of $F_{i}$, we derive that
	$$\mu^{k}_{i}\leq\alpha^{k}_{i}\leq L^{k}_{i}.$$
	Then, the lower and upper bounds can be obtained by the similar argument as
	presented in the proof of Proposition \ref{p1}.
\end{proof}	

\begin{remark}
	If $F_{i}$ is not ill-conditioned relative to $\|\cdot\|_{B_{k}}$, then line search along with $d^{k}$ can achieve a relatively large stepsize.
\end{remark}

The Barzilai-Borwein descent method for MOPs with variable metrics is described as follows.
\begin{algorithm}  
	\caption{{\ttfamily{Barzilai-Borwein\_descent\_method\_for\_MOPs\_with\_variable\_metrics}}}\label{bbvm}
	\LinesNumbered  
	\KwData{$x^{0}\in\mathbb{R}^{n},~0\prec B_{0}$}
	{Choose $x^{-1}$ in a small neighborhood of $x^{0}$}\\
	\For{$k=0,...$}{Update $\alpha^{k}_{i}$ as (\ref{alpha_k}),\ $i\in[m]$\\
		Update $d^{k}:=\mathop{\min}\limits_{d\in\mathbb{R}^{n}}\max\limits_{i\in[m]}\left\{\frac{\langle\nabla F_{i}(x^{k}),d\rangle}{\alpha^{k}_{i}}+\frac{1}{2}\nm{d}^{2}_{B_{k}}\right\}$\\
		\eIf{$d^{k}=0$}{ {\bf{return}} Pareto critical point $x^{k}$ }{$t_{k}:=$ {\ttfamily Armijo\_line\_search}$\left(x^{k},d^{k},JF(x^{k})\right)$\\
			$x^{k+1}:= x^{k}+t_{k}d^{k}$\\
		Update $B_{k}$}}  
\end{algorithm}

\section{Convergence analysis}\label{sec5}
Algorithm \ref{bbvm} terminates with a Pareto critical point in a finite number of iterations or generates an infinite sequence of noncritical points. In the sequel, we will assume that Algorithm \ref{bbvm} produces an infinite sequence of noncritical points. The main goal of this section is to analyze the convergence of BBDMO\_VM.
\subsection{Global convergence}
Before presenting the global convergence of BBDMO\_VM, some mild assumptions are presented as follows.
\begin{assumption}\label{a1}
	The sequence $\{B_{k}\}$ of matrices is uniformly positive definite, i.e., there exist two positive constants, $a$ and $b$, such that
	$$aI\preceq B_{k}\preceq bI,\ \forall k.$$
\end{assumption}
\begin{assumption}\label{a2}
	For any $x^{0}\in\mathbb{R}^{n}$, the level set $\mathcal{L}_{F}(x^{0})=\{x:\ F(x)\preceq F(x^{0})\}$ is compact.
\end{assumption}
\begin{theorem}\rm\label{t1}
	Suppose that Assumptions \ref{a1} and \ref{a2} hold. Let $\{x^{k}\}$ be the sequence generated by Algorithm \ref{bbvm}. Then $\{x^{k}\}$ has at least one accumulation point, and every accumulation point is a Pareto critical point.
\end{theorem}
\begin{proof}
	Note that $\alpha_{\min}\leq\alpha^{k}_{i}\leq\alpha_{\max}$, the assertion can be obtained by using the similar arguments as in the proof of \cite[Theorem 1]{CLY2023}.
\end{proof}
\subsection{Strong convergence}
The strong convergence of gradient descent methods for MOPs is typically analyzed under the convexity assumption. However, as described in \cite{CLY2023}, it is possible to establish the strong convergence property using a second-order sufficient condition, which can hold even without convexity.
\begin{assumption}\label{a3}
	The sequence $\{x^{k}\}$ generated by Algorithm \ref{bbvm} possesses an accumulation point $x^{*}$ and there exists $\lambda^{*}\in\Delta_{m}$ such that
	$$\sum\limits_{i=1}^{m}\lambda^{*}_{i}\nabla F_{i}(x^{*})=0,$$
	and
	$$\dual{d,\left(\sum\limits_{i=1}^{m}\lambda^{*}_{i}\nabla^{2}F_{i}(x^{*})\right)d}>0, \ \ \forall d\in\Omega,$$ where $\Omega:=\{d\neq0: \dual{\nabla F_{i}(x^{*}),d}=0,\ i\in[m]\}$.
\end{assumption}
\begin{theorem}\rm\label{t2}
	Suppose that Assumptions \ref{a1}, \ref{a2} and \ref{a3} hold. Let $\{x^{k}\}$ be the sequence generated by Algorithm \ref{bbvm}. Then $\lim\limits_{k\rightarrow\infty}x^{k}=x^{*}$.  
\end{theorem}
\begin{proof}
	 The proof follows a similar approach as in \cite[Theorem 2]{CLY2023}, we omit it here.
\end{proof}

\subsection{Linear convergence}
Before presenting the linear convergence of BBDMO\_VM, we introduce two types of merit functions for (\ref{MOP}) that quantify the gap between the current point and the optimal solution. 
\begin{equation}\label{u}
	u_{0}^{\alpha}(x):=\sup\limits_{y\in\mathbb{R}^{n}}\min\limits_{i\in[m]}\left\{\frac{F_{i}(x)-F_{i}(y)}{\alpha_{i}}\right\},
\end{equation}

\begin{equation}\label{v}
	w_{\ell}^{\alpha}(x):=\max\limits_{y\in\mathbb{R}^{n}}\min\limits_{i\in[m]}\left\{\frac{
		\left\langle\nabla F_{i}(x),x-y\right\rangle}{\alpha_{i}}-\frac{\ell}{2}\|x-y\|_{B(x)}^{2}\right\},
\end{equation}
where $\alpha\in\mathbb{R}^{m}_{++}$, $\ell>0$.

\par We can demonstrate that $u_{0}^{\alpha}$ and $w_{\ell}^{\alpha}$ serve as merit functions, satisfying the criteria of weak Pareto and critical point, respectively.
\vspace{2mm}
\begin{proposition}
	Suppose that $B(x)$ is positive definite for $x\in\mathbb{R}^{n}$. Let $u_{0}^{\alpha}$ and $w_{\ell}^{\alpha}$ be defined as {\rm(\ref{u})} and {\rm(\ref{v})}, respectively. Then, the following statements hold.
	\begin{itemize}
		\item[$\mathrm{(i)}$]  $x\in\mathbb{R}^{n}$ is a weak Pareto solution of {\rm(\ref{MOP})} if and only if $u_{0}^{\alpha}(x)=0$.
		\item[$\mathrm{(ii)}$]  $x\in\mathbb{R}^{n}$ is a Pareto critical point of {\rm(\ref{MOP})} if and only if $w_{\ell}^{\alpha}(x)=0$.
	\end{itemize}
\end{proposition}
\begin{proof}
	The assertion (i) and (ii) can be obtained by using the same arguments as in the proofs of \cite[Theorem 3.1]{TFY2020} and \cite[Theorem 3.9]{TFY2020}, respectively.
\end{proof}
\par We are now in the position to present the linear convergence of BBDMO\_VM.
\begin{theorem}\label{t3}
	Assume that $F_{i}$ is $L_{i}^{k}$-smooth and $\mu_{i}^{k}$-strongly convex ($\mu_{i}^{k}\geq\delta_{i}>0$) relative to $\|\cdot\|_{B_{k}}$ for $i\in[m]$, and Assumption \ref{a1} holds. Let $\{x^{k}\}$ be the sequence generated by Algorithm \ref{bbvm}. Then, the following statements hold.
\begin{itemize}
	\item[$\mathrm{(i)}$] $\{x^{k}\}$ converges to some Pareto solution $x^{*}$.
	\item[$\mathrm{(ii)}$] $u^{\delta}_{0}(x^{k+1})\leq\left(1-2\gamma\sigma(1-\sigma) \min\limits_{i\in[m]}\left\{\frac{\mu^{k}_{i}\delta_{i}^{2}}{(L^{k}_{i})^{3}}\right\}\right)u^{\delta}_{0}(x^{k}).$
\end{itemize}
\end{theorem}
\begin{proof}
	(i) From the relative $\mu^{k}_{i}$-strong convexity of $F_{i}$ and uniformly positive definiteness of $B_{k}$, we can conclude that $F_{i}$ is strongly convex for $i\in[m]$. Consequently, Assumptions \ref{a2} and \ref{a3} hold. Then, the assertion (i) is a consequence of Theorem \ref{t2}.
	\par(ii) We use the relative $L_{i}^{k}$-smoothness and $\mu^{k}_{i}$-strong convexity of $F_{i}$ to get 
	$$\mu^{k}_{i}\leq\alpha^{k}_{i}\leq L^{k}_{i},~i\in[m].$$
	The line search holds that
	$$F_{i}(x^{k+1})-F_{i}(x^{k})\leq t_{k}\sigma\dual{\nabla F_{i}(x^{k}),d^{k}}.$$
	By direct calculation, we have
		\begin{equation}\label{E22}
		\begin{aligned}
			\frac{F_{i}(x^{k+1})-F_{i}(x^{k})}{\alpha^{k}_{i}}
			&\leq t_{k}\sigma\left(\frac{\left\langle\nabla F_{i}(x^{k}),d^{k}\right\rangle}{\alpha^{k}_{i}}+\frac{1}{2}\|d^{k}\|_{B_{k}}^{2}\right)\\
			&\leq t_{k}\sigma\max\limits_{i\in[m]}\left\{\frac{\left\langle\nabla F_{i}(x^{k}),d^{k}\right\rangle}{\alpha^{k}_{i}}+\frac{1}{2}\|d^{k}\|_{B_{k}}^{2}\right\}\\
			&=-t_{k}\sigma w_{1}^{\alpha^{k}}(x^{k})\\
			&\leq -2\gamma\sigma(1-\sigma) \min\limits_{i\in[m]}\left\{\frac{\mu^{k}_{i}}{L^{k}_{i}}\right\}w_{1}^{\alpha^{k}}(x^{k}),
		\end{aligned}
	\end{equation}
where the last inequality comes form the lower bound of $t_{k}$.
On the other hand, from the relative $\mu^{k}_{i}$-strong convexity of $F_{i}$, we have
\begin{equation}\label{E23}
	\begin{aligned}
		\sup\limits_{x\in\mathbb{R}^{n}}\min\limits_{i\in[m]}\left\{\frac{F_{i}(x^{k})-F_{i}(x)}{\delta_{i}}\right\}
		&\leq\sup\limits_{x\in\mathbb{R}^{n}}\min\limits_{i\in[m]}\left\{\frac{\dual{\nabla F_{i}(x^{k}),x^{k}-x}}{\delta_{i}}-\frac{\mu^{k}_{i}}{2\delta_{i}}\nm{x^{k}-x}^{2}_{B_{k}}\right\}\\
		&\leq\sup\limits_{x\in\mathbb{R}^{n}}\min\limits_{i\in[m]}\left\{\frac{\dual{\nabla F_{i}(x^{k}),x^{k}-x}}{\delta_{i}}-\frac{1}{2}\nm{x^{k}-x}^{2}_{B_{k}}\right\}\\
			&\leq\sup\limits_{x\in\mathbb{R}^{n}}\min\limits_{i\in[m]}\left\{r^{k}\frac{\dual{\nabla F_{i}(x^{k}),x^{k}-x}}{\alpha^{k}_{i}}-\frac{1}{2}\nm{x^{k}-x}^{2}_{B_{k}}\right\}\\
		&=(r^{k})^2\sup\limits_{x\in\mathbb{R}^{n}}\min\limits_{i\in[m]}\left\{\frac{\dual{\nabla F_{i}(x^{k}),\frac{x^{k}-x}{r^{k}}}}{\alpha^{k}_{i}}-\frac{1}{2}\nm{\frac{x^{k}-x}{r^{k}}}^{2}_{B_{k}}\right\}\\
		&=(r^{k})^2w_{1}^{\alpha^{k}}(x^{k}),
	\end{aligned}
\end{equation}

where the second inequality follows by $\mu^{k}_{i}\geq\delta_{i}$, and $r^{k}:=\max\limits_{i\in[m]}\left\{\frac{\alpha^{k}_{i}}{\delta_{i}}\right\}$. We use (\ref{E22}) and (\ref{E23}) to get
\begin{equation}
	\begin{aligned}
		\frac{F_{i}(x^{k+1})-F_{i}(x^{k})}{\delta_{i}}&\leq\frac{F_{i}(x^{k+1})-F_{i}(x^{k})}{\alpha^{k}_{i}}\\
		&\leq -2\gamma\sigma(1-\sigma) \min\limits_{i\in[m]}\left\{\frac{\mu^{k}_{i}}{L^{k}_{i}}\right\}\frac{1}{(r^{k})^2}u^{\delta}_{0}(x^{k})\\
		&\leq -2\gamma\sigma(1-\sigma) \min\limits_{i\in[m]}\left\{\frac{\mu^{k}_{i}\delta_{i}^{2}}{(L^{k}_{i})^{3}}\right\}u^{\delta}_{0}(x^{k}).
	\end{aligned}
\end{equation}
Consequently, for all $x\in\mathbb{R}^{n}$, we have
$$\frac{F_{i}(x^{k+1})-F_{i}(x)}{\delta_{i}}\leq\frac{F_{i}(x^{k+1})-F_{i}(x)}{\delta_{i}}-2\gamma\sigma(1-\sigma) \min\limits_{i\in[m]}\left\{\frac{\mu^{k}_{i}\delta_{i}^{2}}{(L^{k}_{i})^{3}}\right\}u^{\delta}_{0}(x^{k}).$$
Taking the supremum and minimum with respect to $x$ and $i\in[m]$ on both sides, respectively, we conclude that
$$u^{\delta}_{0}(x^{k+1})\leq\left(1-2\gamma\sigma(1-\sigma) \min\limits_{i\in[m]}\left\{\frac{\mu^{k}_{i}\delta_{i}^{2}}{(L^{k}_{i})^{3}}\right\}\right)u^{\delta}_{0}(x^{k}).$$
This completes the proof.
\end{proof}
\begin{remark}
	The BBDMO\_VM achieves rapid convergence provided all objectives are well-conditioned relative to $\|\cdot\|_{B_{k}}$. In practical applications, it is vital to choose an appropriate $B_{k}$ in each iteration to more effectively capture the local curvature of the problem.
\end{remark}

As described in \cite{CTY2023}, linear objectives often introduce substantial imbalances into problems, which significantly decelerate the convergence of SDMO (VMMO). In what follows, we will investigate the convergence property of BBDMO\_VM for problems with some linear objectives. Before presenting the convergence result, we give the following preliminaries related to stepsize and dual variables.
\begin{proposition}\label{p2}
	Assume that $F_{i}$ is linear for $i\in\mathcal{L}$, $L_{i}^{k}$-smooth ($L_{i}\geq L^{k}_{i}$) and $\mu_{i}^{k}$-strongly convex ($\mu_{i}^{k}\geq\delta_{i}>0$) relative to $\|\cdot\|_{B_{k}}$ for $i\in[m]\setminus\mathcal{L}$, respectively, and Assumption \ref{a1} holds. Let $\{x^{k}\}$ be the sequence generated by Algorithm \ref{bbvm}. Then, the following statements hold.
	\begin{itemize}
		\item[$\mathrm{(i)}$] The stepsize has the following lower bound:
		\begin{equation}\label{tk}
			t_{k}\geq2\gamma(1-\sigma)\min\limits_{i\in[m]\setminus\mathcal{L}}\left\{\frac{\mu_{i}^{k}}{L_{i}^{k}}\right\}.
		\end{equation}
		\item[$\mathrm{(ii)}$] The sum of dual variables for linear objectives has the the following upper bound: 
		\begin{equation}\label{lamk}
			\sum\limits_{i\in\mathcal{L}}\lambda^{k}_{i}\leq c\alpha_{\min},
		\end{equation}
	where $c:=(1+\frac{b}{a})\max\limits_{i\in[m]\setminus\mathcal{L}}\frac{\nm{\nabla F_{i}(x^{0})}+bL_{i}R}{\epsilon\delta_{i}}$. 
	\end{itemize}
\end{proposition}
\begin{proof}
	(i) As the proof in Proposition \ref{p1}, when the backtracking is conducted, the similar inequality (\ref{E4.4}) must hold for some $i\in[m]\setminus\mathcal{L}$. Consequently, we can derive the lower bound of $t_{k}$ in this case.
	\par(ii) Denote $a_{i}$ the gradient of linear function $F_{i},~i\in\mathcal{L}$, and $C:=\{\sum\limits_{i\in\mathcal{L}}\lambda_{i}a_{i}:\lambda\in\Delta_{|\mathcal{L}|}\}$ the convex hull of the gradients of linear objectives. Recall the assumption that $x^{k}$ is noncritical point, it follows that $0\notin C$. Consequently, we conclude that $\epsilon:=\min\limits_{x\in C}\|x\|>0$. By simple calculation, we have
	\begin{equation}\label{E27}
		\begin{aligned}
		 b\nm{d^{k}}\geq\nm{B_{k}d^{k}}&=\nm{\sum\limits_{i\in[m]}\lambda^{k}_{i}\frac{\nabla F_{i}(x^{k})}{\alpha^{k}_{i}}}\\
		 &=\nm{\frac{\sum\limits_{i\in\mathcal{L}}\lambda^{k}_{i}a_{i}}{\alpha_{\min}}+\sum\limits_{i\in[m]\setminus\mathcal{L}}\lambda^{k}_{i}\frac{\nabla F_{i}(x^{k})}{\alpha^{k}_{i}}}\\
		 &\geq\nm{\frac{\sum\limits_{i\in\mathcal{L}}\lambda^{k}_{i}a_{i}}{\alpha_{\min}}}-\nm{\sum\limits_{i\in[m]\setminus\mathcal{L}}\lambda^{k}_{i}\frac{\nabla F_{i}(x^{k})}{\alpha^{k}_{i}}}\\
		 &=	\frac{\sum\limits_{i\in\mathcal{L}}\lambda^{k}_{i}}{\alpha_{\min}}\nm{\sum\limits_{i\in\mathcal{L}}\frac{\lambda^{k}_{i}}{\sum\limits_{i\in\mathcal{L}}\lambda^{k}_{i}}a_{i}}-\nm{\sum\limits_{i\in[m]\setminus\mathcal{L}}\lambda^{k}_{i}\frac{\nabla F_{i}(x^{k})}{\alpha^{k}_{i}}}\\
		 &\geq \frac{\epsilon\sum\limits_{i\in\mathcal{L}}\lambda^{k}_{i}}{\alpha_{\min}}-\max\limits_{i\in[m]\setminus\mathcal{L}}\nm{\frac{\nabla F_{i}(x^{k})}{\delta_{i}}}.	 
		\end{aligned}
	\end{equation}
On the other hand, from the dual problem, we can derive that
\begin{align*}
	a\|d^{k}\|^{2}\leq\|d^{k}\|^{2}_{B_{k}}=\nm{\sum\limits_{i\in[m]}\frac{\lambda_{i}^{k}\nabla F_{i}(x^{k})}{\alpha_{i}^{k}}}^{2}_{B_{k}^{-1}}
	\leq\frac{1}{a}\max\limits_{i\in[m]\setminus\mathcal{L}}\nm{\frac{\nabla F_{i}(x^{k})}{\delta_{i}}}^{2}.
\end{align*}
Rearranging and substituting the above inequality into (\ref{E27}), it follows that
\begin{equation}\label{E28}
	\sum\limits_{i\in\mathcal{L}}\lambda^{k}_{i}\leq\frac{\alpha_{\min}}{\epsilon}(1+\frac{b}{a})\max\limits_{i\in[m]\setminus\mathcal{L}}\nm{\frac{\nabla F_{i}(x^{k})}{\delta_{i}}}.
\end{equation}
From the relative $\mu^{k}_{i}$-strong convexity of $F_{i}$ and uniformly positive definiteness of $B_{k}$, we can conclude that $F_{i}$ is strongly convex for $i\in[m]\setminus\mathcal{L}$. Then Assumption \ref{a2} holds. Denote $R:=\max\{\nm{x-y}:x,y\in\mathcal{L}_{F}(x^{0})\}$, we deduce that
$$\nm{\nabla F_{i}(x^{k})}\leq \nm{\nabla F_{i}(x^{0})}+\nm{\nabla F_{i}(x^{0})-\nabla F_{i}(x^{k})}\leq\nm{\nabla F_{i}(x^{0})}+bL_{i}R,~i\in[m]\setminus\mathcal{L},$$
where the last inequality is due to the facts that $F_{i}$ is $L_{i}$-smooth relative to $\|\cdot\|_{B_{k}}$ and $B_{k}\preceq bI$. This, together with (\ref{E28}), yields
$$\sum\limits_{i\in\mathcal{L}}\lambda^{k}_{i}\leq\frac{\alpha_{\min}}{\epsilon}(1+\frac{b}{a})\max\limits_{i\in[m]\setminus\mathcal{L}}\frac{\nm{\nabla F_{i}(x^{0})}+bL_{i}R}{\delta_{i}}.$$
This completes the proof.
\end{proof}
\par We are now in the position to present the linear convergence of BBDMO\_VM for problems with some linear objectives.
\begin{theorem}
	Assume that $F_{i}$ is linear for $i\in\mathcal{L}$, $L_{i}^{k}$-smooth ($L_{i}\geq L^{k}_{i}$) and $\mu_{i}^{k}$-strongly convex ($\mu_{i}^{k}\geq\delta_{i}>0$) relative to $\|\cdot\|_{B_{k}}$ for $i\in[m]\setminus\mathcal{L}$, respectively, and Assumption \ref{a1} holds. Let $\{x^{k}\}$ be the sequence generated by Algorithm \ref{bbvm}. Then, the following statements hold.
	\begin{itemize}
		\item[$\mathrm{(i)}$] $\{x^{k}\}$ converges to some Pareto solution $x^{*}$.
		\item[$\mathrm{(ii)}$] $u^{\delta^{\alpha}}_{0}(x^{k+1})\leq\left(1-2\gamma\sigma(1-\sigma)(1-c\alpha_{\min}) \min\limits_{i\in[m]}\left\{\frac{\mu^{k}_{i}\delta_{i}^{2}}{(L^{k}_{i})^{3}}\right\}\right)u^{\delta^{\alpha}}_{0}(x^{k}).$
	\end{itemize}
\end{theorem}
\begin{proof}
	(i) As described in the proof of Proposition \ref{p2}, we deduce that Assumption \ref{a2} holds. From Theorem \ref{t1}, $\{x^{k}\}$ has an accumulation point $x^{*}$, and there exists $\lambda^{*}\in\Delta_{m}$ such that 
	$$\sum\limits_{i\in[m]}\lambda^{*}_{i}\nabla F_{i}(x^{*})=0.$$
	In what follow, we prove that $\sum\limits_{i\in[m]\setminus\mathcal{L}}\lambda^{*}_{i}\neq0$. If otherwise, we have 
	$$\sum\limits_{i\in[m]}\lambda^{*}_{i}\nabla F_{i}(x^{*})=\sum\limits_{i\in\mathcal{L}}\lambda^{*}_{i}\nabla F_{i}(x^{*})=\sum\limits_{i\in\mathcal{L}}\lambda^{*}_{i}a_{i}=0,$$
	where $a_{i}$ is the gradient of linear function $F_{i},~i\in\mathcal{L}$. This indicates that every $x\in\mathbb{R}^{n}$ is a Pareto critical point. It contradicts the assumption that  Algorithm \ref{bbvm} produces an infinite sequence of noncritical points. As a result, we have $\sum_{i\in[m]\setminus\mathcal{L}}\lambda^{*}_{i}\neq0$. This, together with the strong convexity of $F_{i},~i\in[m]\setminus\mathcal{L}$, yields the Assumption \ref{a3}. Then, the assertion (i) is a consequence of Theorem \ref{t2}.
	\par(ii) We refer to the proof in Theorem \ref{t3}, then 
	$$\mu^{k}_{i}\leq\alpha^{k}_{i}\leq L^{k}_{i},~i\in[m]\setminus\mathcal{L},~\alpha^{k}_{i}=\alpha_{\min},~i\in\mathcal{L}.$$
	and 
	\begin{equation}\label{E25}
	\frac{F_{i}(x^{k+1})-F_{i}(x^{k})}{\alpha^{k}_{i}}\leq-t_{k}\sigma w_{1}^{\alpha^{k}}(x^{k}).
	\end{equation}
	Denote 
	\begin{equation*}
		\mathds{1}_{[m]\setminus\mathcal{L}}(i):=\left\{
		\begin{aligned}
			1,~~~~~&i\in[m]\setminus\mathcal{L},\\
			0,~~~~~&x\in\mathcal{L}.
		\end{aligned}
		\right.
	\end{equation*}
	By simple calculation, we have
		\begin{align*}
			&~~~~\min\limits_{d\in\mathbb{R}^{n}}\max\limits_{i\in[m]}\left\{\dual{\frac{\nabla F_{i}(x^{k})}{\alpha^{k}_{i}},d}+\frac{\mathds{1}_{[m]\setminus\mathcal{L}}(i)}{2}\nm{d}^{2}_{B_{k}}\right\}\\
			&\geq\min\limits_{d\in\mathbb{R}^{n}}\left\{\dual{\sum\limits_{i\in[m]}\frac{\lambda^{k}_{i}\nabla F_{i}(x^{k})}{\alpha^{k}_{i}},d}+\frac{\sum\limits_{i\in[m]\setminus\mathcal{L}}\lambda^{k}_{i}}{2}\nm{d}^{2}_{B_{k}}\right\}
			\end{align*}
			\begin{align*}
			&=\frac{1}{\sum\limits_{i\in[m]\setminus\mathcal{L}}\lambda^{k}_{i}}\min\limits_{d\in\mathbb{R}^{n}}\left\{\dual{\sum\limits_{i\in[m]}\frac{\lambda^{k}_{i}\nabla F_{i}(x^{k})}{\alpha^{k}_{i}},\left(\sum\limits_{i\in[m]\setminus\mathcal{L}}\lambda^{k}_{i}\right)d}+\frac{1}{2}\nm{\left(\sum\limits_{i\in[m]\setminus\mathcal{L}}\lambda^{k}_{i}\right)d}^{2}_{B_{k}}\right\}\\
			&=\frac{1}{\sum\limits_{i\in[m]\setminus\mathcal{L}}\lambda^{k}_{i}}\min\limits_{d\in\mathbb{R}^{n}}\left\{\dual{\sum\limits_{i\in[m]}\frac{\lambda^{k}_{i}\nabla F_{i}(x^{k})}{\alpha^{k}_{i}},d}+\frac{1}{2}\nm{d}^{2}_{B_{k}}\right\}\\
			&=-\frac{1}{\sum\limits_{i\in[m]\setminus\mathcal{L}}\lambda^{k}_{i}}w_{1}^{\alpha^{k}}(x^{k}),
		\end{align*}

Substituting the above inequality into (\ref{E25}), it follows that
\begin{equation}\label{E26}
	\frac{F_{i}(x^{k+1})-F_{i}(x^{k})}{\alpha^{k}_{i}}\leq t_{k}\sigma\left(\sum\limits_{i\in[m]\setminus\mathcal{L}}\lambda^{k}_{i}\right)\min\limits_{d\in\mathbb{R}^{n}}\max\limits_{i\in[m]}\left\{\dual{\frac{\nabla F_{i}(x^{k})}{\alpha^{k}_{i}},d}+\frac{\mathds{1}_{[m]\setminus\mathcal{L}}(i)}{2}\nm{d}^{2}_{B_{k}}\right\}.
\end{equation}
	Denote $$\delta^{\alpha}:=\{(\delta^{\alpha}_{1},...,\delta^{\alpha}_{m}):\delta^{\alpha}_{i}=\delta_{i},\delta^{\alpha}_{j}=\alpha_{\min},i\in[m]\setminus\mathcal{L},j\in\mathcal{L}\}.$$
	We use the relative $\mu^{k}_{i}$-strong convexity and of $F_{i}$ for $i\in[m]\setminus\mathcal{L}$ to get
	\begin{align*}
		&~~~~\sup\limits_{x\in\mathbb{R}^{n}}\min\limits_{i\in[m]}\left\{\frac{F_{i}(x^{k})-F_{i}(x)}{\delta^{\alpha}_{i}}\right\}\\
		&\leq\sup\limits_{x\in\mathbb{R}^{n}}\min\limits_{i\in[m]}\left\{\frac{\dual{\nabla F_{i}(x^{k}),x^{k}-x}}{\delta^{\alpha}_{i}}-\frac{\mathds{1}_{[m]\setminus\mathcal{L}}(i)}{2}\|x^{k}-x\|^{2}_{B_{k}}\right\}\\
		&\leq\sup\limits_{x\in\mathbb{R}^{n}}\min\limits_{i\in[m]}\left\{\bar{r}^{k}\frac{\dual{\nabla F_{i}(x^{k}),x^{k}-x}}{\alpha^{k}_{i}}-\frac{\mathds{1}_{[m]\setminus\mathcal{L}}(i)}{2}\|x^{k}-x\|^{2}_{B_{k}}\right\}\\
		&={(\bar{r}^{k})^{2}}\sup\limits_{x\in\mathbb{R}^{n}}\min\limits_{i\in[m]}\left\{\frac{\dual{\nabla F_{i}(x^{k}),\frac{x^{k}-x}{\bar{r}^{k}}}}{\alpha^{k}_{i}}-\frac{\mathds{1}_{[m]\setminus\mathcal{L}}(i)}{2}\|\frac{x^{k}-x}{\bar{r}^{k}}\|^{2}_{B_{k}}\right\}\\
		&=-{(\bar{r}^{k})^{2}}\min\limits_{d\in\mathbb{R}^{n}}\max\limits_{i\in[m]}\left\{\dual{\frac{\nabla F_{i}(x^{k})}{\alpha^{k}_{i}},d}+\frac{\mathds{1}_{[m]\setminus\mathcal{L}}(i)}{2}\nm{d}^{2}_{B_{k}}\right\}
	\end{align*}
where $\bar{r}^{k}:=\max\limits_{i\in[m]}\left\{\frac{\alpha^{k}_{i}}{\delta^{\alpha}_{i}}\right\}$. Substituting the above inequality into (\ref{E26}) and utilizing $\delta^{\alpha}_{i}\leq\alpha^{k}_{i}$, we conclude that
\begin{equation*}
	\begin{aligned}
	\frac{F_{i}(x^{k+1})-F_{i}(x^{k})}{\delta^{\alpha}_{i}}&\leq\frac{F_{i}(x^{k+1})-F_{i}(x^{k})}{\alpha^{k}_{i}}\\
	&\leq -t_{k}\sigma\left(\sum\limits_{i\in[m]\setminus\mathcal{L}}\lambda^{k}_{i}\right)\frac{1}{(\bar{r}^{k})^{2}}u^{\delta^{\alpha}}_{0}(x^{k})\\
	&\leq-t_{k}\sigma\left(\sum\limits_{i\in[m]\setminus\mathcal{L}}\lambda^{k}_{i}\right)\left(\min\limits_{i\in[m]\setminus\mathcal{L}}\left\{\frac{\delta_{i}}{L_{i}^{k}}\right\}\right)^{2}u^{\delta^{\alpha}}_{0}(x^{k}),
	\end{aligned}
\end{equation*}
where the last inequality comes from the facts $\alpha^{k}_{i}\leq L^{k}_{i},~i\in[m]\setminus\mathcal{L}$ and $\alpha^{k}_{i}=\alpha_{\min},~i\in\mathcal{L}$.	
Then, for all $x\in\mathbb{R}^{n}$, we have
$$\frac{F_{i}(x^{k+1})-F_{i}(x)}{\delta^{\alpha}_{i}}\leq\frac{F_{i}(x^{k})-F_{i}(x)}{\delta^{\alpha}_{i}}-t_{k}\sigma\left(\sum\limits_{i\in[m]\setminus\mathcal{L}}\lambda^{k}_{i}\right)\left(\min\limits_{i\in[m]\setminus\mathcal{L}}\left\{\frac{\delta_{i}}{L_{i}^{k}}\right\}\right)^{2}u^{\delta^{\alpha}}_{0}(x^{k}).$$	
Taking the supremum and minimum with respect to $x\in\mathbb{R}^{n}$ and $i\in[m]$ on both sides, respectively, we obtain	
$$u^{\delta^{\alpha}}_{0}(x^{k+1})\leq\left(1-t_{k}\sigma\left(\sum\limits_{i\in[m]\setminus\mathcal{L}}\lambda^{k}_{i}\right)\left(\min\limits_{i\in[m]\setminus\mathcal{L}}\left\{\frac{\delta_{i}}{L_{i}^{k}}\right\}\right)^{2}\right)u^{\delta^{\alpha}}_{0}(x^{k}).$$	
Thus, the desired result follows by substituting (\ref{tk}) and (\ref{lamk}) into the above inequality.	
\end{proof}

\section{Metric selection in VMPGMO}\label{sec6}
In the context of the linear convergence of BBDMO\_VM discussed in the previous section, the selection of $B_{k}$ can profoundly influence the convergence rate. Naturally, a remaining question is: How do we choose $B_{k}$ to enhance performance? In the realm of single-objective optimization problems (SOPs), a well-known approach is to set $B_{k}=\nabla ^{2}f(x^{k})$, which corresponds to Newton's method. However, for multiobjective optimization problems (MOPs), we cannot use any single $\nabla^{2}f_{i}(x^{k})$ to simultaneously approximate all Hessian matrices, especially when the Hessian matrices are distinct from each other. Notably, SDMO can be interpreted as an implicit gradient descent method with adaptive scalarization (where the weight vector is the optimal solution of a dual problem). From the perspective of scalarization, a judicious choice for $B_{k}$ is to approximate the variable aggregated Hessian. The following subsections will provide details on selecting $B_{k}$.
\subsection{Trade-off of Hessian matrices}
In this subsection, $d^{k}$ approximates multiobjective Newton direction by selecting appropriate $B_{k}$. At each iteration $k$, the multiobjective Newton direction is 
$$-[\nabla^{2}F_{\lambda^{N}(x^{k})}(x^{k})]^{-1}\nabla F_{\lambda^{N}(x^{k})}(x^{k}),$$
the Barzilai-Borwein descent direction with variable metric is 
$$-B_{k}^{-1}\nabla F_{\frac{\lambda^{k}}{\alpha^{k}}}(x^{k}).$$
From observation, a reasonable choice of $B_{k}$ is $\nabla^{2}F_{\frac{\lambda^{k}}{\alpha^{k}}}(x^{k})$. 
In general $\frac{\lambda^{k}}{\alpha^{k}}\not\in\Delta_{m}$, but there exists a $\bar{\lambda}^{k}:=\frac{\lambda^{k}}{\alpha^{k}}/(\sum_{i\in[m]}\frac{\lambda^{k}_{i}}{\alpha^{k}_{i}})\in\Delta_{m}$ such that $$d^{k}=-[\nabla^{2}F_{\bar{\lambda}^{k}}(x^{k})]^{-1}\nabla F_{\bar{\lambda}^{k}}(x^{k}).$$
As a result, the selected $B_{k}$ can be perceived as the trade-off among Hessian matrices, with the weight vector being adaptively updated in each iteration.  Unfortunately, $\lambda^{k}$ is unavailable before computing the subproblem, and $\alpha^{k}$ is determined using $B_{k}$.  As an alternative, we can replace $\frac{\lambda^{k}}{\alpha^{k}}$ with $\frac{\lambda^{k-1}}{\alpha^{k-1}}$.
\subsection{Trade-off of quasi-Newton approximation}
Two remaining shortcomings exist with $B_{k}=\nabla^{2}f_{\frac{\lambda^{k-1}}{\alpha^{k-1}}}(x^{k})$: Hessian matrices are not readily available, and obtaining the inverse of $B_{k}$ is computationally expensive. To address the issues, we use BFGS formulation to update $B_{k}$, specifically,

\begin{equation}\label{Bk}
	B_{k+1}=\left\{
	\begin{aligned}
		\begin{split}
			&B_{k}-\frac{B_{k}s_{k}s_{k}^{T}B_{k}}{\dual{s_{k},B_{k}s_{k}}}+\frac{y_{k}y_{k}^{T}}{\dual{s_{k},y_{k}}}, & &\dual{s_{k},y_{k}}>0, \\
			&B_{k}, & &\text{otherwise},\\
		\end{split}
	\end{aligned}
	\right.
\end{equation}
where $s_{k}=x^{k+1}-x^{k}$, $y_{k}=\nabla F_{\frac{\lambda^{k-1}}{\alpha^{k-1}}}(x^{k+1}) - \nabla F_{\frac{\lambda^{k-1}}{\alpha^{k-1}}}(x^{k})$. The condition $\dual{s_{k},y_{k}}>0$ guarantees the positive definiteness of $B_{k+1}$. When the condition does not hold, we set $B_{k+1}=B_{k}$ to keep the positive definiteness. The corresponding inverse matrix can be updated by
\begin{equation}\label{Bk-}
	B_{k+1}^{-1}=\left\{
	\begin{aligned}
		&\left(I-\frac{s_{k}y_{k}^{T}}{\dual{s_{k},y_{k}}}\right)B_{k}^{-1}\left(I-\frac{y_{k}s_{k}^{T}}{\dual{s_{k},y_{k}}}\right)+\frac{s_{k}s_{k}^{T}}{\dual{s_{k},y_{k}}},& &\dual{s_{k},y_{k}}>0, \\
		&B_{k}^{-1}, & &\text{otherwise}.\\
	\end{aligned}
	\right.
\end{equation}
\section{Numerical results}\label{sec7}
In this section, we present numerical results to demonstrate the performance of BBDMO\_VM for various problems, where the variable metric is the trade-off of quasi-Newton approximation. We also compare BBDMO\_VM with quasi-Newton method  (QNMO) \cite{P2014}, variable metric method (VMMO) \cite{AP2015,CLY2023} and Barzilai-Borwein descent method (BBDMO) \cite{CTY2023} to show its efficiency. All numerical experiments were implemented in Python 3.7 and executed on a personal computer with an Intel Core i7-11390H, 3.40 GHz processor, and 16 GB of RAM. The For BBDMO and BBDMO\_VM, we set $\alpha_{\min}=10^{-3}$ and $\alpha_{\max}=10^{3}$ to truncate the Barzilai-Borwein's parameter\footnote{For larger-scale and more complicated problems, smaller values for $\alpha_{\min}$ and larger values for $\alpha_{\max}$ should be selected.}. We set $\sigma=10^{-4}$ and $\gamma=0.5$ in the line search procedure. To ensure that the algorithms terminate after a finite number of iterations, we use the stopping criterion $\|d(x)\|\leq 10^{-6}$ for all tested algorithms. We also set the maximum number of iterations to 500. For each problem, we use the same initial points for different tested algorithms. The initial points are randomly selected within the specified lower and upper bounds. The subproblem of QNMO is solved by {\ttfamily scipy.optimize}, a Python-embedded modelling language for optimization problems. Based on the Frank-Wolfe method, our codes solve the subproblems of VMMO, BBDMO, and BBDMO\_VM. The recorded averages from the 200 runs include the number of iterations, the number of function evaluations, and the CPU time.

\subsection{Ordinary test problems}
The tested algorithms are executed on several test problems, and the problem illustration is given in Table \ref{tab1}. The dimensions of variables and objective functions are presented in the second and third columns, respectively. $x_{L}$ and $x_{U}$ represent lower bounds and upper bounds of variables, respectively.
\begin{table}[H]
	\centering
	\resizebox{.95\columnwidth}{!}{
		\begin{tabular}{lllllllllll}
			\hline
			Problem    &  & $n$     &           & $m$     &  & $x_{L}$               &  & $x_{U}$             &  & Reference \\ \hline
			BK1        &  & 2     & \textbf{} & 2     &  & (-5,-5)         &  & (10,10)       &  & \cite{BK1}         \\
			DD1        &  & 5     & \textbf{} & 2     &  & (-20,...,-20)   &  & (20,...,20)   &  & \cite{DD1998}         \\
			Deb        &  & 2     & \textbf{} & 2     &  & (0.1,0.1)       &  & (1,1)         &  & \cite{D1999}          \\
			Far1       &  & 2     & \textbf{} & 2     &  & (-1,-1)         &  & (1,1)         &  & \cite{BK1}         \\
			FDS        &  & 5     & \textbf{} & 3     &  & (-2,...,-2)     &  & (2,...,2)     &  & \cite{FD2009}         \\
			FF1        &  & 2     & \textbf{} & 2     &  & (-1,-1)         &  & (1,1)         &  & \cite{BK1}         \\
			Hil1       &  & 2     & \textbf{} & 2     &  & (0,0)           &  & (1,1)         &  & \cite{Hil1}         \\
			Imbalance1 &  & 2     & \textbf{} & 2     &  & (-2,-2)         &  & (2,2)         &  & \cite{CTY2023}         \\
			Imbalance2 &  & 2     & \textbf{} & 2     &  & (-2,-2)         &  & (2,2)         &  & \cite{CTY2023}         \\
			JOS1a      &  & 50    & \textbf{} & 2     &  & (-2,...,-2)     &  & (2,...,2)     &  & \cite{JO2001}         \\
			JOS1b      &  & 100   & \textbf{} & 2     &  & (-2,...,-2)     &  & (2,...,2)     &  & \cite{JO2001}         \\
			JOS1c      &  & 100   & \textbf{} & 2     &  & (-50,...,-50)   &  & (50,...,50)   &  & \cite{JO2001}         \\
			JOS1d      &  & 100   & \textbf{} & 2     &  & (-100,...,-100) &  & (100,...,100) &  & \cite{JO2001}        \\
			LE1        &  & 2     & \textbf{} & 2     &  & (-5,-5)         &  & (10,10)       &  & \cite{BK1}        \\
			PNR        &  & 2     & \textbf{} & 2     &  & (-2,-2)         &  & (2,2)         &  & \cite{PN2006}        \\
			VU1        &  & 2     & \textbf{} & 2     &  & (-3,-3)         &  & (3,3)         &  & \cite{BK1}        \\
			WIT1       &  & 2     & \textbf{} & 2     &  & (-2,-2)         &  & (2,2)         &  & \cite{W2012}       \\
			WIT2       &  & 2     & \textbf{} & 2     &  & (-2,-2)         &  & (2,2)         &  & \cite{W2012}        \\
			WIT3       &  & 2     & \textbf{} & 2     &  & (-2,-2)         &  & (2,2)         &  & \cite{W2012}        \\
			WIT4       &  & 2     & \textbf{} & 2     &  & (-2,-2)         &  & (2,2)         &  & \cite{W2012}        \\
			WIT5       &  & 2     & \textbf{} & 2     &  & (-2,-2)         &  & (2,2)         &  & \cite{W2012}        \\
			WIT6       &  & 2 & \textbf{} & 2 &  & (-2,-2)         &  & (2,2)         &  & \cite{W2012}        \\ \hline
		\end{tabular}
	}
	\caption{Description of all test problems used in numerical experiments.}
	\label{tab1}
\end{table}

%
%
%
%
%

\begin{figure}[H]
	\centering
	\subfigure[QNMO]
	{
		\begin{minipage}[H]{.22\linewidth}
			\centering
			\includegraphics[scale=0.22]{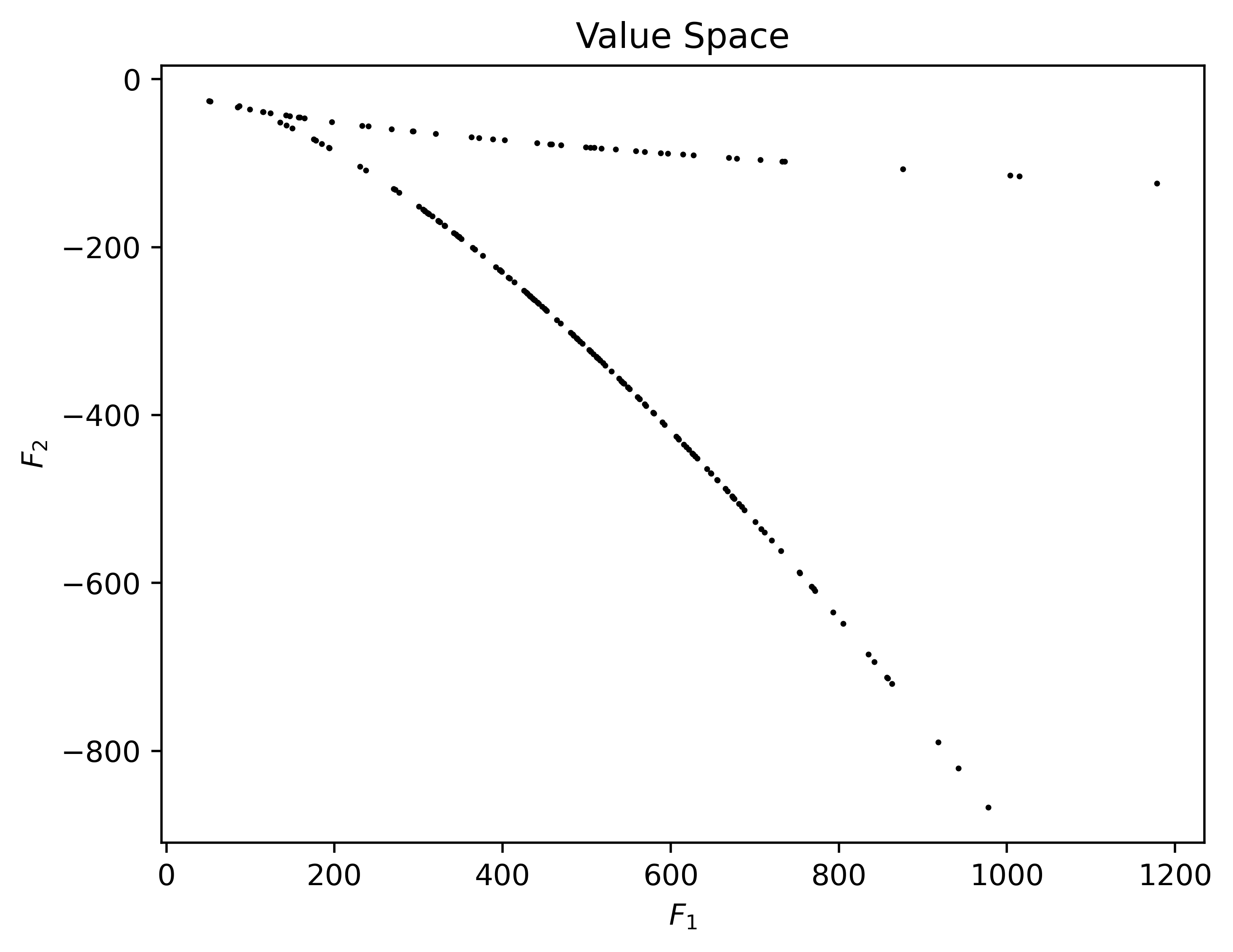} \\
			\includegraphics[scale=0.22]{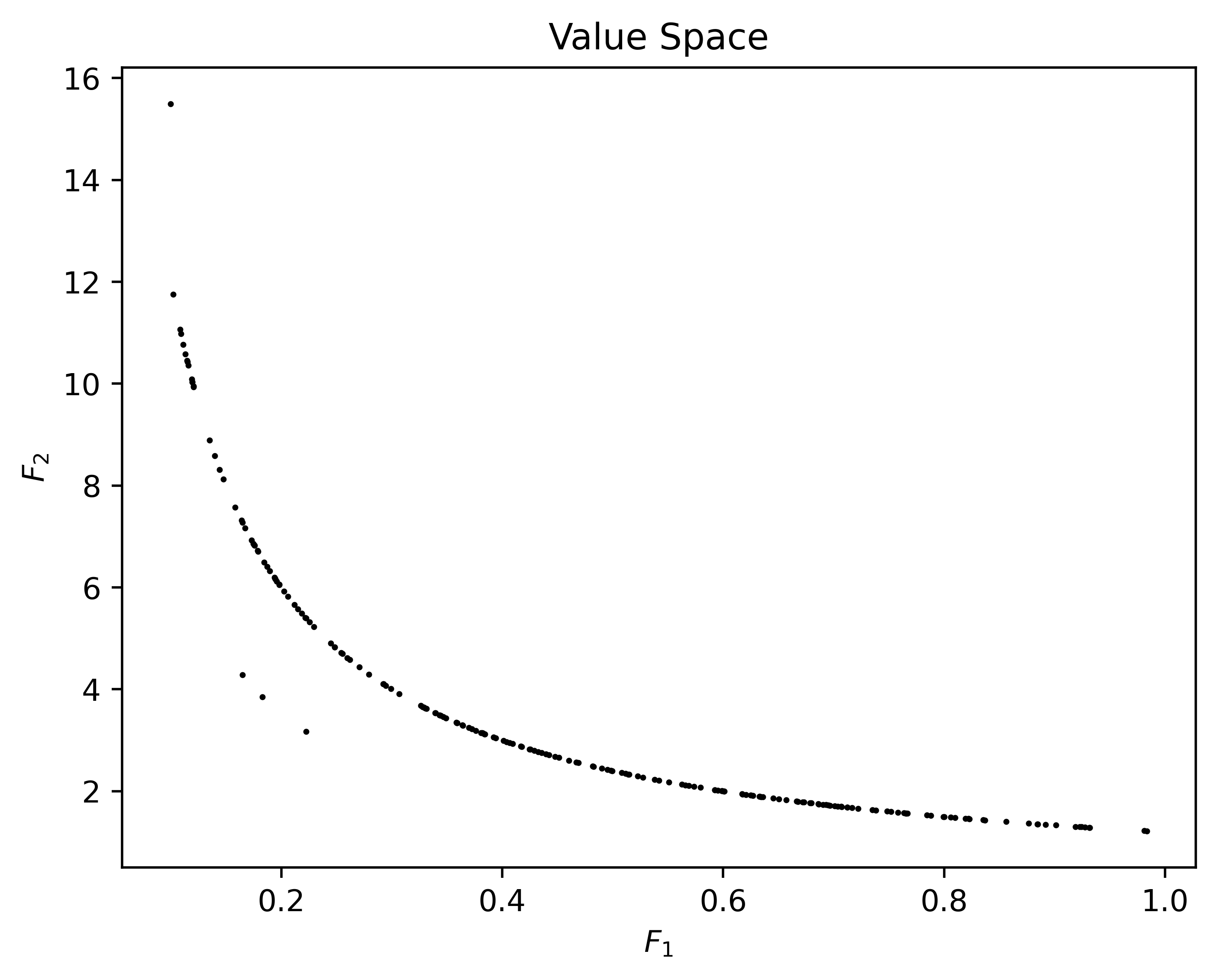} \\
			\includegraphics[scale=0.22]{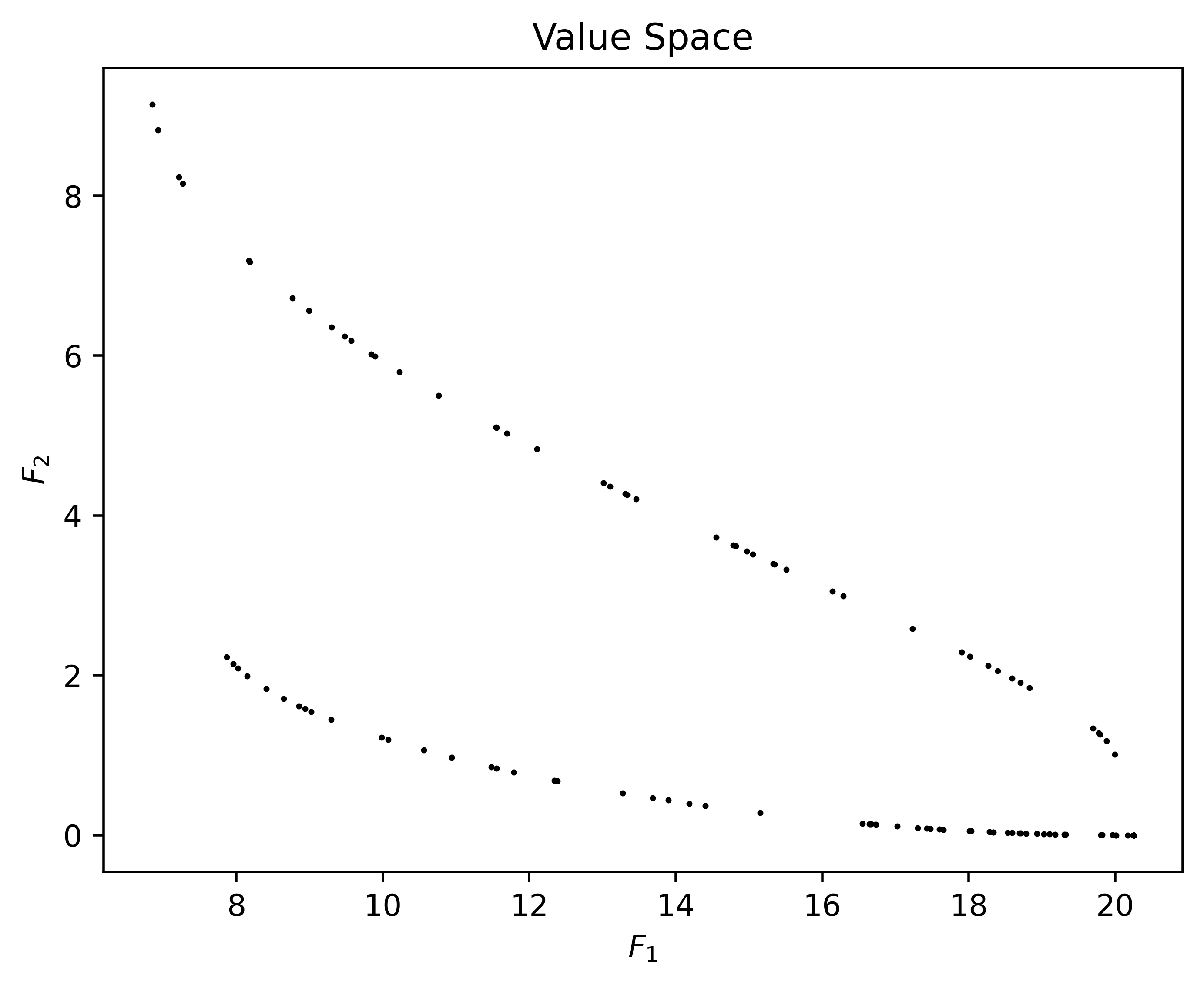}
		\end{minipage}
	}
	\subfigure[VMMO]
	{
		\begin{minipage}[H]{.22\linewidth}
			\centering
			\includegraphics[scale=0.22]{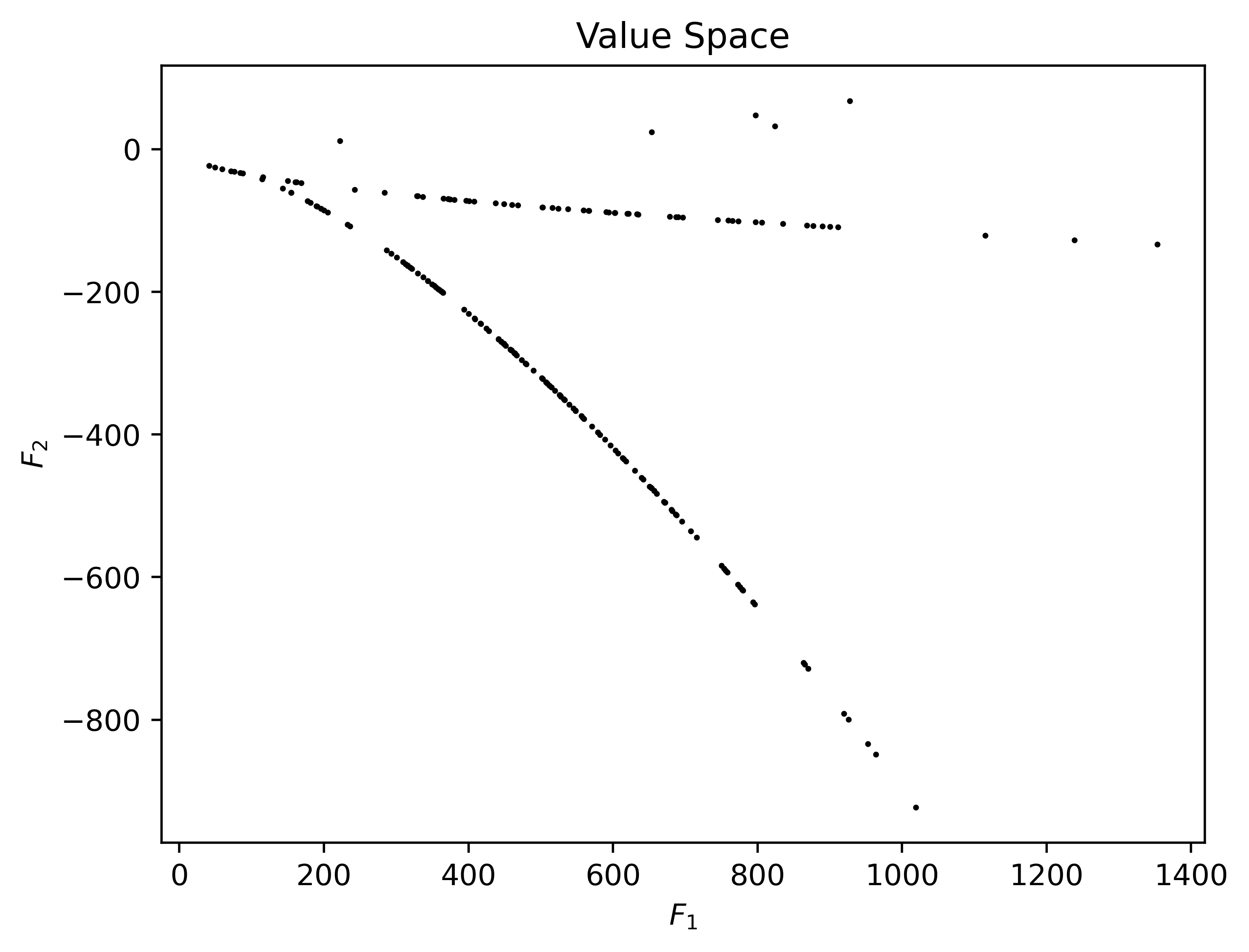}\\
			\includegraphics[scale=0.22]{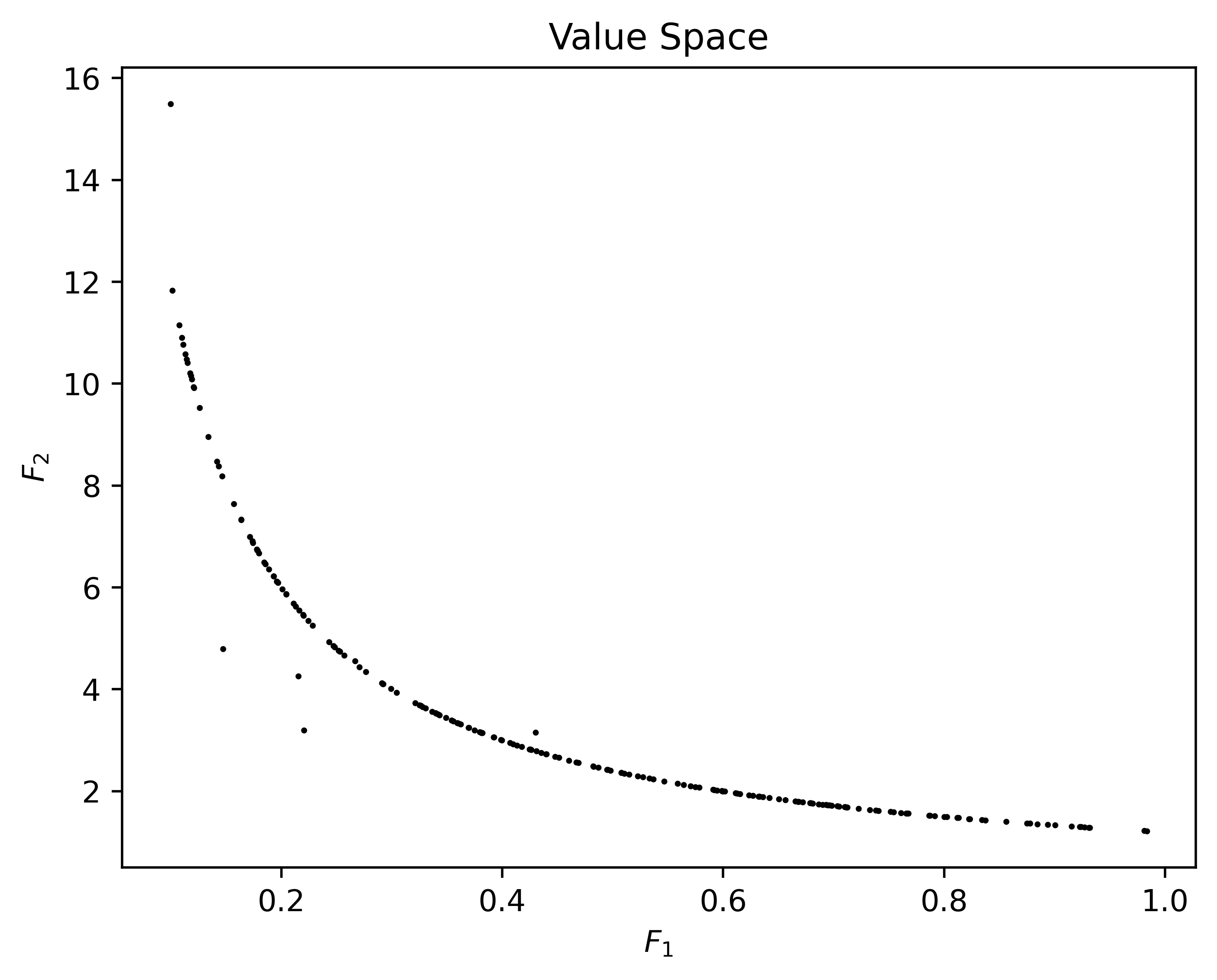}\\
			\includegraphics[scale=0.22]{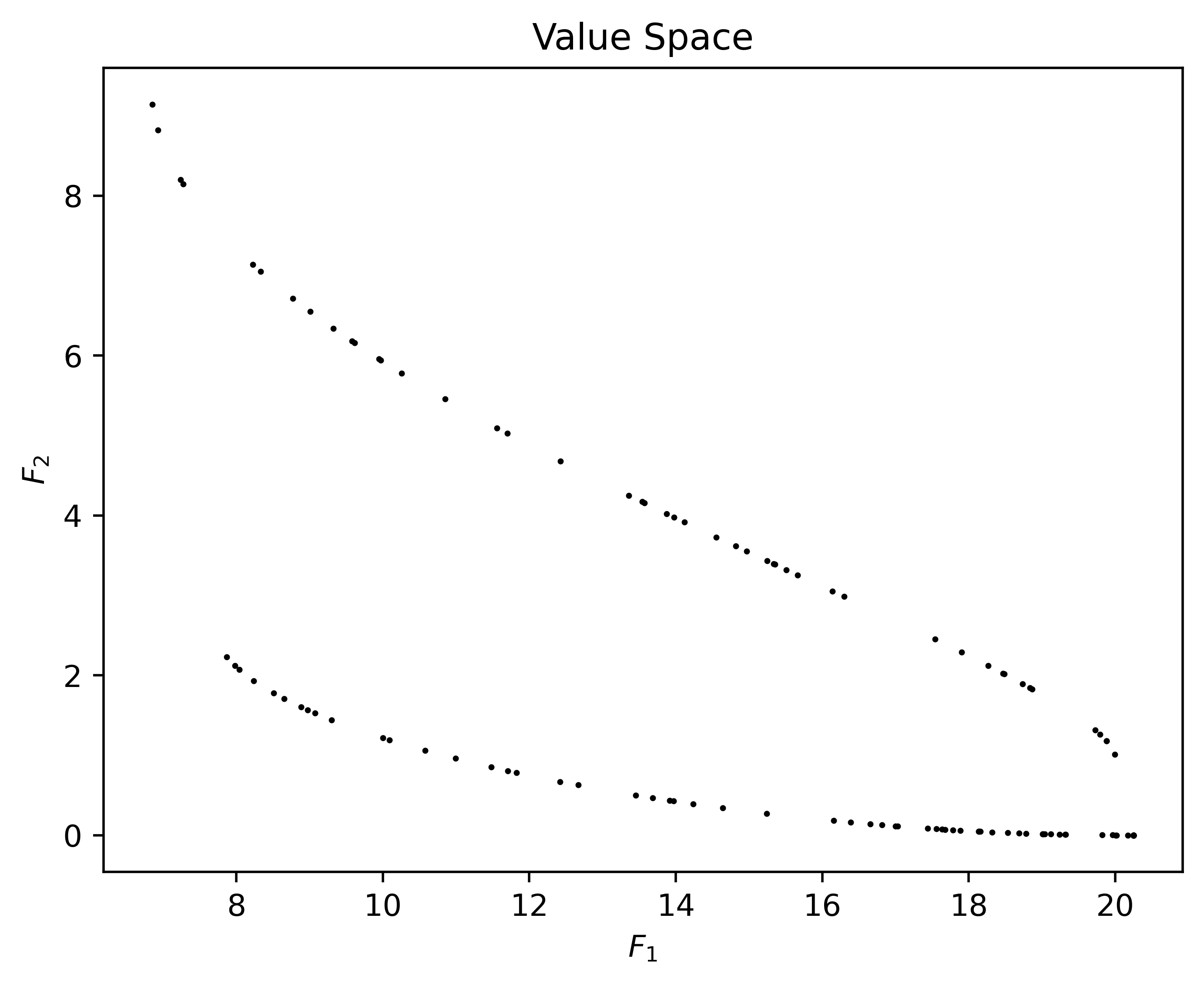}
		\end{minipage}
	}
	\subfigure[BBDMO]
	{
		\begin{minipage}[H]{.22\linewidth}
			\centering
			\includegraphics[scale=0.22]{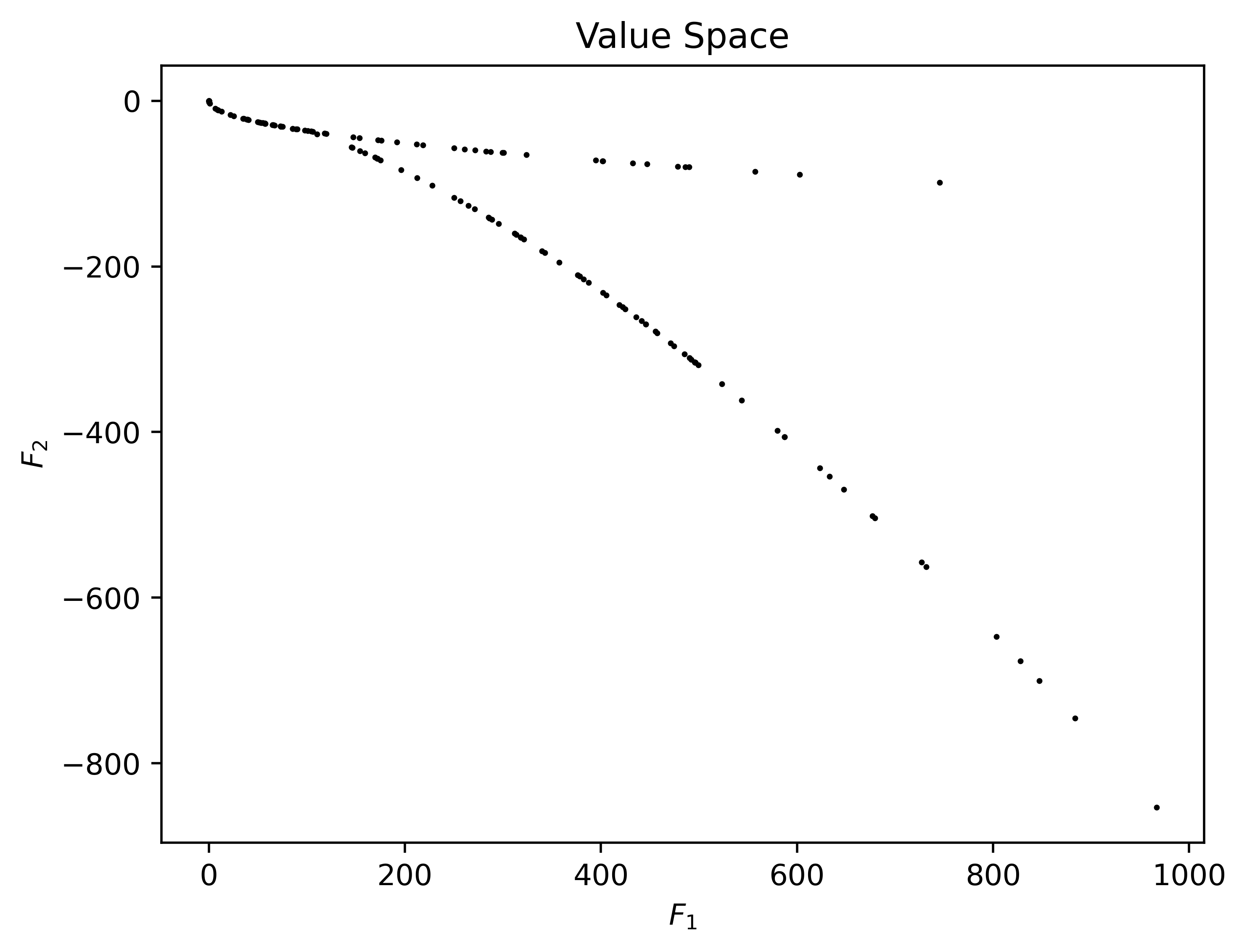} \\
			\includegraphics[scale=0.22]{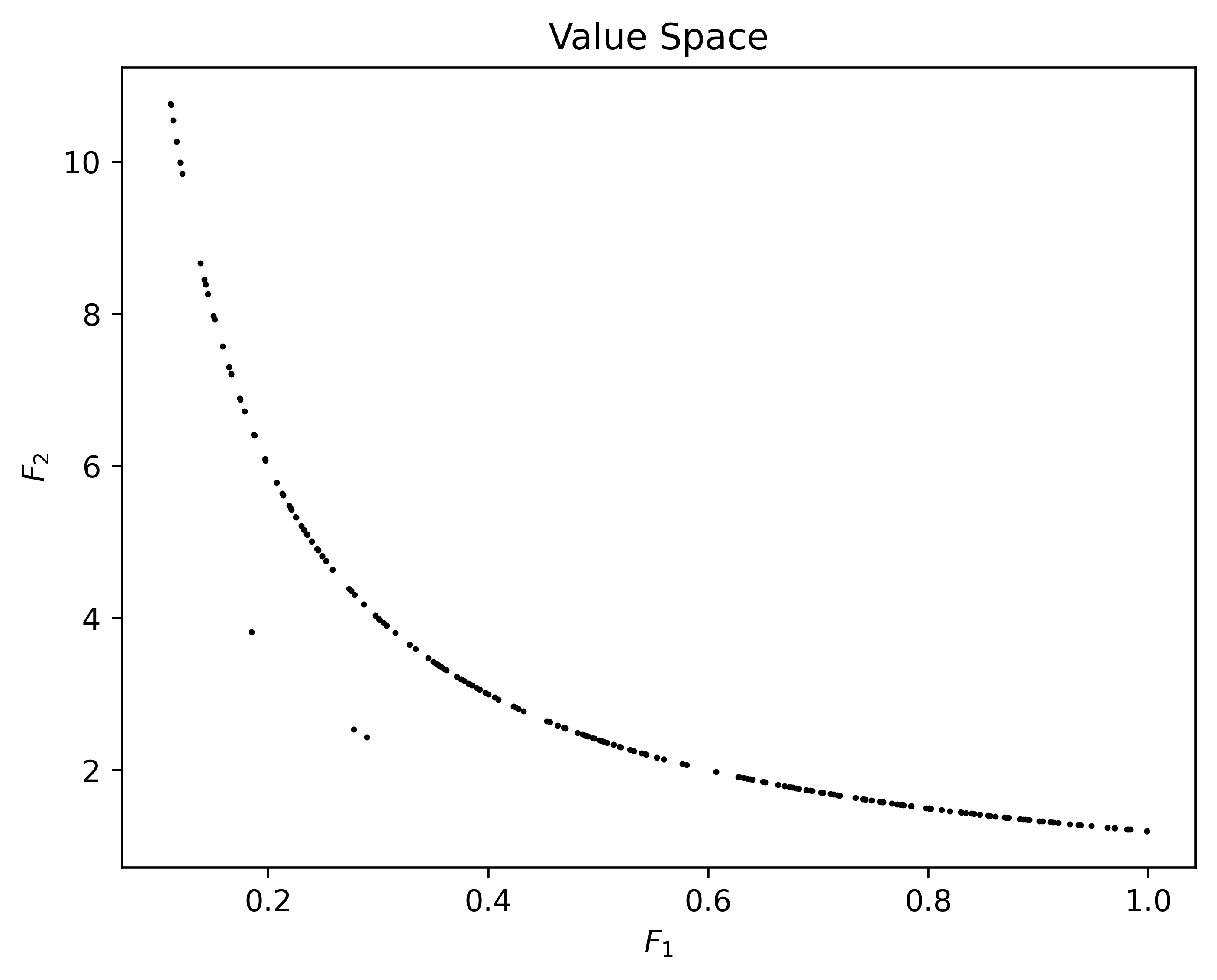}\\
			\includegraphics[scale=0.22]{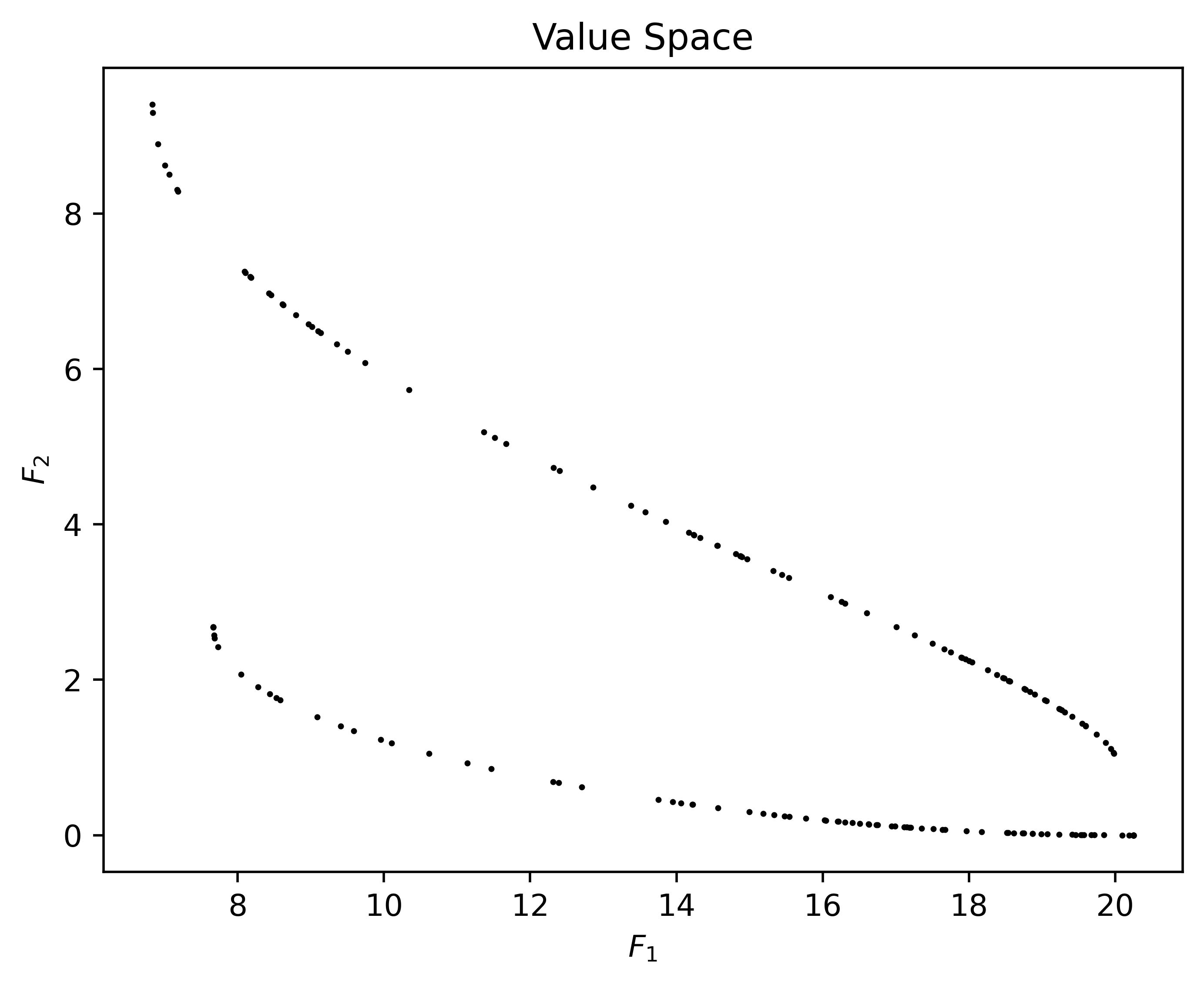}
		\end{minipage}
	}
	\subfigure[BBDMO\_VM]
	{
		\begin{minipage}[H]{.22\linewidth}
			\centering
			\includegraphics[scale=0.22]{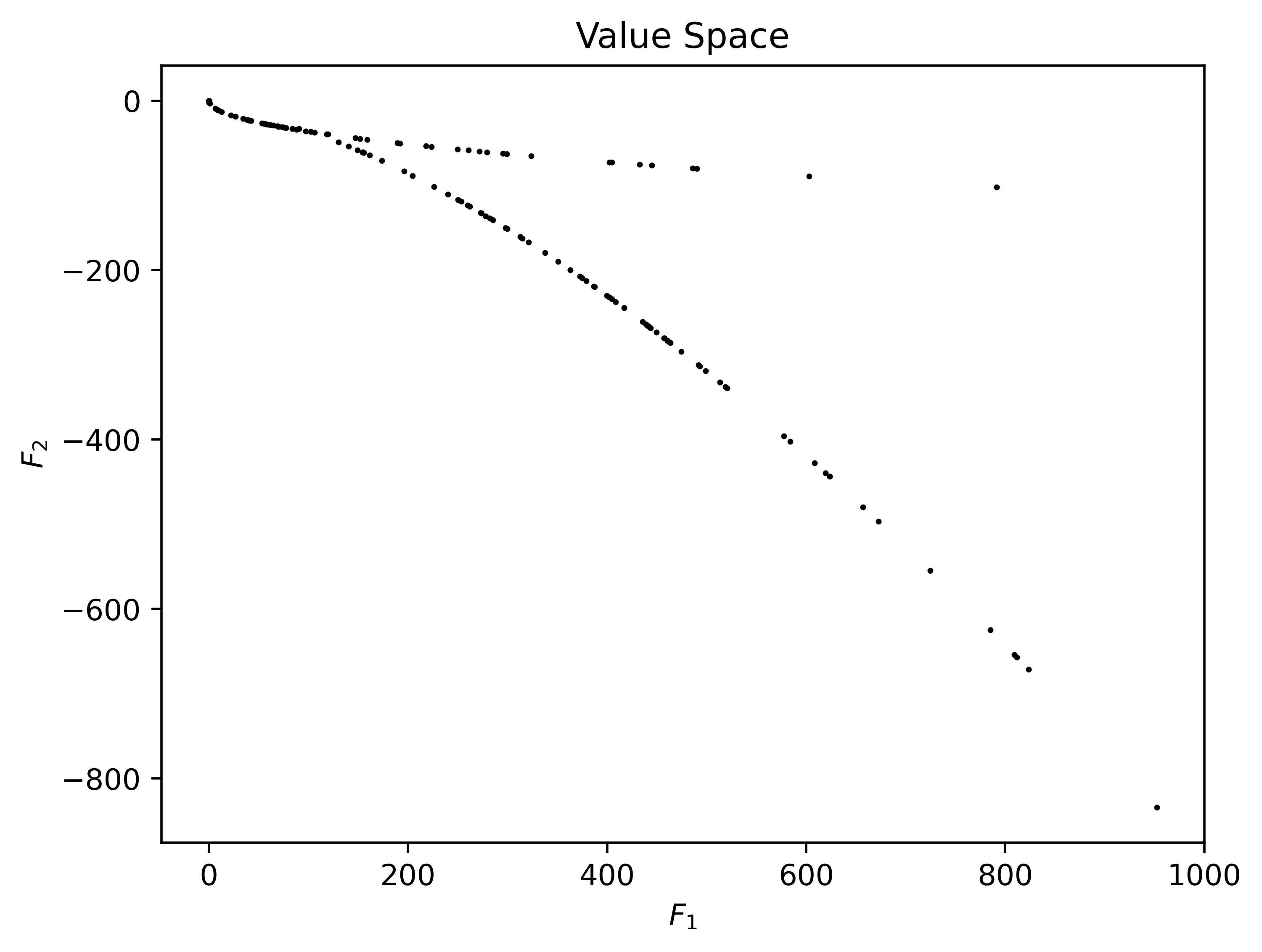} \\
			\includegraphics[scale=0.22]{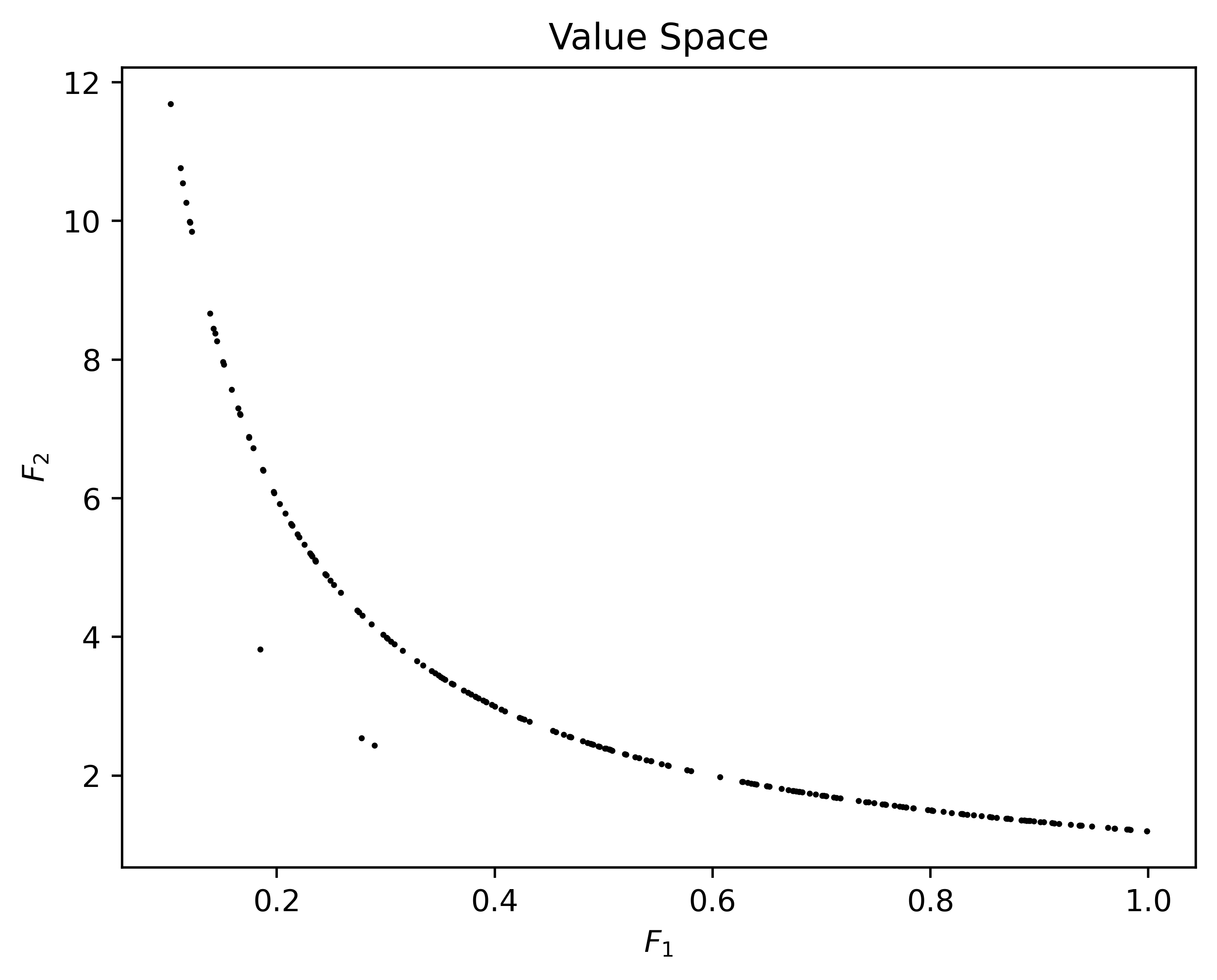}\\
			\includegraphics[scale=0.22]{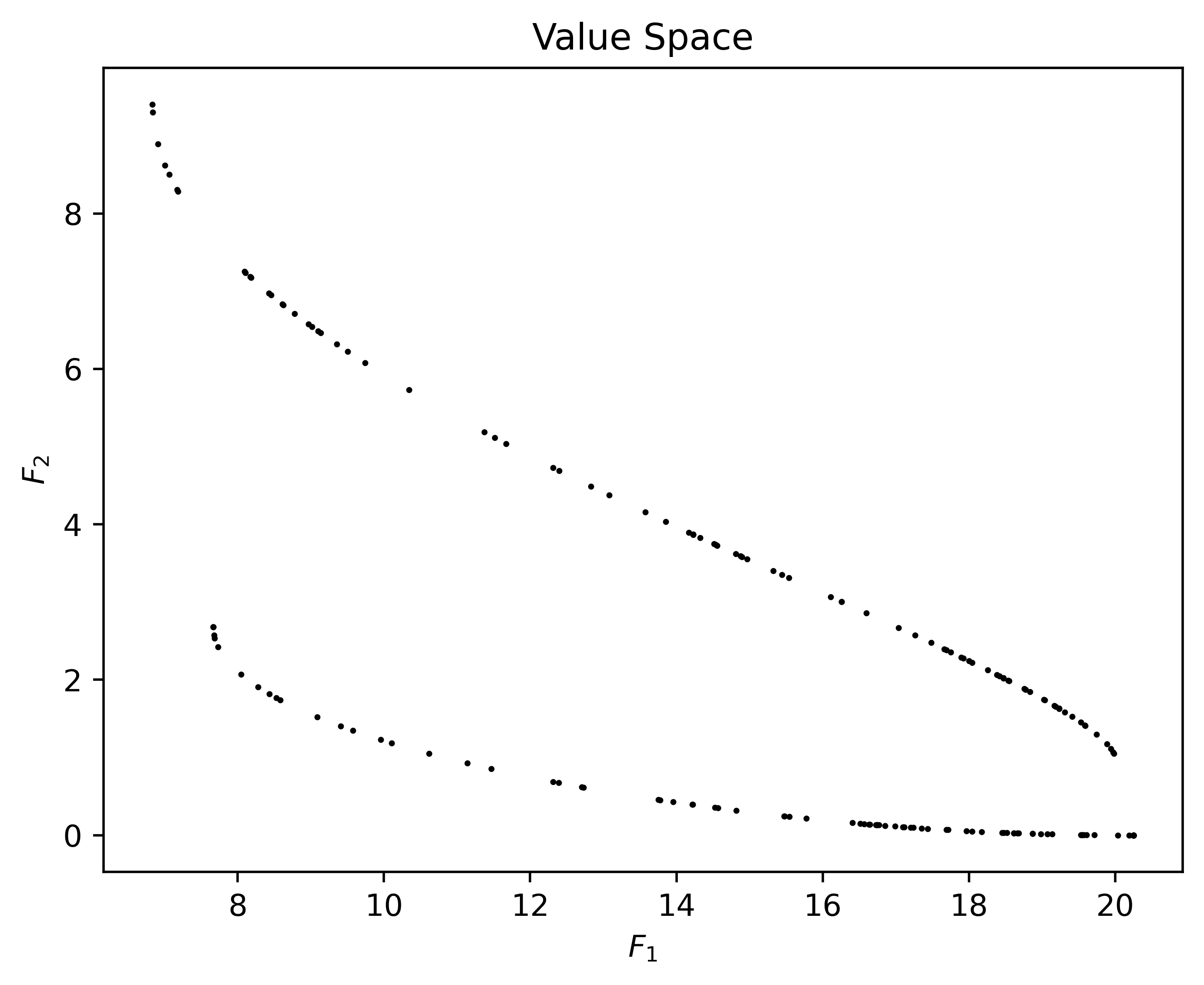}
		\end{minipage}
	}
	\caption{Numerical results in value space for problems DD1, DEB and PNR.}
	\label{f2}
\end{figure}

\begin{figure}[H]
	\centering
	\subfigure[QNMO]
	{
		\begin{minipage}[H]{.22\linewidth}
			\centering
			\includegraphics[scale=0.22]{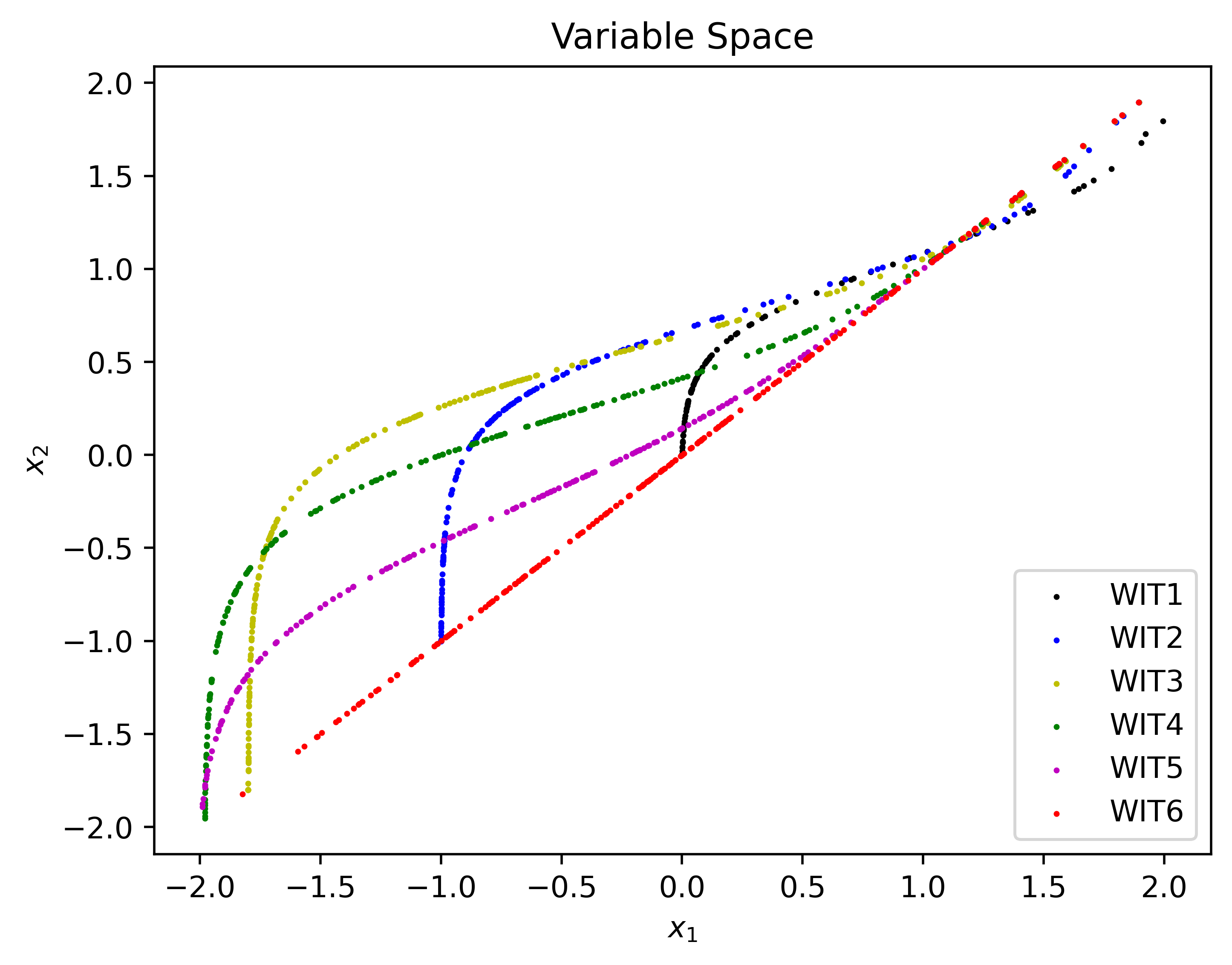} \\
			\includegraphics[scale=0.22]{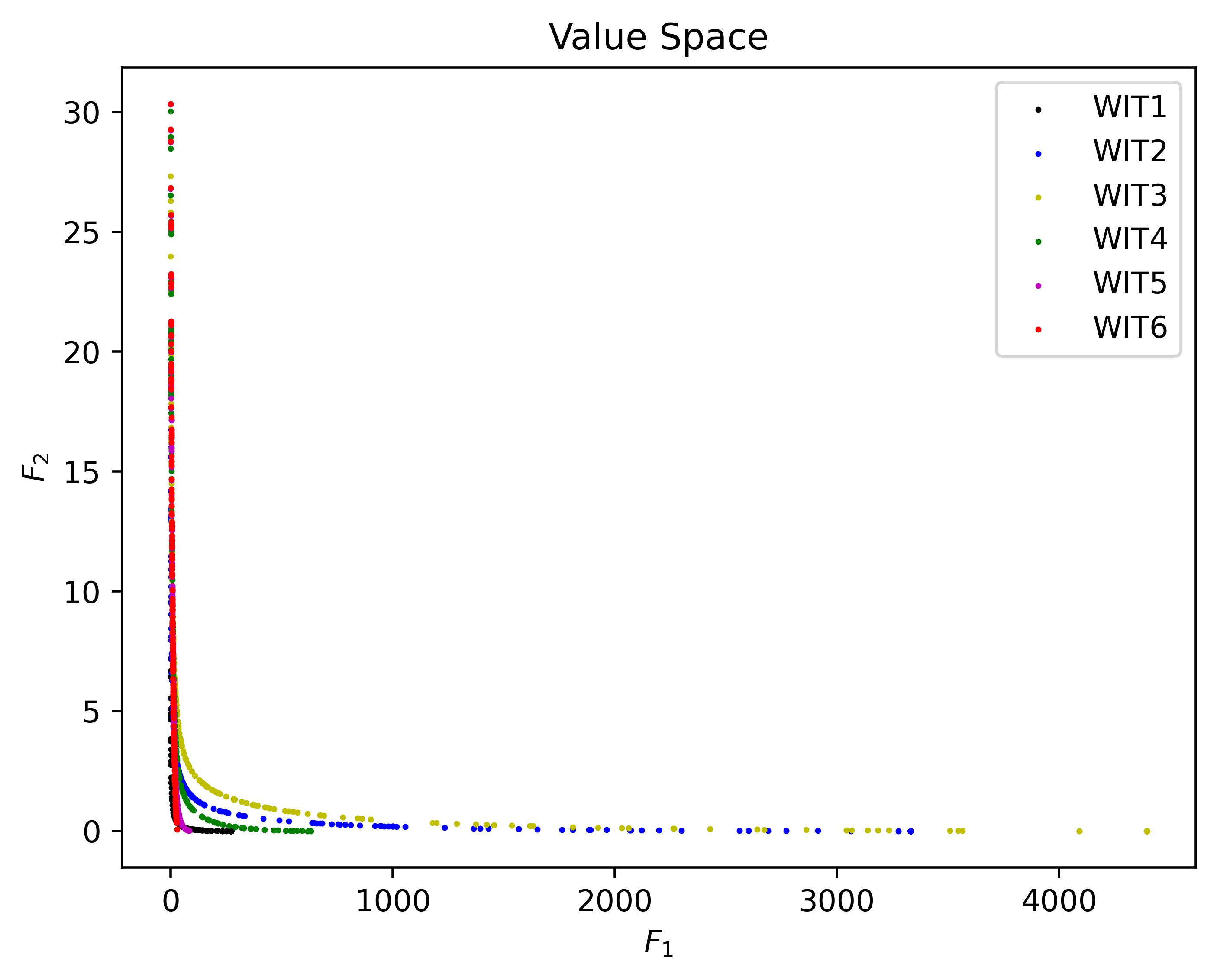} 
		\end{minipage}
	}
	\subfigure[VMMO]
	{
		\begin{minipage}[H]{.22\linewidth}
			\centering
			\includegraphics[scale=0.22]{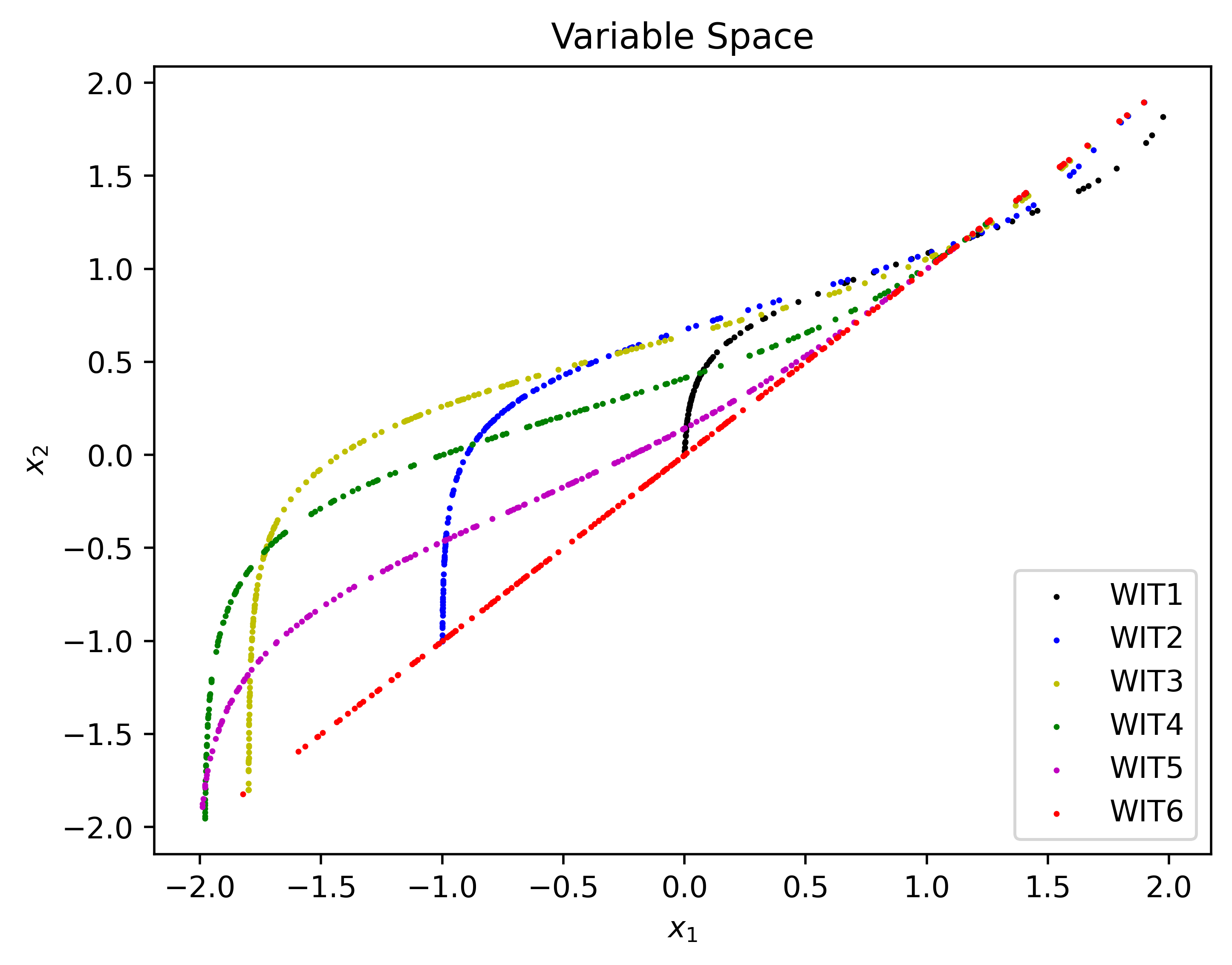} \\
			\includegraphics[scale=0.22]{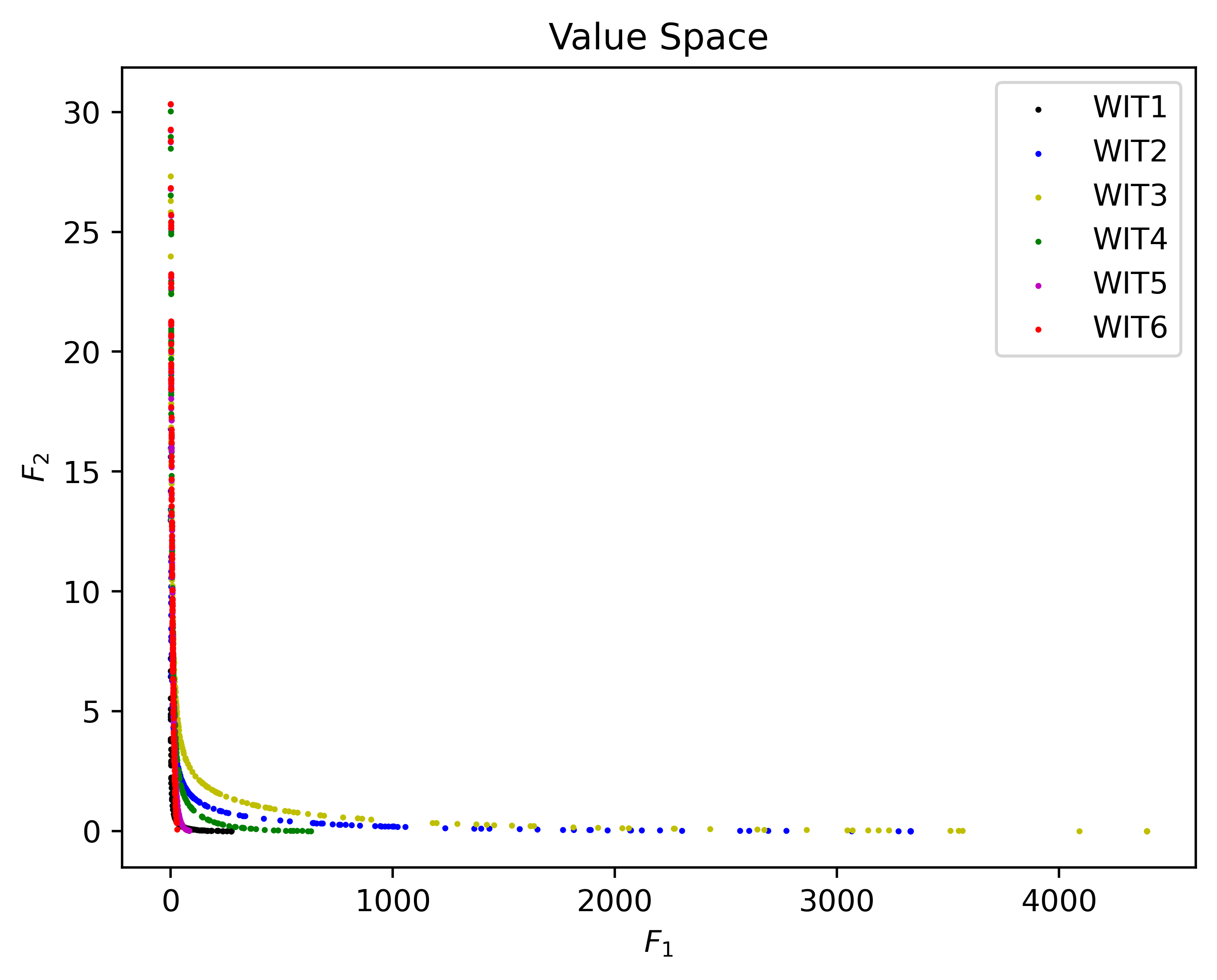} 
		\end{minipage}
	}
	\subfigure[BBDMO]
	{
		\begin{minipage}[H]{.22\linewidth}
			\centering
			\includegraphics[scale=0.22]{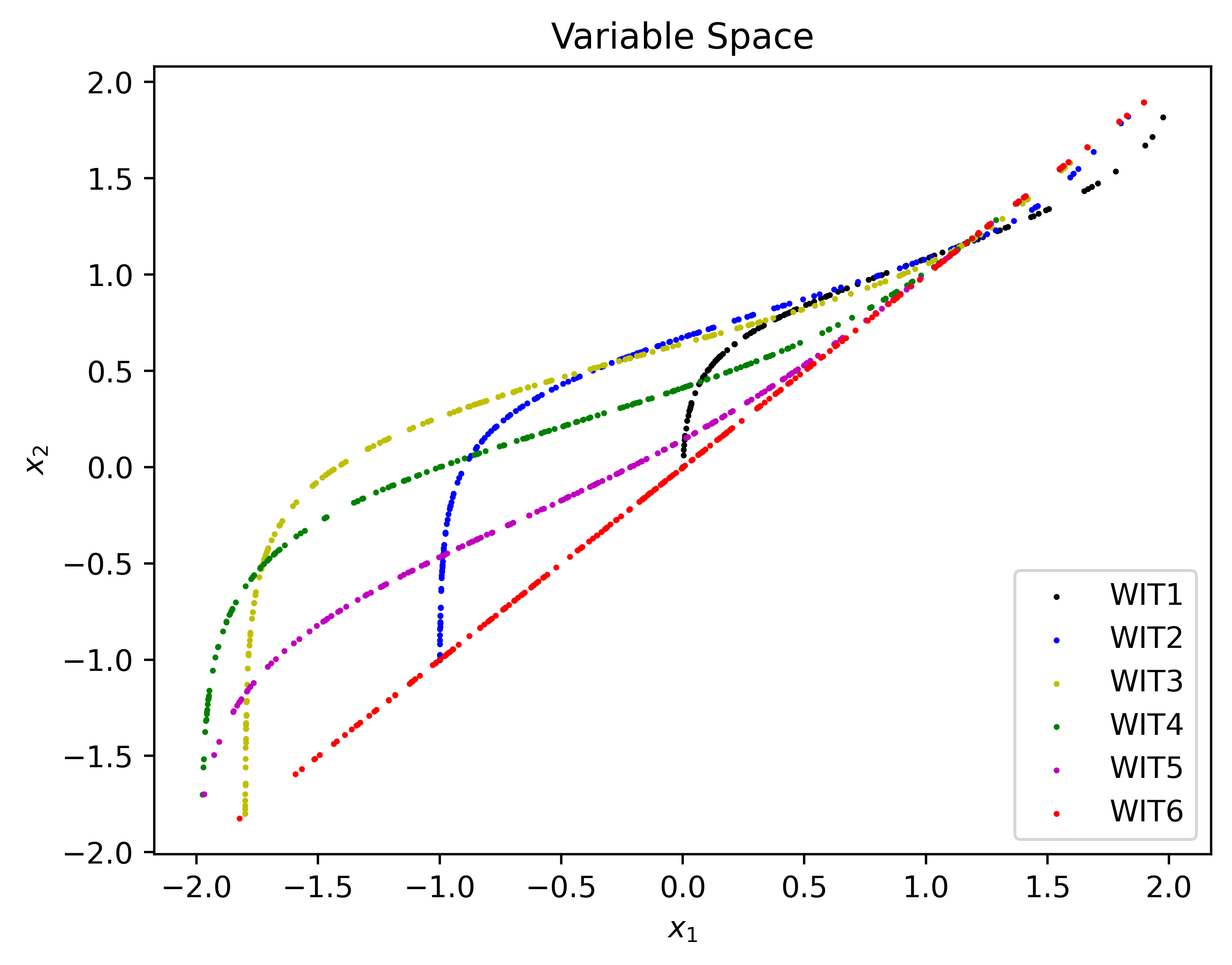} \\
			\includegraphics[scale=0.22]{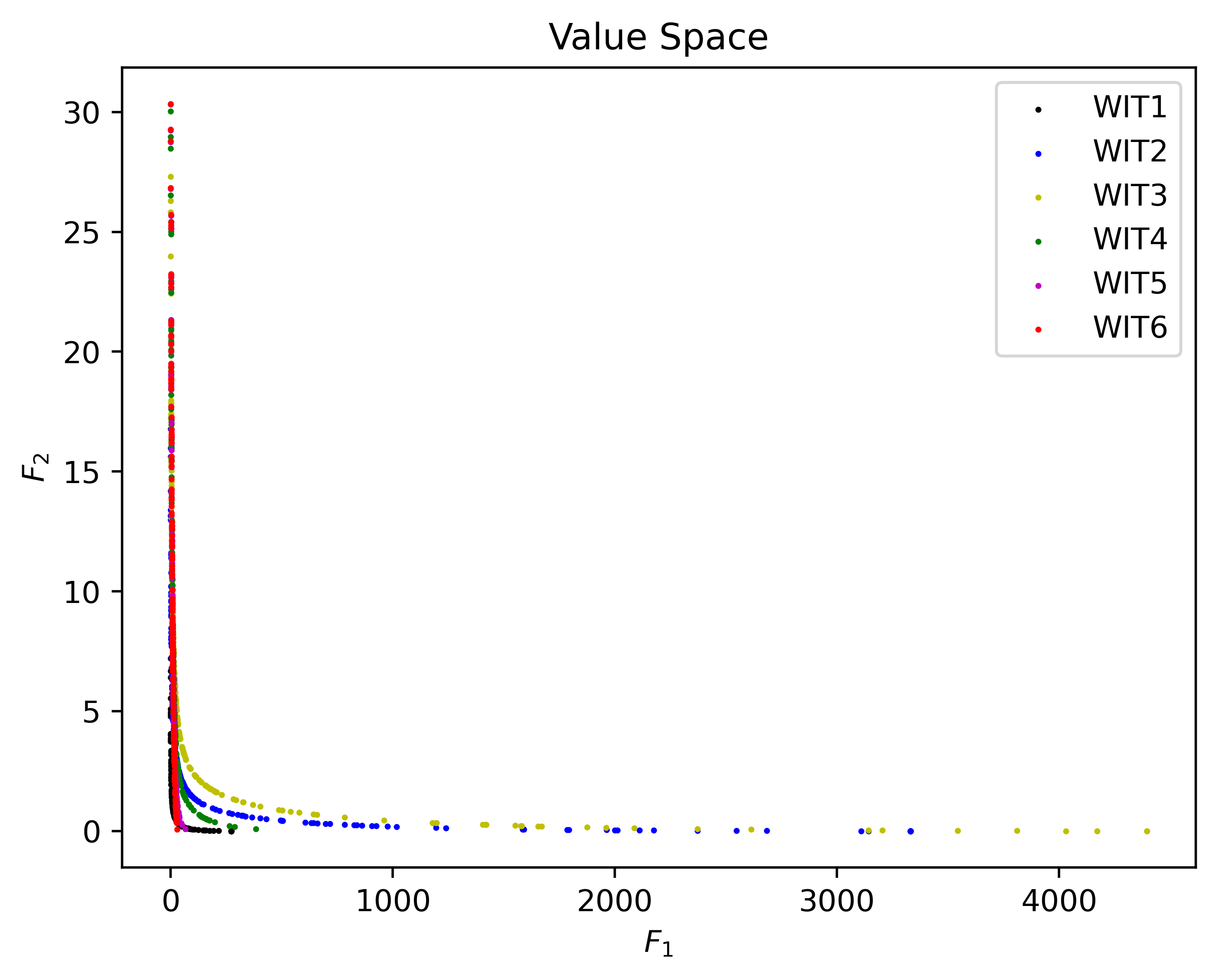} 
			
		\end{minipage}
	}
	\subfigure[BBDMO\_VM]
	{
		\begin{minipage}[H]{.22\linewidth}
			\centering
			\includegraphics[scale=0.22]{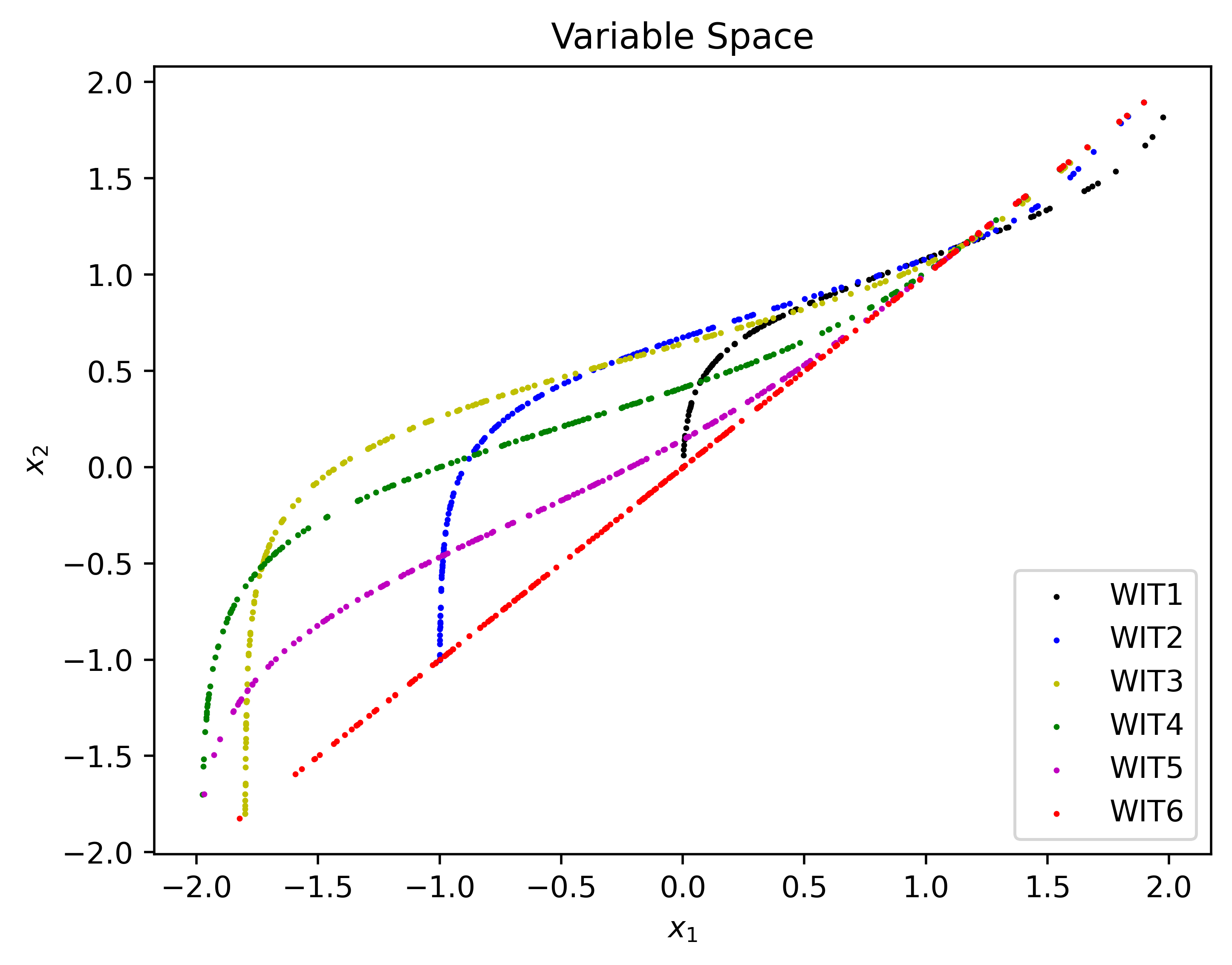}\\
			\includegraphics[scale=0.22]{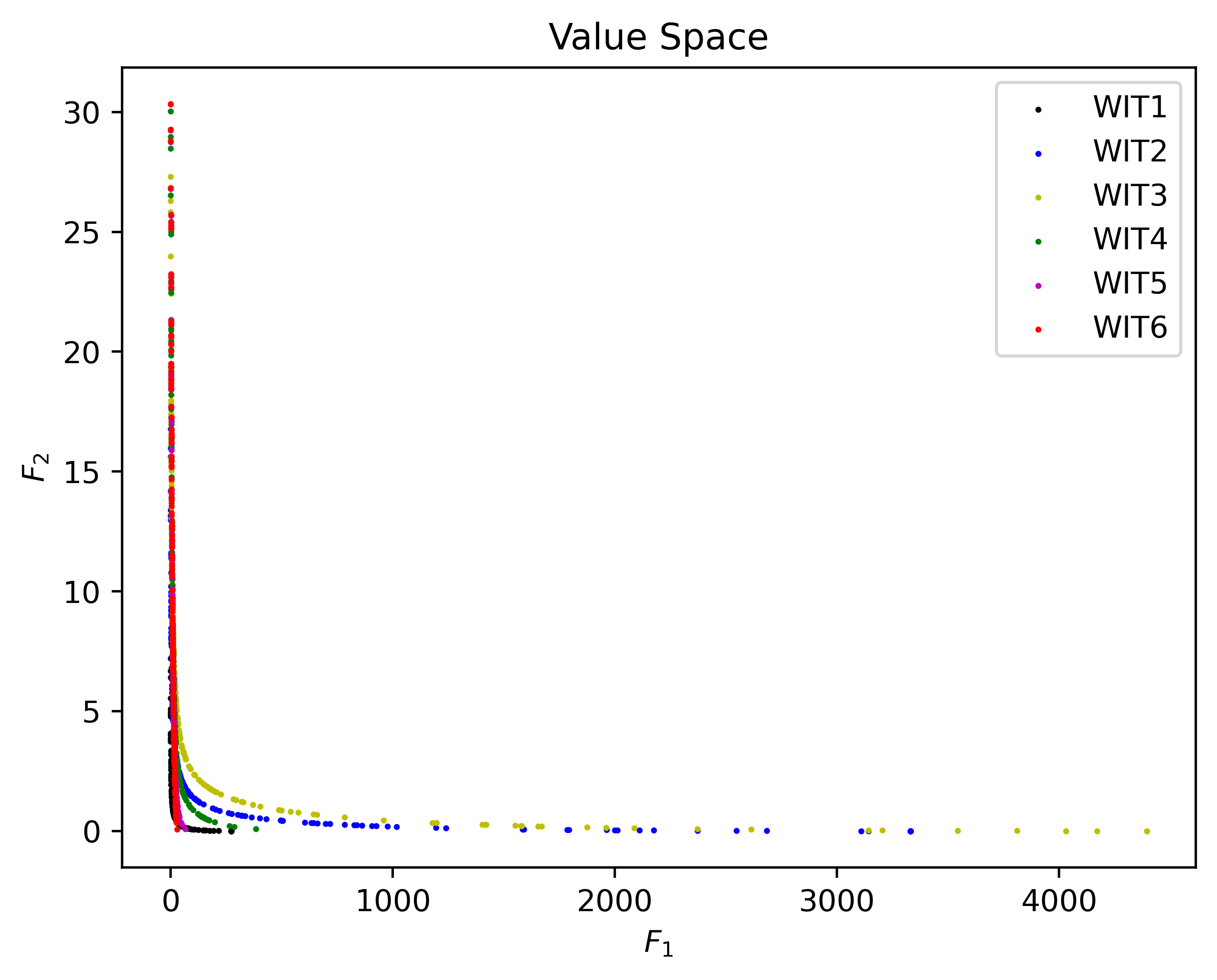} 
		\end{minipage}
	}
	\caption{Numerical results in variable space and value space obtained by the tested algorithms for problems WIT1-6.}
	\label{f0}
\end{figure}

\par For each test problem, the number of average iterations (iter), number of average function evaluations (feval), and average CPU time (time($ms$)) of the different algorithms are listed in Table \ref{tab3}. The problems DD1, Deb, FDS, Imbalance1-2, VU1 and WIT1-2 involve imbalanced objective functions, such as higher-order and exponential functions, leading to poor VMMO performance. In contrast to VMMO, the other methods perform well on these problems, demonstrating their ability to alleviate objectives' imbalances. Nevertheless, BBDMO and BBDMO\_VM require much less CUP time, particularly for high-dimensional problems, than QNMO. The BBDMO and BBDMO\_VM exhibit superior performance for the test problems due to the good conditioning.

\begin{table}[H]
	\centering
	\resizebox{.95\columnwidth}{!}{
	\begin{tabular}{lrrrrrrrrrrrrrrr}
		\hline
		Problem &
		\multicolumn{3}{l}{QNMO} &
		&
		\multicolumn{3}{l}{VMMO} &
		\multicolumn{1}{l}{} &
		\multicolumn{3}{l}{BBDMO} &
		\multicolumn{1}{l}{} &
		\multicolumn{3}{l}{BBDMO\_VM} \\ \cline{2-4} \cline{6-8} \cline{10-16} 
		&
		iter &
		feval &
		time &
		\textbf{} &
		iter &
		feval &
		time &
		&
		iter &
		feval &
		time &
		&
		iter &
		feval &
		time \\ \hline
		BK1        & 1.02 & 2.01 & 3.52 &  & \textbf{1.00} & 2.00 & \textbf{0.31} &  & \textbf{1.00} & \textbf{1.00} & 0.39 &  & \textbf{1.00} & \textbf{1.00} & 0.47 \\
		DD1        & 23.76 & 24.42 & 56.83 &  & 47.62 & 179.97 & 15.72 &  & \textbf{7.49} & \textbf{8.76} & \textbf{1.73} &  & 14.54 & 23.93 & 5.48 \\
		Deb        & 6.20 & 11.10 & 8.97 &  & 65.68 & 370.07 & 26.97 &  & 4.41 & 6.58 & \textbf{1.17} &  & \textbf{4.26} & \textbf{4.69} & 2.05 \\
		Far1       & 35.45 & 37.97 & 41.04 &  & 50.84 & 209.39 & 29.52 &  & 85.16 & 85.64 & 19.23 &  & \textbf{17.12} & \textbf{23.03} & \textbf{7.55} \\
		FDS        & 10.78 & 15.05 & 50.64 &  & 164.97 & 1110.91 & 341.85 &  & \textbf{4.57} & \textbf{5.20} & \textbf{4.78} &  & 4.89 & 5.39 & 6.44 \\
		FF1        & 7.35 & 7.36 & 9.90 &  & 18.03 & 62.27 & 5.43 &  & 4.91 & 6.13 & \textbf{1.18} &  & \textbf{4.86} & \textbf{5.82} & 2.04 \\
		Hil1       & 10.25 & 12.90 & 17.05 &  & 15.91 & 61.41 & 6.47 &  & 11.32 & 12.15 & \textbf{2.83} &  & \textbf{7.99} & \textbf{8.68} & 3.23 \\
		Imbalance1 & \textbf{2.51} & 5.14 & 5.98 &  & 54.76 & 229.67 & 16.24 &  & 2.61 & 3.54 & \textbf{0.71} &  & {2.55} & \textbf{3.28} & 1.17 \\
		Imbalance2 & 1.51 & 5.38 & 4.32 &  & 227.67 & 1595.90 & 75.57 &  & \textbf{1.00} & \textbf{1.00} & \textbf{0.42} &  & \textbf{1.00} & \textbf{1.00} & 0.47 \\
		JOS1a      & 2.00 & 2.00 & 25.00 &  & 2.00 & 2.00 & 0.70 &  & \textbf{1.00} & \textbf{1.00} & \textbf{0.47} &  & \textbf{1.00} & \textbf{1.00} & 0.55 \\
		JOS1b      & 2.00 & 2.00 & 50.19 &  & 2.00 & 2.00 & 1.02 &  & \textbf{1.00} & \textbf{1.00} & \textbf{0.48} &  & \textbf{1.00} & \textbf{1.00} & 0.62 \\
		JOS1c      & 2.00 & 2.00 & 83.53 &  & 2.00 & 2.00 & 1.02 &  & \textbf{1.00} & \textbf{1.00} & \textbf{0.47} &  & \textbf{1.00} & \textbf{1.00} & 0.63 \\
		JOS1d      & 2.00 & 2.00 & 124.60 &  & 2.00 & 2.00 & 1.10 &  & \textbf{1.00} & \textbf{1.00} & \textbf{0.47} &  & \textbf{1.00} & \textbf{1.00} & 0.63 \\
		LE1        & 9.90 & 42.15 & 14.62 &  & 10.41 & 25.97 & 2.97 &  & 4.55 & 7.03 & \textbf{1.22} &  & \textbf{4.52} & \textbf{6.53} & 1.88 \\
		PNR        & \textbf{2.59} & 4.61 & 5.48 &  & 7.81 & 19.61 & 2.36 &  & {4.18} & 4.74 & \textbf{1.02} &  & 4.23 & \textbf{4.57} & 1.74 \\
		VU1        & 64.85 & 66.37 & 71.74 &  & 144.66 & 817.75 & 46.72 &  & 13.99 & 14.04 & \textbf{3.60} &  & \textbf{11.85} & \textbf{12.44} & 4.86 \\
		WIT1       & \textbf{2.60} & 5.47 & 5.77 &  & 50.76 & 263.04 & 19.40 &  & 3.53 & 3.62 & \textbf{0.90} &  & {3.44} & \textbf{3.53} & 1.42 \\
		WIT2       & \textbf{3.87} & 8.49 & 8.11 &  & 76.28 & 382.81 & 29.61 &  & 3.98 & 4.08 & \textbf{1.13} &  & {3.89} & \textbf{4.00} & 1.57 \\
		WIT3       & \textbf{3.89} & 7.07 & 8.06 &  & 35.27 & 130.63 & 10.97 &  & 5.09 & 5.18 & \textbf{1.28} &  & {4.96} & \textbf{5.04} & 2.03 \\
		WIT4       & \textbf{3.03} & \textbf{4.31} & 5.75 &  & 7.24 & 13.20 & 2.04 &  & 5.26 & 5.31 & \textbf{1.33} &  & {5.16} & {5.21} & 2.11 \\
		WIT5       & \textbf{2.99} & \textbf{4.01} & 5.60 &  & 5.37 & 8.57 & 1.33 &  & 4.37 & 4.39 & \textbf{1.18} &  & {4.34} & {4.36} & 1.78 \\
		WIT6       & 1.06 & 2.00 & 2.90 &  & \textbf{1.00} & 2.00 & \textbf{0.34} &  & \textbf{1.00} & \textbf{1.00} & 0.39 &  & \textbf{1.00} & \textbf{1.00} & 0.54 \\ \hline
	\end{tabular}
}\caption{Number of average iterations (iter), number of average function evaluations (feval), and average CPU time (time($ms$)) of QNMO, VMMO, BBDMO, and BBDMO\_VM implemented on different test problems.}\label{tab3}
\end{table}

\subsection{Quadratic ill-conditioned problems}
In this subsection, we test the algorithm on ill-conditioned problems. We consider a series of quadratic problems defined as follows:
$$F_{i}(x)=\frac{1}{2}\left\langle x,A_{i}x\right\rangle + \left\langle b_{i},x\right\rangle,~i=1,2,$$
where $A_{i}$ is a positive definite matrix. We set $A_{i}=H_{i}D_{i}H_{i}^{T}$, where $H_{i}$ is a random
orthogonal matrix and $D_{i}=Diag(d^{1}_{i},d^{2}_{i},...,d^{n}_{i})$ with $\max_{j}d^{j}_{i}/\min_{j}d^{j}_{i}=\kappa_{i}$. The problem illustration is given in Table \ref{tab2}. The second and third columns present the objective functions' dimension and condition numbers, respectively. While $x_L$ and $x_U$ represent the lower and upper bounds of the variables, respectively. 
\begin{table}[h]
	\centering
	\resizebox{.65\columnwidth}{!}{
		\begin{tabular}{llllll}
			\hline
			Problem   & $n$   & $(\kappa_{1},\kappa_{2})$ & $x_{L}$                & $x_{U}$               \\ \hline
			QPa & 10   & $(10,10)$ & 10{[}-1,...,-1{]}       & 10{[}1,...,1{]}            \\
			QPb & 10   & $(10^{2},10^{2})$ & 10{[}-1,...,-1{]}      & 10{[}1,...,1{]}           \\
			QPc & 100  & $(10^{2},10^{2})$ & 100{[}-1,...,-1{]}    & 100{[}1,...,1{]}        \\
			QPd & 100 & $(10^{3},10^{3})$ & 100{[}-1,...,-1{]}    & 100{[}1,...,1{]}         \\
			QPe & 500 & $(10^{3},10^{3})$ & 500{[}-1,...,-1{]}  & 500{[}1,...,1{]}    \\
			QPf & 500 & $(10^{4},10^{4})$ & 500{[}-1,...,-1{]} & 500{[}1,...,1{]}    \\
			QPg & 100 & $(10^{5},10^{2})$ & 100{[}-1,...,-1{]} & 100{[}1,...,1{]}    \\ \hline
		\end{tabular}	
	}
	\caption{Description of quadratic problems.}
	\label{tab2}
\end{table}

\begin{figure}[H]
	\centering
	\subfigure[QNMO]
	{
		\begin{minipage}[H]{.3\linewidth}
			\centering
			\includegraphics[scale=0.3]{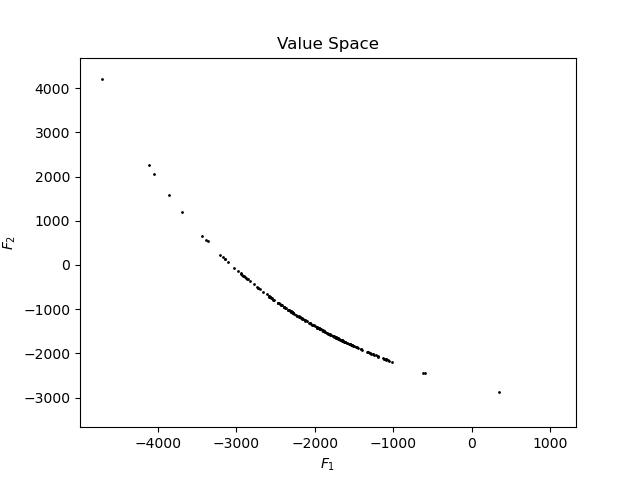} 
		\end{minipage}
	}
	\subfigure[VMMO]
	{
		\begin{minipage}[H]{.3\linewidth}
			\centering
			\includegraphics[scale=0.3]{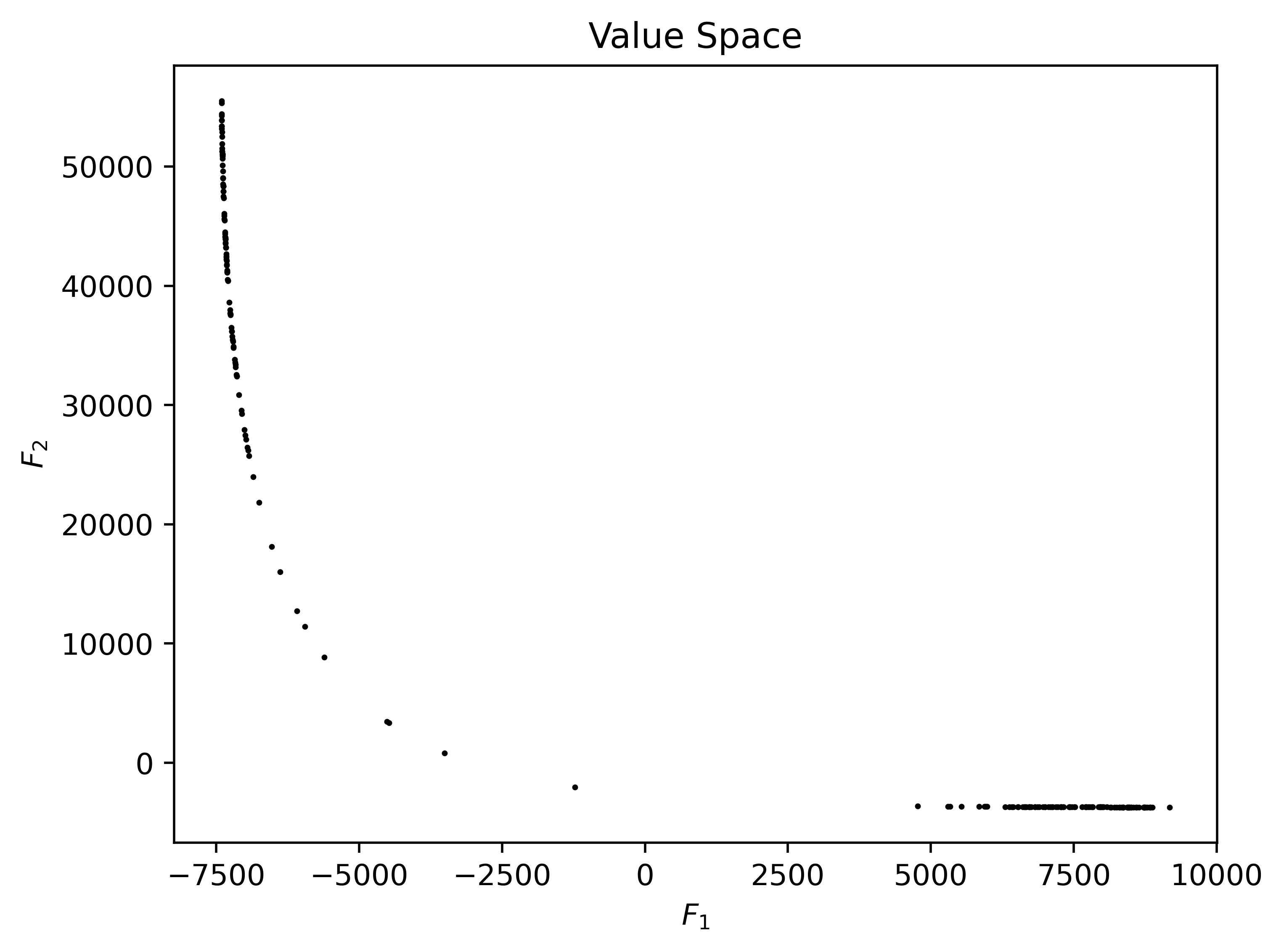}
		\end{minipage}
	}\\
	\subfigure[BBDMO]
	{
		\begin{minipage}[H]{.3\linewidth}
			\centering
			\includegraphics[scale=0.3]{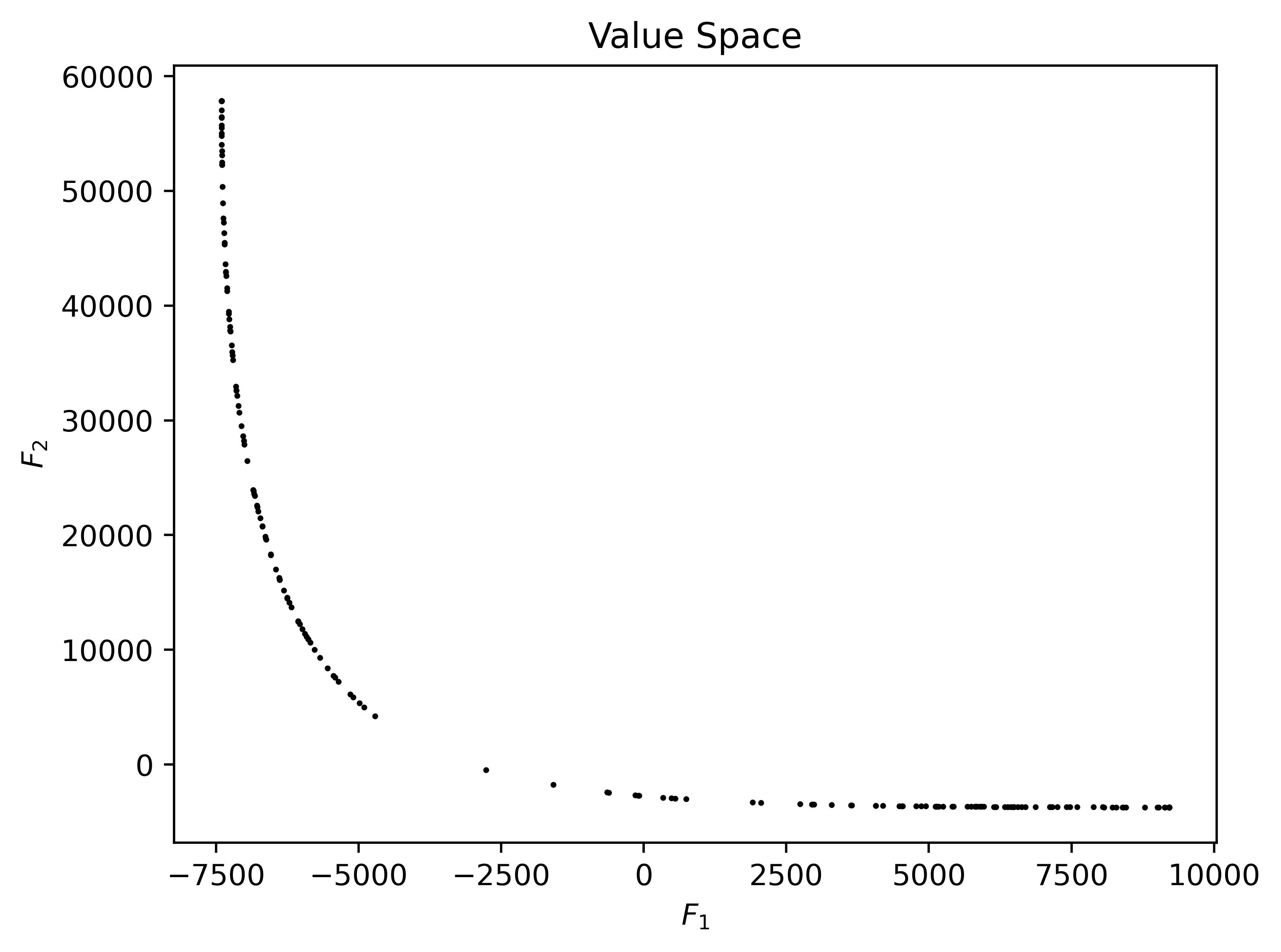} 
		\end{minipage}
	}
	\subfigure[BBDMO\_VM]
	{
		\begin{minipage}[H]{.3\linewidth}
			\centering
			\includegraphics[scale=0.3]{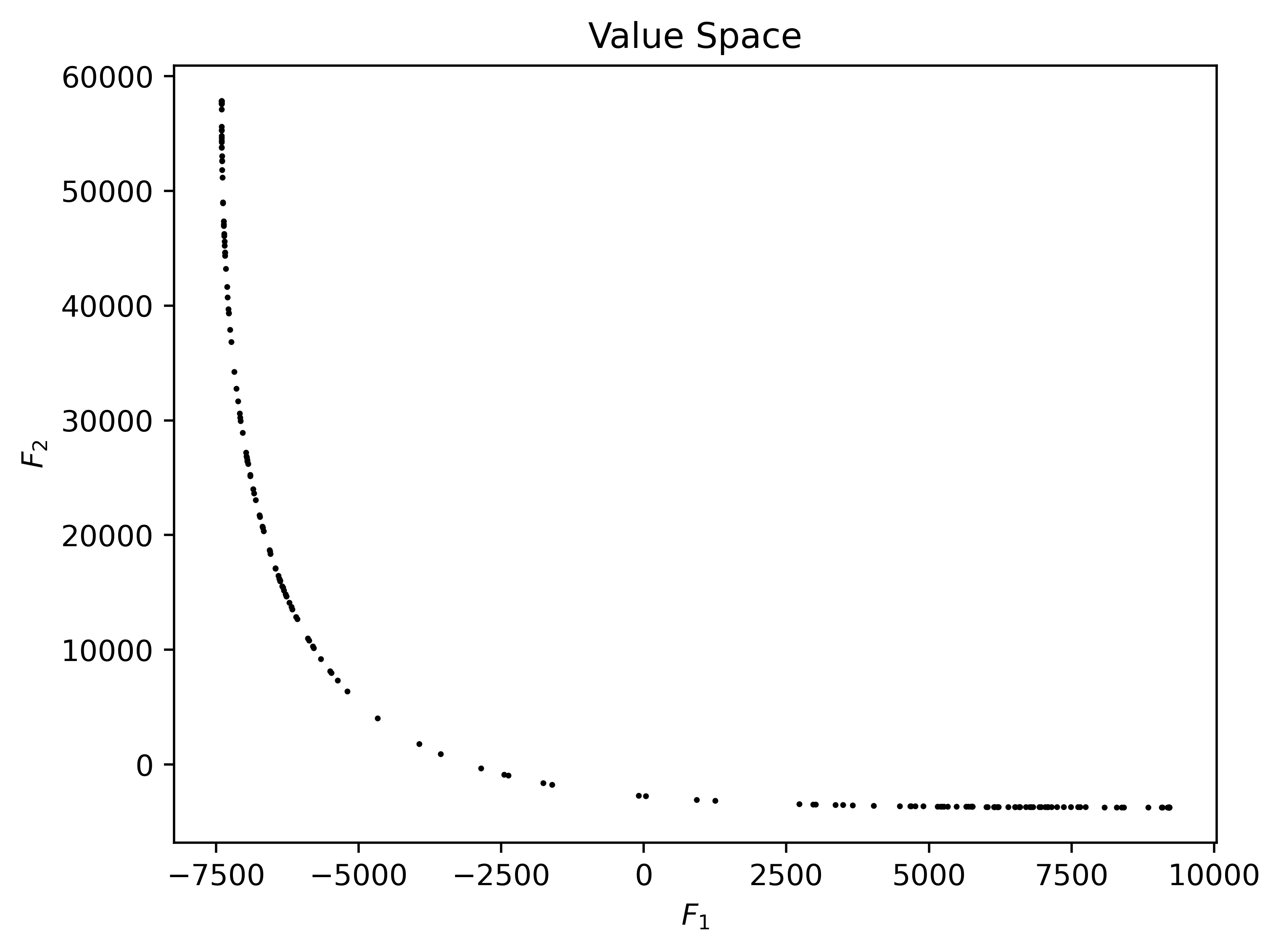} 
		\end{minipage}
	}
	\caption{Numerical results in value space for problem \textbf{QPc}.}
	\label{f3}
\end{figure}

\begin{figure}[H]
	\centering
	\subfigure[QPd]
	{
		\begin{minipage}[H]{.22\linewidth}
			\centering
			\includegraphics[scale=0.22]{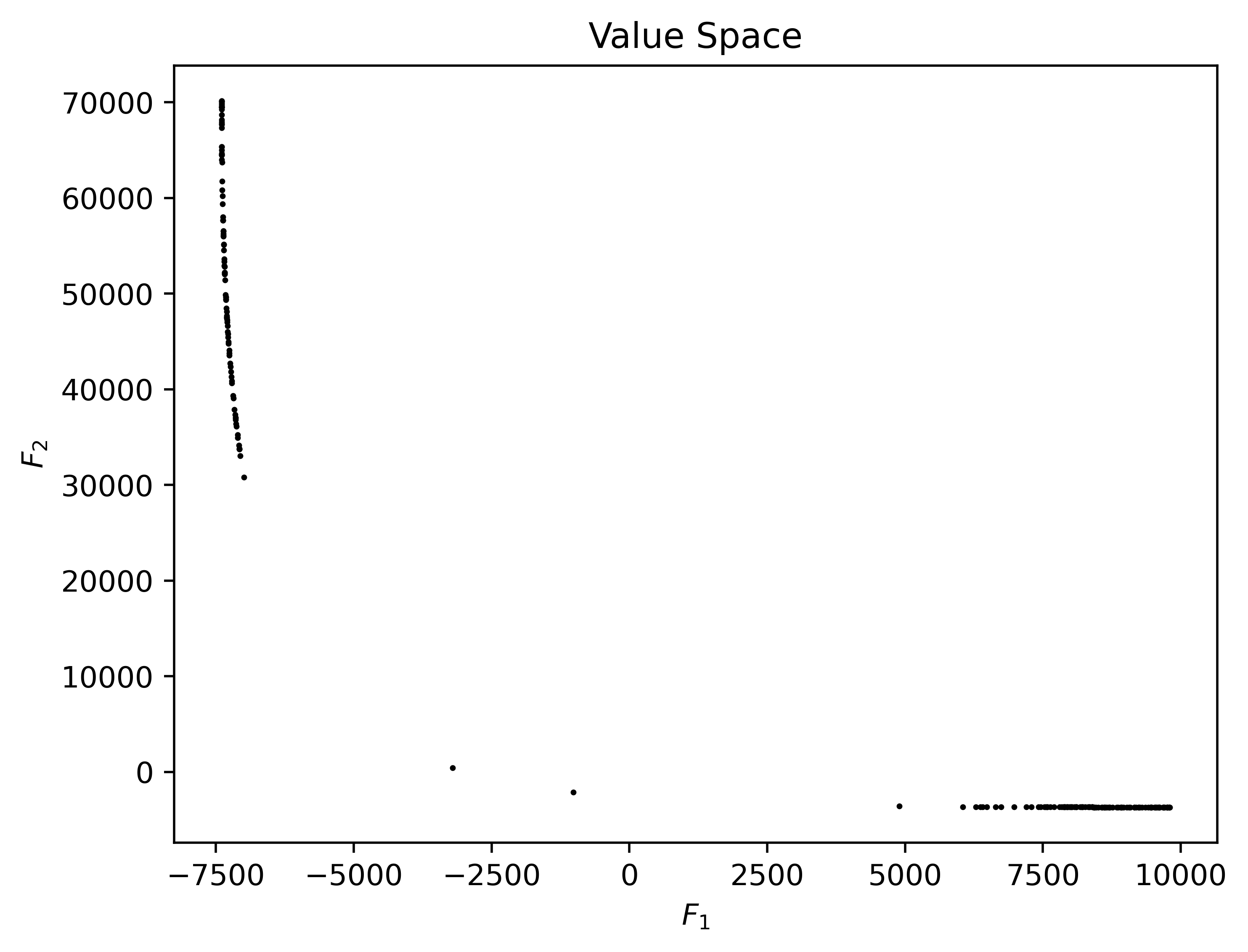} \\
			\includegraphics[scale=0.22]{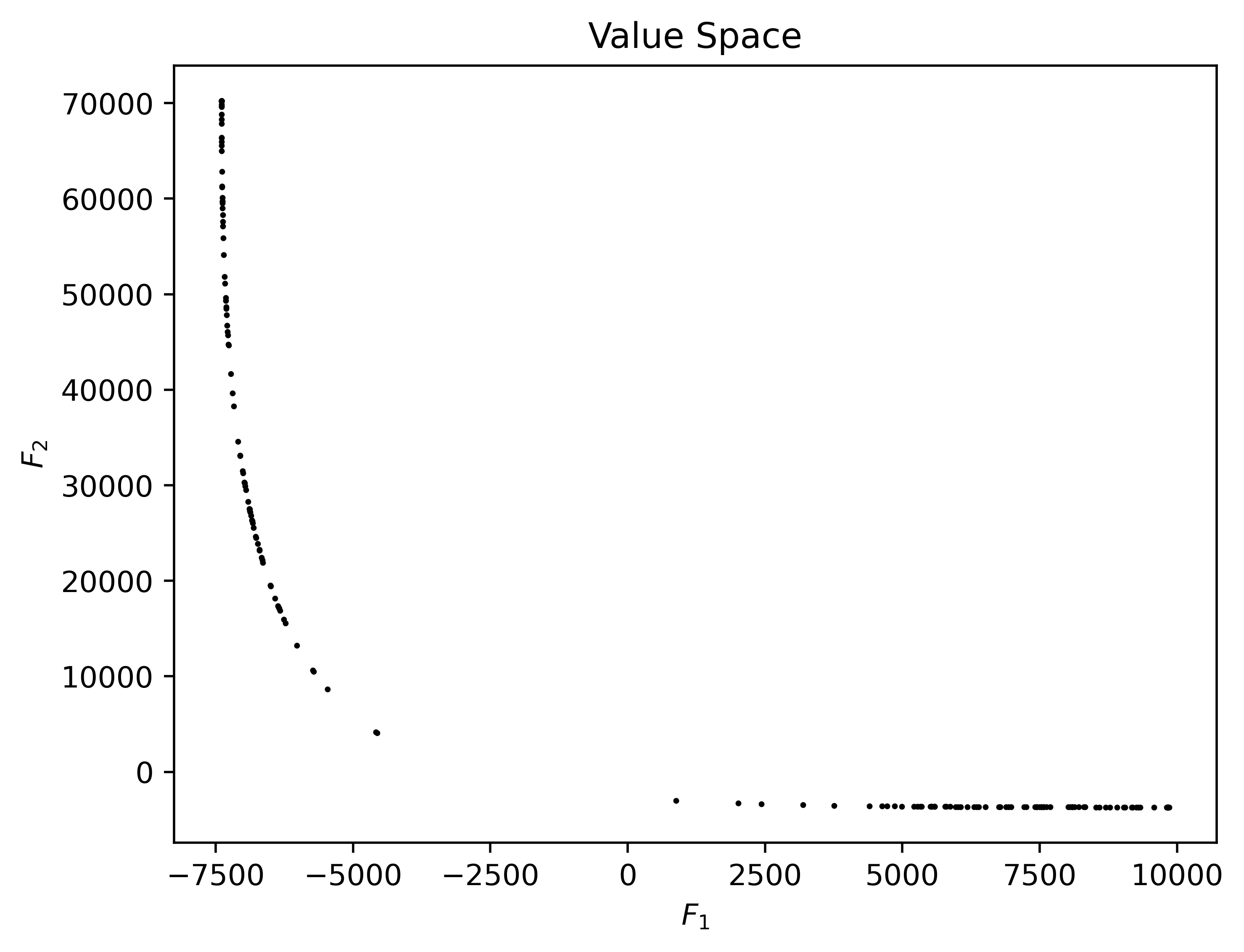} \\
			\includegraphics[scale=0.22]{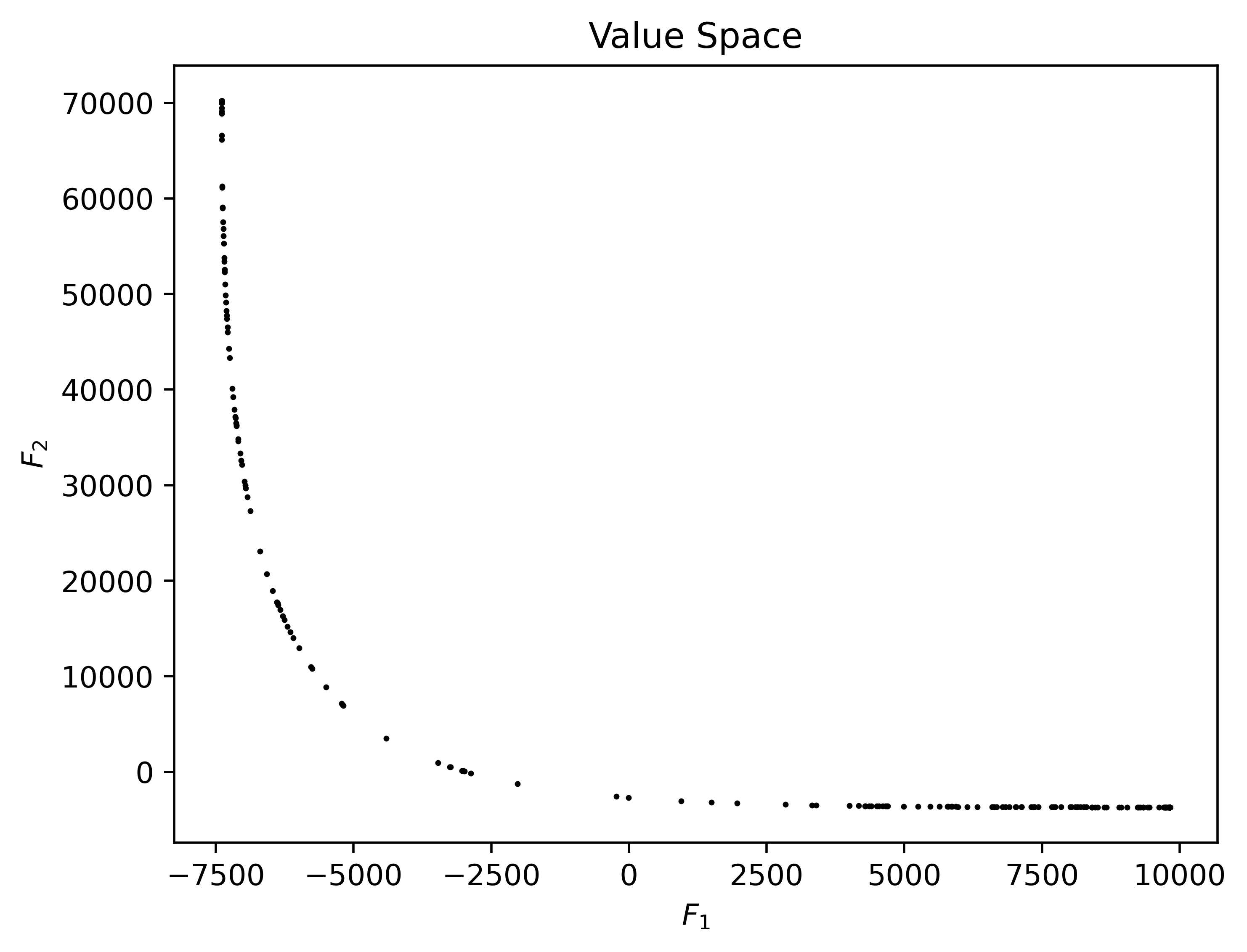}
		\end{minipage}
	}
	\subfigure[QPe]
	{
		\begin{minipage}[H]{.22\linewidth}
			\centering
			\includegraphics[scale=0.22]{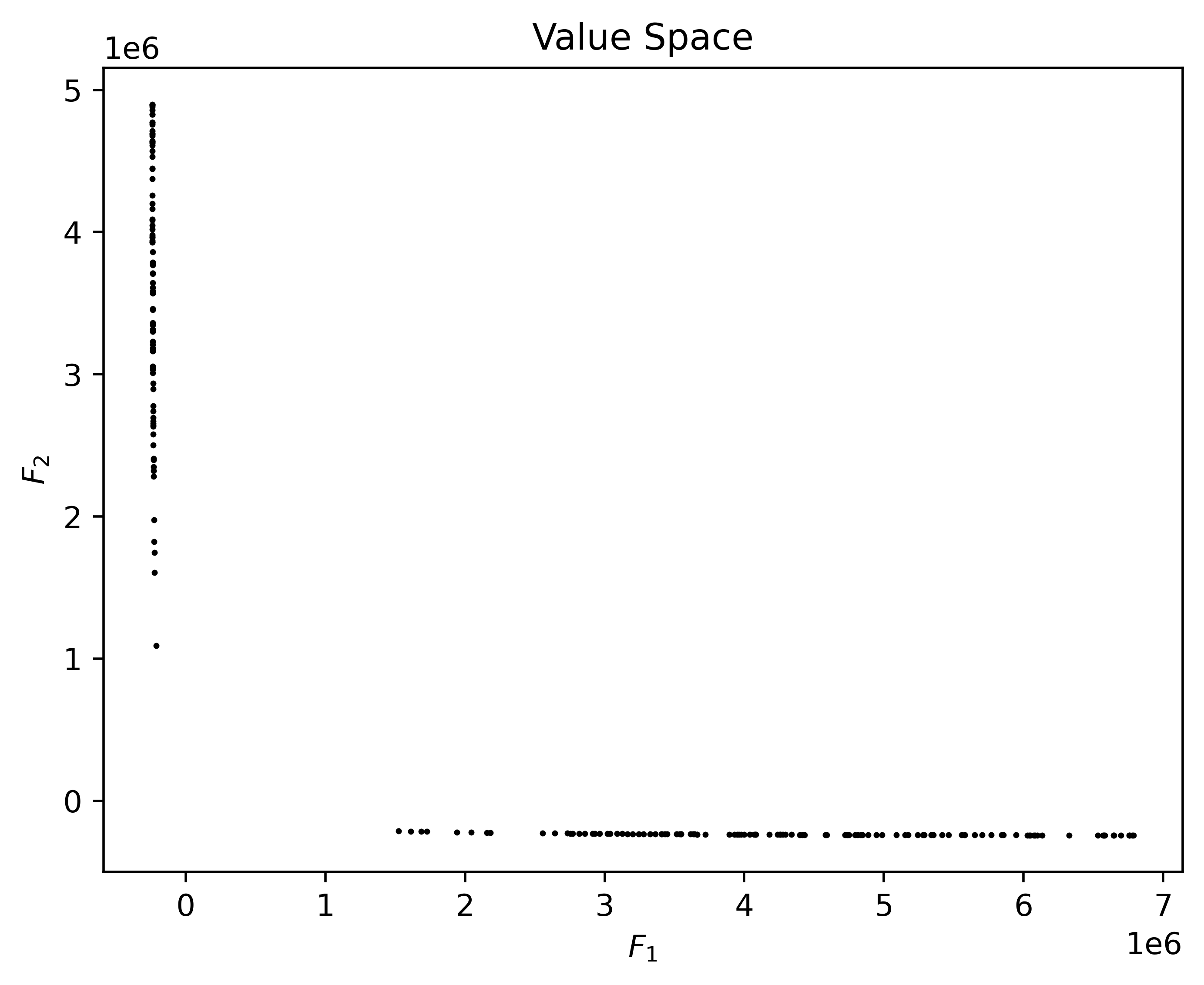} \\
			\includegraphics[scale=0.22]{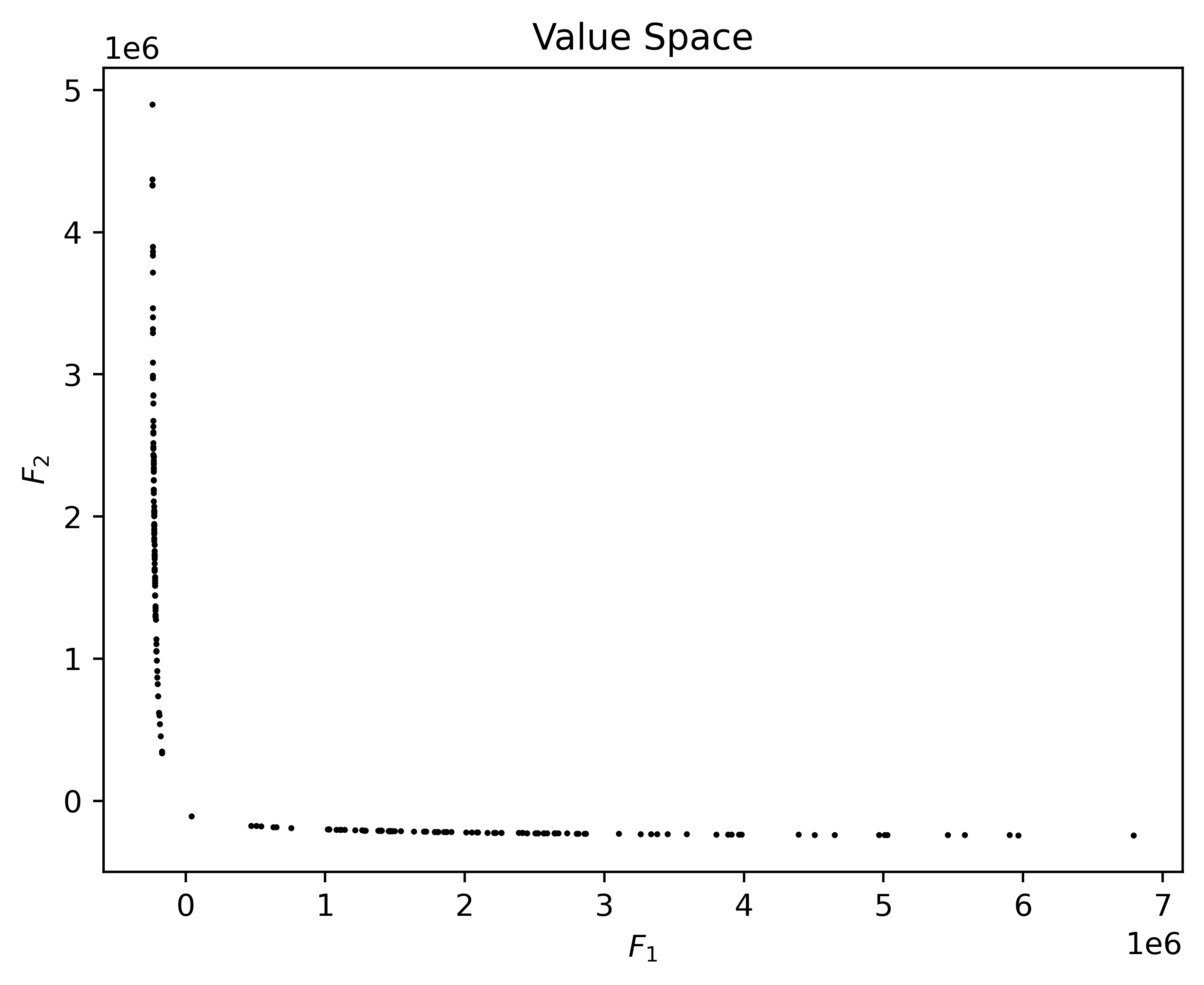} \\
			\includegraphics[scale=0.22]{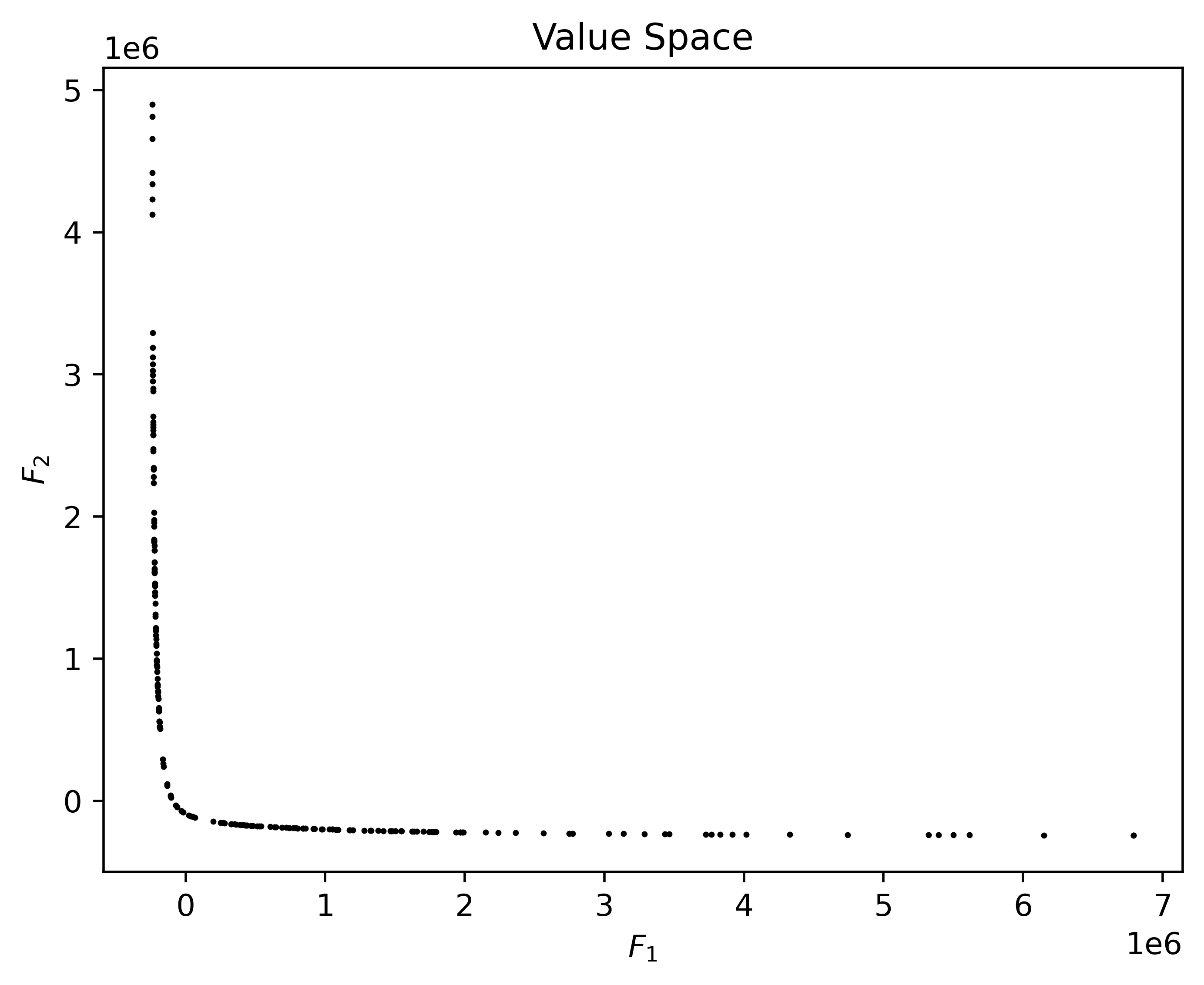}
		\end{minipage}
	}
	\subfigure[QPf]
	{
		\begin{minipage}[H]{.22\linewidth}
			\centering
			\includegraphics[scale=0.22]{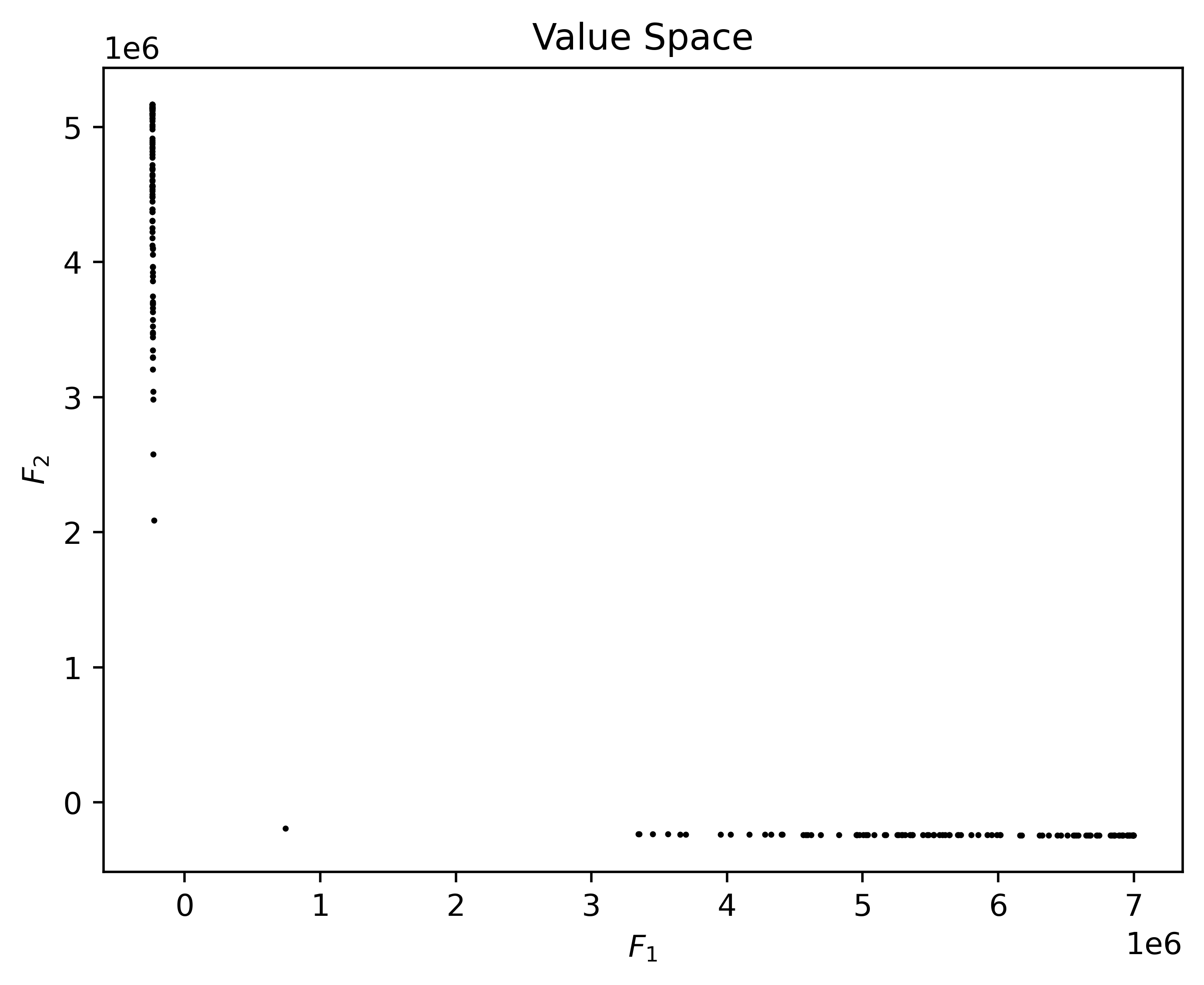} \\
			\includegraphics[scale=0.22]{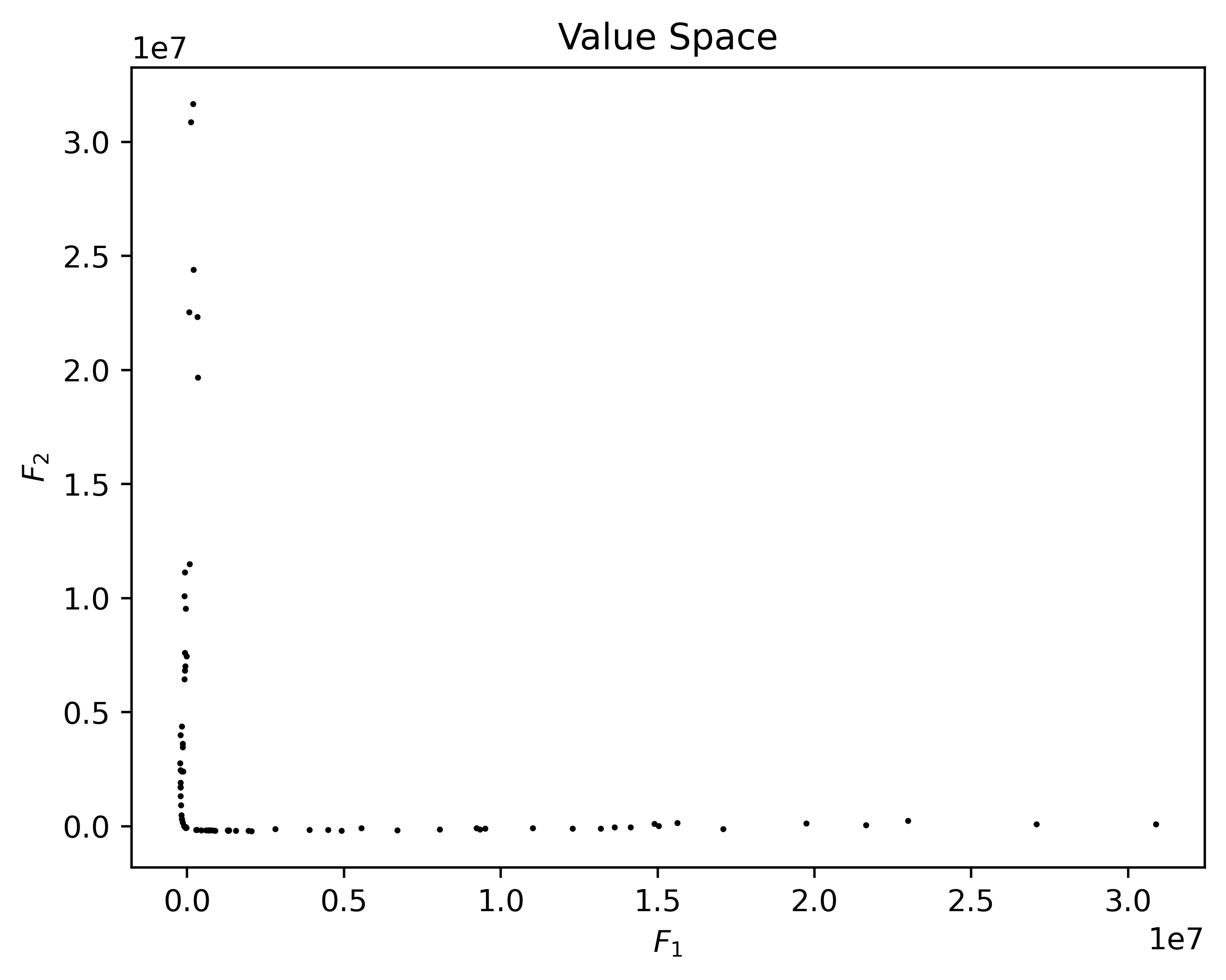} \\
			\includegraphics[scale=0.22]{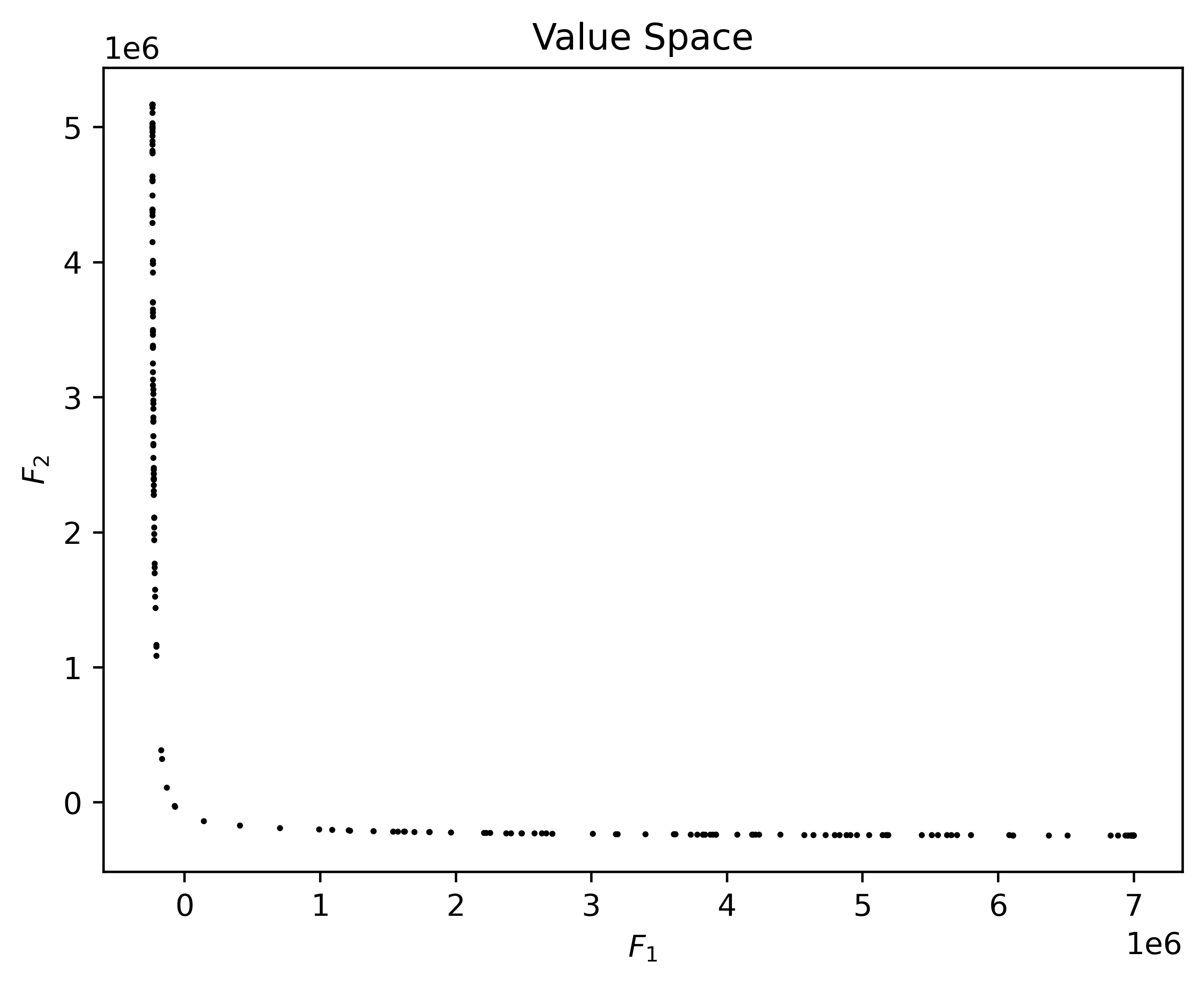}
		\end{minipage}
	}
	\subfigure[QPg]
	{
		\begin{minipage}[H]{.22\linewidth}
			\centering
			\includegraphics[scale=0.22]{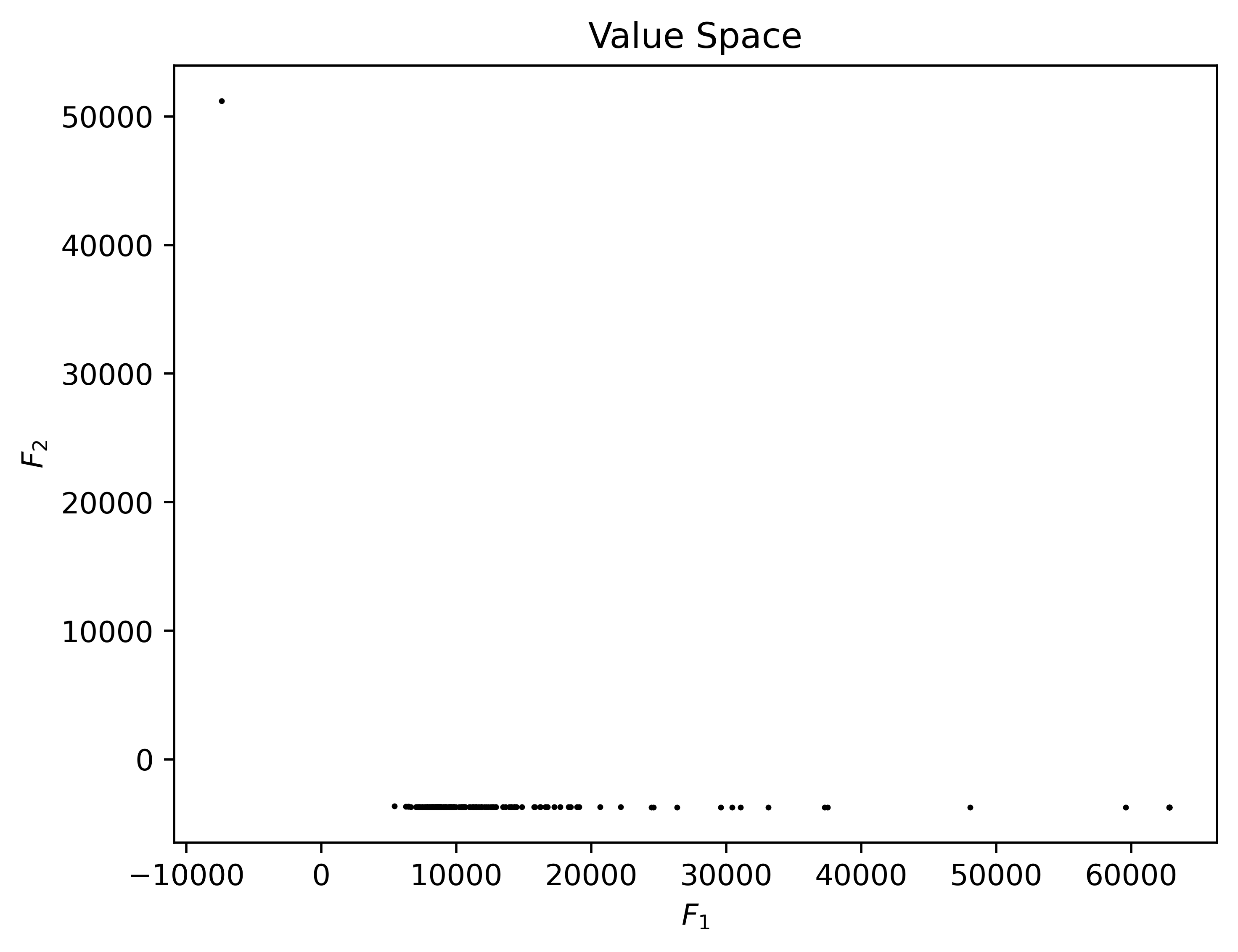} \\
			\includegraphics[scale=0.22]{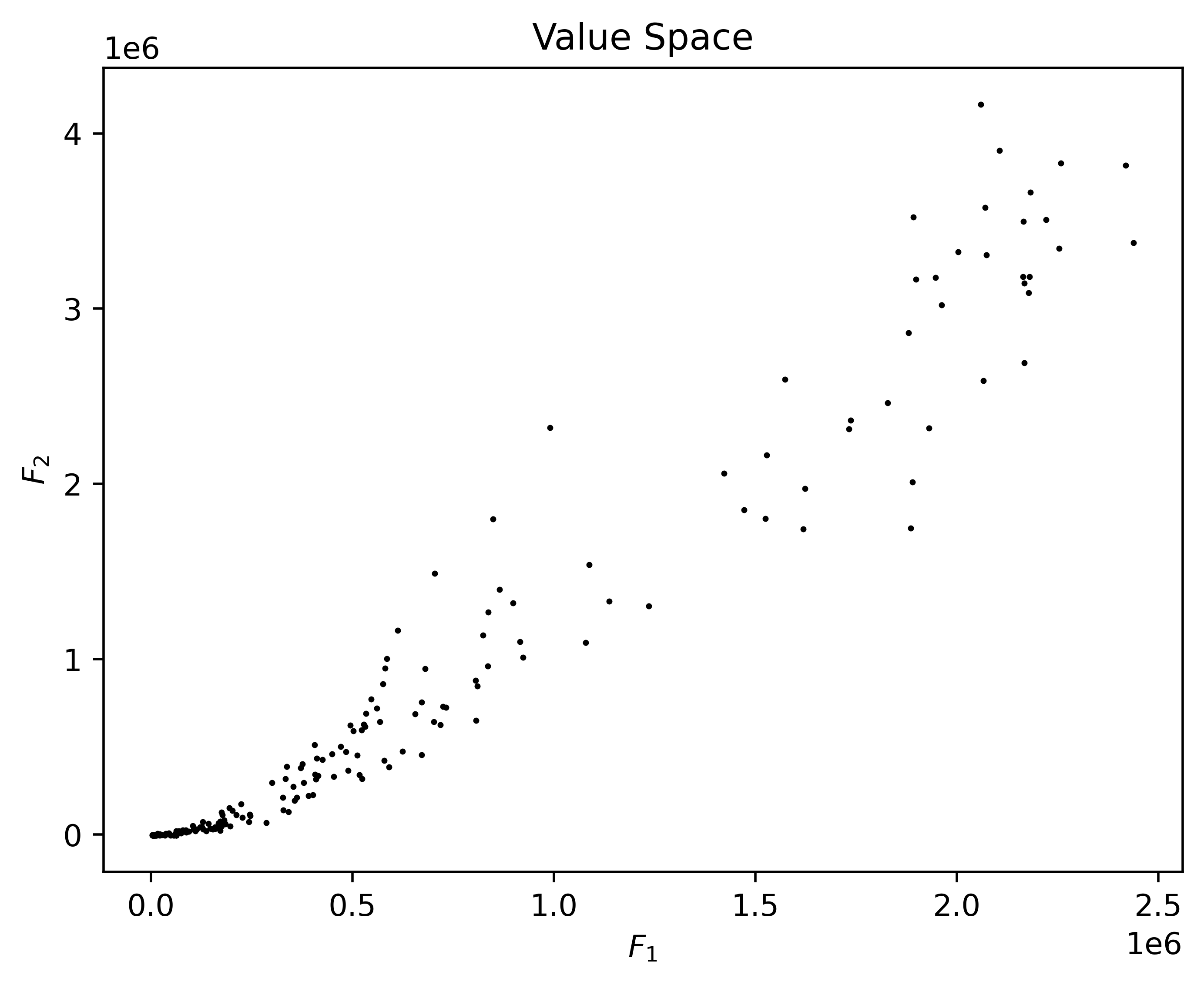} \\
			\includegraphics[scale=0.22]{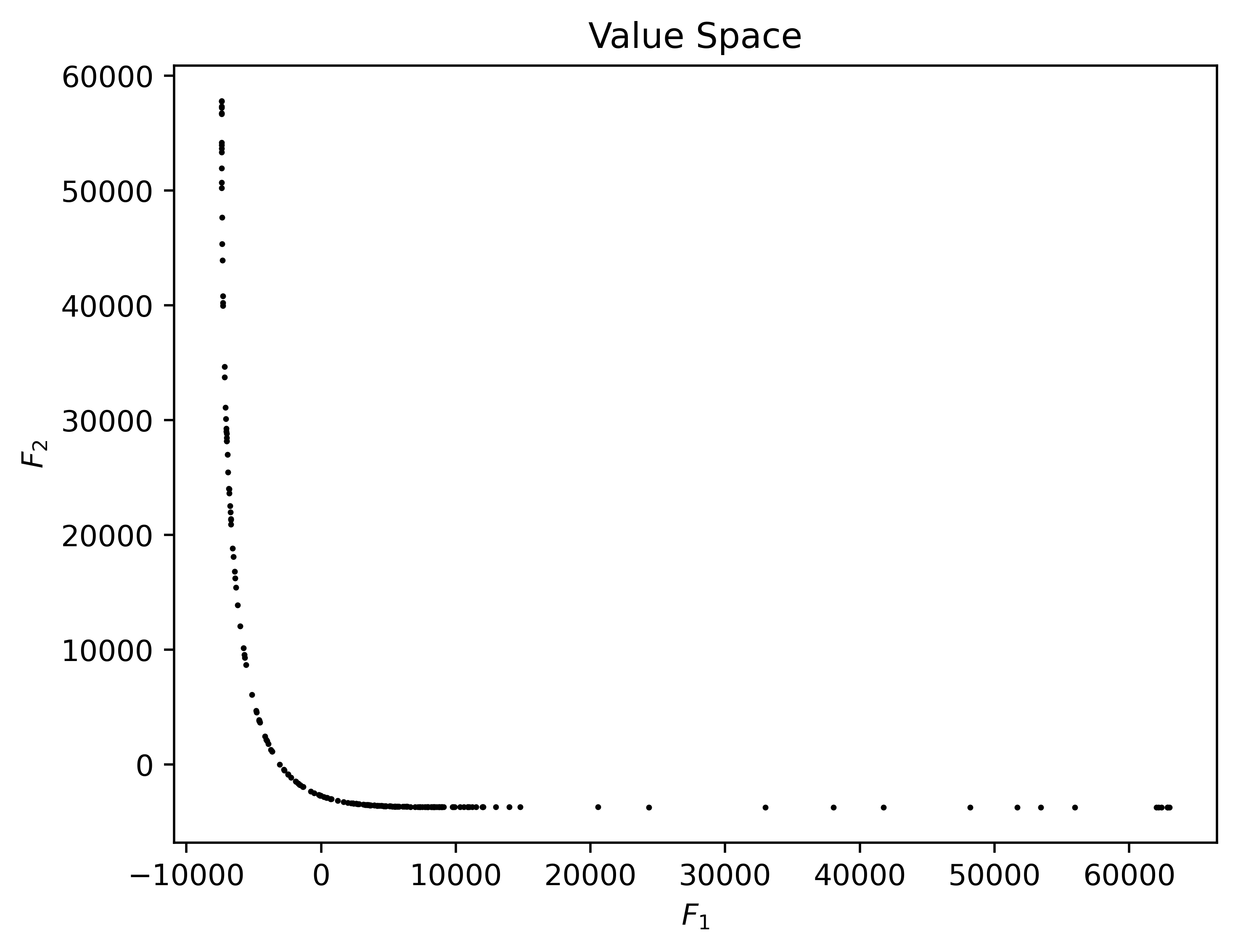}
		\end{minipage}
	}
	\caption{Numerical results in variable space obtained by VMM ({\bf top}), BBDMO ({\bf middle}) and BBDMO\_VM for problems QPd, QPe, QPf, and QPg.}
	\label{f4}
\end{figure}

\begin{table}[h]
	\centering
	\resizebox{.95\columnwidth}{!}{
	\begin{tabular}{lrrrrrrrrrrrrrrr}
		\hline
		Problem &
		\multicolumn{3}{l}{QNMO} &
		&
		\multicolumn{3}{l}{VMMO} &
		\multicolumn{1}{l}{} &
		\multicolumn{3}{l}{BBDMO} &
		\multicolumn{1}{l}{} &
		\multicolumn{3}{l}{BBDMO\_VM} \\ \cline{2-4} \cline{6-8} \cline{10-16} 
		&
		\multicolumn{1}{r}{iter} &
		\multicolumn{1}{r}{feval} &
		\multicolumn{1}{r}{time} &
		\textbf{} &
		\multicolumn{1}{r}{iter} &
		\multicolumn{1}{r}{feval} &
		\multicolumn{1}{r}{time} &
		\multicolumn{1}{r}{} &
		\multicolumn{1}{r}{iter} &
		\multicolumn{1}{r}{feval} &
		\multicolumn{1}{r}{time} &
		\multicolumn{1}{r}{} &
		\multicolumn{1}{r}{iter} &
		\multicolumn{1}{r}{feval} &
		\multicolumn{1}{r}{time} \\ \hline
		QPa & 16.44 & 32.44 & 108.34 &  & 17.29 & 35.89 & \textbf{3.60} &  & 16.97 & 21.66 & 3.80 &  & \textbf{12.80} & \textbf{13.77} & 4.59 \\
		QPb & \textbf{19.27} & {38.52} & 142.33 &  & {27.76} & 82.16 & \textbf{8.03} &  & 61.37 & 102.41 & 13.29 &  & 30.79 & \textbf{33.57} & 11.10 \\
		QPc & 91.95 & 423.02 & 28735.00 &  & 115.24 & 674.43 & 87.14 &  & 75.14 & 124.56 & 16.12 &  & \textbf{47.38} & \textbf{48.56} & \textbf{18.53} \\
		QPd & -- & -- & -- &  & 142.18 & 810.84 & 107.59 &  & 266.58 & 536.73 & 57.64 &  & \textbf{61.20} & \textbf{67.02} & \textbf{23.99} \\
		QPe & -- & -- & -- &  & 459.26 & 4150.27 & 6073.18 &  & 253.19 & 490.93 & \textbf{169.21} &  & \textbf{89.27} & \textbf{90.65} & 793.86 \\
		QPf & -- & -- & -- &  & 473.83 & 4206.95 & 6095.90 &  & 498.66 & 1314.98 & \textbf{467.40} &  & \textbf{166.59} & \textbf{178.25} & {1525.37} \\
		QPg & -- & -- & -- &  & 228.98 & 1471.43 & 185.22 &  & 467.33 & 2940.88 & 168.07 &  & \textbf{217.34} & \textbf{227.83} & \textbf{87.39} \\ \hline
	\end{tabular}
}
\caption{Number of average iterations (iter), number of average function evaluations (feval), and average CPU time (time($ms$)) of QNMO, VMMO, BBDMO, and BBDMO\_VM implemented on quadratic problems.}\label{tab4}
\end{table}
Table \ref{tab4} presents the number of average iterations (iter), number of average function evaluations (feval), and average CPU time (time($ms$)) over 200 experimental runs for every quadratic problem. All the tested methods exhibit convergence for well-conditioned and low-dimensional problems (QPa-b), except that QNMO requires significantly more CPU time due to the expensive per-step cost. The CPU time required by QNMO increases substantially for problem QPc, making it impractical for high-dimensional problems. For ill-conditional and high-dimensional problems (QPd-f), QNMO fails to converge, while BBDMO\_VM significantly outperforms VMMO and BBDMO. It is worth noting that VMMO and BBDMO\_VM are second-order methods which have the potential to capture the local curvature for ill-conditioned problems. However, VMMO can not handle the imbalances among the objectives, resulting in biased solutions (see Fig. \ref{f4}). On the other hand, BBDMO is a first-order method that can cope with ill-conditioned problems due to Barzilai-Borwein's rule but fails to converge for extremely large-scale and high-dimensional problems (QPf). In summary, the primary experiment results confirm that for ill-conditional and high-dimensional MOPs, the proposed BBDMO\_VM can better balance the curvature exploration and per-step cost than QNMO, VMMO and BBDMO.
\section{Conclusions}\label{sec8}
In this paper, we proposed a Barzilai-Borwein descent method with variable trade-off metrics that enjoys cheap per-step cost and is not sensitive to imbalances and conditioning. Theoretical analysis indicates that this method can effectively mitigate imbalances among objectives and achieve rapid linear convergence with appropriate metric selection. Our numerical results validate the superiority of the proposed method, incorporating the trade-off of quasi-Newton approximation. It significantly outperforms QNMO, VMMO, and BBDMO, particularly in the case of large-scale and ill-conditioned problems.
\par From a methodological perspective, it may be worth considering the following points:
\begin{itemize}
	\item Given that the variable metric in BBDMO\_VM is determined by the trade-off (approximation of) Hessian matrices, a natural question arises: Does BBDMO\_VM exhibit superlinear convergence? From a computational perspective, the limited memory BFGS formulation is also recommended for high-dimensional problems.
	\item In this work, we have focused on metric selection from the perspective of approximating Newton's method, which has shown promising performance in quadratic problems. Besides, exploring alternative preconditioning methods could be an intriguing avenue for future research.
	\item Chen et al. recently \cite{CTY2023c} established superlinear convergence of the Newton-type proximal method for MOPs. Consequently, it is meaningful to extend BBDMO\_VM for solving ill-conditioned multiobjective composite problems. Given the potential for expensive proximal operators with non-diagonal matrices, exploring approaches that capture the local geometry using diagonal matrices, such as diagonal Barzilai-Borwein stepsize \cite{PDB2020} is practical. 
\end{itemize}

\bibliographystyle{abbrv}
\bibliography{references}

\begin{acknowledgements}
This work was funded by the Major Program of the National Natural Science Foundation of China [grant numbers 11991020, 11991024]; the National Natural Science Foundation of China [grant numbers 11971084, 12171060]; NSFC-RGC (Hong Kong) Joint Research Program [grant number 12261160365]; the Team Project of Innovation Leading Talent in Chongqing [grant number CQYC20210309536]; the Natural Science Foundation of Chongqing [grant number ncamc2022-msxm01]; and Foundation of Chongqing Normal University [grant numbers 22XLB005, 22XLB006].
\end{acknowledgements}

\end{document}